\documentclass[12pt, reqno]{amsart}

\usepackage{tikz, pgf-umlsd} 
\usetikzlibrary{shapes,arrows}

\usepackage[margin=0.9in]{geometry}

\usepackage[pagebackref]{hyperref}
\usepackage[english]{babel}
\usepackage{amsrefs}
\usepackage{amsmath,amssymb,amsfonts,amsthm,enumerate}
\usepackage{url}
\usepackage{graphicx,epstopdf,color}

\numberwithin{equation}{section}

\newtheorem{theorem}{Theorem}[section]
\newtheorem{lemma}[theorem]{Lemma}
\newtheorem{definition}[theorem]{Definition}
\newtheorem{remark}[theorem]{Remark}
\newtheorem{proposition}[theorem]{Proposition}
\newtheorem{corollary}[theorem]{Corollary}
\newtheorem{openpb}[theorem]{Open problem}

\newcommand{\R}{\mathbb{R}}

\newcommand{\N}{\mathbb{N}}

\renewcommand{\leq}{\leqslant}
\renewcommand{\le}{\leqslant}
\renewcommand{\geq}{\geqslant}
\renewcommand{\ge}{\geqslant}

\renewcommand{\epsilon}{\varepsilon}

\newcommand{\e}{\varepsilon}

\newcommand{\CUNO}{C_3}
\newcommand{\CDUE}{C_2}
\newcommand{\CDUEDUE}{C_1}

\newcommand{\CALERRE}{E}

\usepackage{mathtools}
\mathtoolsset{showonlyrefs}

\author[X. Cabr\'e]{Xavier Cabr\'e}
\address{Xavier Cabr\'e \textsuperscript{1,2}
\newline
\textsuperscript{1} ICREA, Pg. Lluis Companys 23, 08010 Barcelona, Spain
\newline
\textsuperscript{2} Universitat Polit\`ecnica de Catalunya,  IMTech and Departament de Matem\`{a}tiques, Diagonal 647, 08028 Barcelona, Spain}
\email{xavier.cabre@upc.edu}

\author[S. Dipierro]{Serena Dipierro}
\address{Serena Dipierro \textsuperscript{3}
\newline
\textsuperscript{3}
University of Western Australia,
Department of Mathematics and Statistics,
35 Stirling Highway,
Crawley, Perth,
WA 6009, Australia}
\email{serena.dipierro@uwa.edu.au}

\author[E. Valdinoci]{Enrico Valdinoci}
\address{Enrico Valdinoci \textsuperscript{3}
\newline
\textsuperscript{3}
University of Western Australia,
Department of Mathematics and Statistics,
35 Stirling Highway,
Crawley, Perth WA 6009, Australia}
\email{enrico.valdinoci@uwa.edu.au}

\title[The Bernstein technique for integro-differential equations]
{The Bernstein technique \\ for integro-differential equations}

\thanks{The first author is supported by grants MTM2017-84214-C2-1-P and RED2018-102650-T funded by MCIN/AEI/10.13039/501100011033 and by ``ERDF A way of making Europe''. He is member of the research group 2017SGR1392 and of the Barcelona Graduate School of Mathematics. 
The second author is supported by
the Australian Research Council DECRA DE180100957
``PDEs, free boundaries and applications''. 
The third author is supported by
the Australian Laureate Fellowship
FL190100081
``Minimal surfaces, free boundaries and partial differential equations''.
The second and third authors are members of INdAM and AustMS}

\begin{document}

\begin{abstract}
We extend the classical Bernstein technique to the setting of integro-differential operators. As a consequence, we provide first and one-sided second derivative estimates for solutions to fractional equations, including some convex fully nonlinear equations of order smaller than two ---for which we prove uniform estimates as their order approaches two. Our method is robust enough to be
applied to some Pucci-type extremal equations and to obstacle problems for fractional operators, although several of the results are new even in the linear case.
We also raise some intriguing open questions, one of them concerning the ``pure'' linear fractional Laplacian, another one
being the validity of one-sided second derivative estimates for Pucci-type convex equations associated to linear operators with general kernels.
\end{abstract}


\maketitle

\section{Introduction}
 
The Bernstein technique is a powerful tool to establish derivative estimates, through the use of auxiliary functions and the maximum principle, for solutions of elliptic equations. The goal of this paper is to extend this method to the setting of fractional equations. Up to our knowledge, this is done in the current work for the first time in a systematic way,
even for the fractional Laplacian.
The technique will allow us to establish first and one-sided second derivative estimates for a large class of integro-differential equations, including some fully nonlinear equations of order smaller than two. 

Fractional diffusions arise in classical models, notably in the description of water waves, of atom dislocations in crystals, and in the displacement of
an elastic membrane on a thin obstacle. These problems can be efficiently attacked
by transforming them into a fractional setting on a lower dimensional (or boundary) object.
More recent models include energy transfer in nanotubes,
plasma physics, price oscillations in stock markets, and biological dispersals
in sparse environments ---see e.g.~\cite{MR3469920} and references
therein.

The Bernstein technique ---as introduced in the local case by Bernstein
himself~\cite{MR1511375, MR1511579}--- 
relies on considering some auxiliary functions 
which involve the solution, its derivatives, and suitable cutoff functions. In view of certain equations (inequalities, rather) satisfied by the auxiliary functions and thanks to the maximum principle, they allow to estimate higher derivatives of the solution in terms of lower order ones,
by paying a price in the size of the reference domain. 

Let us recall this procedure with the simplest example in local equations, the Laplace operator. This will serve as a preparation for the nonlocal framework.
Given two functions~$u$ and~$\eta$ ($u$ must be thought as the solution of an equation, while $\eta$ will be a cutoff), both smooth enough, consider the auxiliary function
\begin{equation}\label{FACIL0}
\varphi:= \eta^2(\partial_e u)^2+\sigma u^2,
\end{equation} 
where $e\in\R^n$ is a unit vector, $|e|=1$, and $\sigma>0$ is a constant. We have that
\begin{equation}\label{PREFA}\begin{split}
-\Delta\varphi \;&=\; 
-2|\nabla \eta|^2 (\partial_e u)^2-
2\eta \Delta\eta (\partial_e u)^2-
8\eta \nabla\eta \cdot \nabla (\partial_e u)\partial_e u\\
&\qquad-
2\eta^2|\nabla\partial_e u|^2-
2\eta^2\partial_e u\, \partial_e \Delta u
-2 \sigma |\nabla u|^2-
2 \sigma u \Delta u.\end{split}\end{equation}
Since, by the Cauchy-Schwarz inequality
$$ \big|8\eta \nabla\eta \cdot \nabla(\partial_e u)\partial_e u\big|\le
2\eta^2\,|\nabla\partial_e u|^2
+8|\nabla\eta|^2 (\partial_e u)^2,
$$
equation~\eqref{PREFA} yields
\begin{equation*}\begin{split}
-\Delta\varphi \;&\le\;
6|\nabla \eta|^2 (\partial_e u)^2-
2\eta \Delta\eta (\partial_e u)^2-2 \sigma |\nabla u|^2
+2\eta^2\partial_e u\, (-\Delta )\partial_eu
+2 \sigma u (-\Delta u).
\end{split}\end{equation*}
In particular, by choosing~$\sigma\ge C_n \|\eta\|^2_{C^2(\R^n)}$ for an appropriate constant $C_n$ depending on $n$ (more specifically, on the precise way in which the $C^2$-norm of a function in $\R^{n}$ is defined),
we obtain that
\begin{equation}\label{FACIL}\begin{split}
-\Delta\left( \eta^2(\partial_e u)^2+\sigma u^2\right) \;\le\;
2\eta^2\partial_e u\, (-\Delta )\partial_eu
+2 \sigma u \,(-\Delta u) \qquad \text{ if } \sigma\ge C_n \|\eta\|^2_{C^2(\R^n)}.
\end{split}\end{equation}
This is a ``clean'', key inequality satisfied by any function $u$ (not necessarily a solution of an equation). 

Now, if we assume the function~$u$
to be harmonic, say~$-\Delta u=0$ in~$B_1\subset\R^n$, we deduce from~\eqref{FACIL}
that~$-\Delta\varphi\le0$ in~$B_1$. If, in addition, we take the function~$\eta\in C^\infty_c(B_1)$ to have compact support in $B_1$ and to satisfy~$\eta=1$ in~$B_{1/2}$,
the maximum principle for subharmonic functions ensures that~$\varphi$
attains its maximum along~$\partial B_1$. As a consequence,
$$ 
\sup_{B_{1/2}} (\partial_e u)^2\le
\sup_{B_{1/2}}\varphi\le
\sup_{B_1}\varphi=\sup_{\partial B_1}\varphi= \sigma\sup_{\partial B_1}u^2\le \sigma \|u\|_{L^\infty(B_1)}^2,$$
thus yielding an explicit interior gradient estimate for the solution~$u$. As we will see, simple variations
of this method, in which higher derivatives are taken into account
within the auxiliary function, lead to higher order estimates as well.

In spite of its rather elementary flavor, the Bernstein method is a powerful nonvariational tool that finds applications in several contexts and
for a large number of equations. The quadratic auxiliary function above (which is the one that we will consider within the nonlocal setting) finds applications even for second order fully nonlinear uniformly elliptic equations; see the monograph \cite[Chapter 9]{CC}. 
More sophisticated auxiliary functions (with other nonlinear dependences on $u$ and $\partial_{e}u$) lead to gradient estimates for the prescribed mean curvature equation; see e.g. \cite{MR1617971}
and references therein.

Instead, to the best of our knowledge, the Bernstein method  for the quadratic auxiliary function~\eqref{FACIL0} has not been yet studied in relation with
fractional and integro-differential equations, not even for the fractional Laplacian. In this respect, the closest work that we could find is~\cite{MR2405856}, by Biswas, Jakobsen, and Karlsen, which concerns an integro-differential equation  of parabolic type posed in the whole space. Here the Bernstein technique is applied to a quadratic auxiliary function depending
on incremental quotients (but not containing the cut-off function $\eta$) to obtain suitable
Lipschitz bounds, which are then exploited to deal with the convergence
of the discrete scheme under consideration.

In accordance with our proofs and results, we must merge
the operators that we treat into two
categories. The first one consists of
equations built from {\em operators that admit a local extension in one more
dimension}, as it is the case of the fractional Laplacian.
Our second category of equations consists of linear {\em integro-differential 
operators with general kernels}, as well as fully nonlinear operators 
built from them.

In the next subsections, we describe in detail the
framework and results of our work.

\subsection{Pucci-type equations in the presence of extensions}\label{sub:prima}
We start dealing with the case of Pucci-type equations associated to affine transformations of the fractional Laplacian with elliptic matrices. 
To built them, given constants $0<\lambda\leq \Lambda$ we let
\begin{equation}\label{ELLMA}
{\mathcal{A}}={\mathcal{A}}_{\lambda,\Lambda}\text{ be the set of $n$-dimensional symmetric
matrices with eigenvalues in }[\lambda,\Lambda].
\end{equation}
Now, for a given~$s\in(0,1)$, we define the operator
\begin{equation} \label{TUTTIS}\begin{split}
{\mathcal{L}}_A u(x)\,&:= c_{n,s} \, {\rm P.V.}\,
\int_{\R^n}\frac{u(x)-u(y)}{|A(x-y)|^{n+2s}}\, dy\\&
:=c_{n,s} \, \lim_{\e\searrow0}\,
\int_{\R^n\setminus B_\e(x)}\frac{u(x)-u(y)}{|A(x-y)|^{n+2s}}\, dy,
\end{split}\end{equation}
where $c_{n,s}>0$ is a suitable normalizing
constant which makes that, when $A={\rm Id}$ is the identity, ${\mathcal{L}}_{\rm Id}$ becomes a fraction of the classical Laplacian, that is,
$$
{\mathcal{L}}_{\rm Id} = (-\Delta)^{s}.
$$
The above limit is well defined, and finite, whenever $u$ is a $W^{2,\infty}=C^{1,1}$ function (locally) which is bounded in all of $\R^n$.
To ease the notation, the principal value ${\rm P.V.}$ will be omitted from now on.

We can now consider the maximal operator equation
\begin{equation}\label{MUCCI:EQ}
{\mathcal{M}}_{{\mathcal{A}}} u(x):=
\sup_{A\in {\mathcal{A}}}\Big( {\mathcal{L}}_A u(x)-g_A(x)\Big)= 0\quad {\mbox{for all }}x\in B_1,
\end{equation}
where $g_A$ 
is a given continuous function in the unit ball~$B_1\subset\R^n$,
for every~$A\in {{\mathcal{A}}}$.
We will assume continuity of~$g_A$
with respect to the parameters~$A\in{{\mathcal{A}}}$:
\begin{equation}\label{Gcon}
{\mbox{if~$A_k\to A$ as~$k\to+\infty$, then }}\lim_{k\to+\infty}
g_{A_k}(x)= g_A(x) {\mbox{ for all }}x\in B_1.
\end{equation}
Some existence and regularity results for \eqref{MUCCI:EQ} will be described in Appendix~\ref{sec:regularity-new}.

By developing a Bernstein technique in this framework, we 
establish first and one-sided second derivative bounds
for solutions of~\eqref{MUCCI:EQ}. 
Our estimates are uniform as the order of the operators converges to two, providing a unified theory up to the local case, with uniform constants in the bounds. In this respect,
note that the operators~${\mathcal{L}}_A$ 
recover, in the limit~$s\nearrow1$, every 
second order linear elliptic operator in nondivergence with constant coefficients 
(see~Section~6 in~\cite{caff-silv} and Remark~\ref{affine-indefinite} below).

The following is our first result. Here we need the smooth function $u$ to belong to $W^{2,\infty}(\R^n)$, since, within the proofs, the fractional operators will act on
derivatives of $u$ up to order two; in this way, second derivatives will be smooth functions bounded in all space.
 
\begin{theorem}\label{MUCCI}
Given $s\in (0,1)$, $0<\lambda\leq\Lambda$, and
functions $g_{A}\in W^{1,\infty}( B_1)$ for $A\in {\mathcal{A}}$, assume~\eqref{Gcon} and let~$u \in C^\infty (\R^n)\cap W^{2,\infty} (\R^n)$ be a solution of~\eqref{MUCCI:EQ}.

Then, 
\begin{equation}\label{9ixsjXVaBAM-1}
\sup_{B_{1/2}} |\nabla u|
\le {C}\,\Big( \| u\|_{L^\infty(\R^n)}+
\sup_{A\in{{\mathcal{A}}}}\|g_A\|_{W^{1,\infty}(B_1)} 
\Big)
\end{equation}
for some constant~$C$ depending only
on~$n$, $\lambda$, and~$\Lambda$.

If, in addition,  $g_A \in W^{2,\infty}(B_1)$
for all~$A\in{\mathcal{A}}$, then we have
\begin{equation}\label{9ixsjXVaBAM-2}\begin{split}&
\sup_{B_{1/2}} \partial^2_e u\le C\,
\Big(
\| u\|_{L^\infty(\R^n)}+\sup_{A\in{\mathcal{A}}}\|g_A\|_{W^{2,\infty}(B_1)}
\Big)
\end{split}\end{equation}
for every~$e\in\R^n$ with~$|e|=1$,
where~$C$ is as before.\footnote{As it will be apparent from the proof, estimates~\eqref{9ixsjXVaBAM-1}
and~\eqref{9ixsjXVaBAM-2} can be stated in a more precise way in relation with their dependence on the functions $g_A$. This is explained in Remark~\ref{Rk_on_postive_parts}.}
\end{theorem}

Our one-sided second derivative bound (also called semiconcavity bound)\footnote{As customary, we say that a function~$u$ is semiconcave
if there exists~$C\ge0$ such that the function~$u(x)-C|x|^2$ is concave.} is new and somehow surprising. Indeed, since the order of the operator is smaller than two, one should not expect a regularity theory up to the second order.\footnote{The best regularity theory available for this equation arrives at~$C^{1+\varepsilon}\cup C^{2s+\varepsilon}$,
with~$\max\{1+\varepsilon,2s+\varepsilon\}<2$; see \cite[Theorem 13.1] {caff-silv} and \cite[Theorem 1.1] {caff-silv-2}, respectively, and Appendix~\ref{sec:regularity-new} below.}
In this respect, some other one-sided second derivative estimates have been previously proved for fractional problems. This has been achieved, in the context of  the thin obstacle problem, by Athanasopoulos and Caffarelli~\cite{AC04}. In Corollary~\ref{92OBS} we will address their result, which uses the Bernstein technique but with a different, less flexible, auxiliary function than in the local theory or in the current work. Their auxiliary function is linear in the second derivatives, while ours is quadratic and thus, as we will see, it has already allowed for applications to more general situations in thin obstacle problems. Another semiconcavity estimate is that of Mou \cite{Mou},
which applies to some integro-differential equations under certain (not so simple) hypotheses. The proof in \cite{Mou} does not rely on the Bernstein technique, but on the Ishii-Lions method.

First derivative bounds have already been proved for large classes of fully nonlinear integro-differential equations, using different methods than ours. Among other papers, we point up the works by Jakobsen and Karlsen \cite{JakKar}, by Caffarelli and Silvestre~\cite{caff-silv}, and by Barles, Chasseigne, Ciomaga, and Imbert~\cite{MR2911421}.
The seminal work \cite{caff-silv} establishes a $C^{1+\alpha}$ bound for a large class of fully nonlinear integro-differential equations that includes Isaacs-type equations made from uniformly elliptic linear operators with general kernels in the class ${\mathcal{L}}_{1}$. Their proof relies on ABP-type and Harnack inequalities, and thus it is an extension of the Krylov-Safonov local theory. Instead, the work~\cite{MR2911421} (as \cite{JakKar} did before)
relies on the Ishii-Lions method (where an auxiliary function with doubled variables is used) and leads to a Lipschitz bound. It requires H\"older continuous coefficients but allows for weaker ellipticity assumptions. Thus, we are presenting here a third approach that applies to some new equations but not to all of 
the equations in the papers mentioned above, since we only cover convex equations.

We point out that the results in both \cite{caff-silv} and \cite{MR2911421} apply to viscosity solutions. Extending our method to the viscosity framework will require some new ideas that we have not found implemented in the literature, even in the local case.\footnote{Recall that in the monograph \cite{CC}, for instance, the Bernstein technique is carried out only for smooth solutions.} 
We also recall that solutions of equation~\eqref{MUCCI:EQ} are  not, in general, smooth;\footnote{In any case, \label{FOOTNOTE4}
notice that the smooth setting in  the gradient estimate of Theorem~\ref{MUCCI} applies to a large number of equations. Indeed, given any smooth function~$u\in W^{2,\infty} (\R^n)$ we may define~$f:= \sup_{A\in {\mathcal{A}}} {\mathcal{L}}_A u$ and $g_A=f$ for all $A\in  {\mathcal{A}}$. Then, $u$ solves equation \eqref{MUCCI:EQ}. In addition, since $u$ is smooth and has bounded derivatives, one can check  that $g_A=f$ is Lipschitz ---which suffices for the validity of the first derivative estimate.}
see Appendix~\ref{sec:regularity-new}.

To prove Theorem \ref{MUCCI}, we first need to extend the Bernstein technique to the simplest linear operator: the fractional Laplacian $(-\Delta)^s$ with $0<s<1$, as defined above. The computations \eqref{FACIL0}-\eqref{FACIL} for the classical Laplacian will easily carry over the extension operator for the fractional Laplacian, leading to the following analogue of \eqref{FACIL}. Note that the result is uniform as $s$ tends to 1. To guarantee that the fractional Laplacian is well defined when acting on a smooth function $u$ and also on the auxiliary functions built from its first derivatives, we will assume that both $u$ and $\nabla u$ are bounded in all of $\R^n$, that is, $u\in W^{1,\infty}(\R^n)$. 

\begin{proposition}\label{Berns-frac}
Let~$s\in (0,1)$, $u \in C^\infty (\R^n)\cap W^{1,\infty} (\R^n)$,
$\eta \in C^\infty (\R^n)\cap W^{2,\infty} (\R^n)$,
$\sigma>0$, and $e\in\R^{n}$ with $|e|=1$.
Then, we have
\begin{equation}\label{VE:FRA-new}\begin{split}
(-\Delta)^s\left( 
\eta^2(\partial_e u)^2+\sigma u^2\right)\;\le\;
2\eta^2\partial_e u\,(- \Delta)^s \partial_eu
+2 \sigma u\, (- \Delta)^s u\qquad \text{ if } \sigma\ge \sigma_0,
\end{split}\end{equation}
everywhere  in all of $\R^n$, for some constant $\sigma_0$ depending only on $n$ and $\Vert\eta\Vert_{C^2(\R^n)}$ ---and, in particular, independent of~$s$.
\end{proposition}
 
The proof of the first derivative estimate in Theorem~\ref{MUCCI}
will follow from inequality~\eqref{VE:FRA-new} by
choosing an appropriate cutoff function~$\eta$, after a change of variables to replace~$(-\Delta)^s$ by the operators ${\mathcal{L}}_A$. 

Our method to establish one-sided second derivative bounds will be similar. For this, in~\eqref{VE:FRA-new} we first need to replace $u$ by $v=\partial_e u$, but since we only expect a one-sided second derivative bound from above, we must consider instead the auxiliary function
\begin{equation}\label{eqn-v-pos}
\eta^2(\partial_{e} v)_{+}^2+\sigma v^2
\end{equation}
involving a positive part,\footnote{In the local case there is no need to consider positive parts; see~\cite[Chapter 9]{CC}. It is enough to apply the maximum principle in a ball intersected with 
the set where $\partial_{e} v=\partial_{e}^{2} u$ is positive, and then check that
the auxiliary function is controlled on the boundary of such set. This approach does not work in the nonlocal framework due to the influence of the exterior datum.} where $v=\partial_{e}u$. 
The analogue of inequality \eqref{VE:FRA-new} for the auxiliary function \eqref{eqn-v-pos} is the content of Proposition~\ref{PROP-66glo}. 

For the operators with general kernels treated in the
next subsection, we will prove a rather delicate extension of Proposition~\ref{Berns-frac}. Instead, an analogue inequality for the auxiliary function~\eqref{eqn-v-pos} is unknown; see Open problem~\ref{OP2BIS}.

A similar result to Theorem~\ref{MUCCI} but dealing with linear and with convex fully nonlinear operators of indefinite order will be presented in Subsection~\ref{SEct:EXTE}. 

\subsection{Pucci-type equations for general integro-differential operators}\label{sec:seconda}

In this paper we also take into account the case of operators with more 
general kernels, more precisely, kernels which are not pure powers, neither 
rotationally invariant. This setting
takes into account anisotropic environments.

Let~$K:\R^n\to (0,+\infty]$ satisfy
\begin{equation}\label{EVEN}
K(z)=K(-z)\qquad{\mbox{ for all }} z\in\R^n\setminus\{0\}
\end{equation}
and, for some~$s\in(0,1)$,
\begin{equation}\label{KLIM}
\frac{\CDUEDUE\,s(1-s)}{|z|^{n+2s}}\le
K(z)\le \frac{\CDUE \,s(1-s)}{|z|^{n+2s}}\qquad{\mbox{ for all }}z\in\R^n\setminus\{0\},\end{equation}
where~$0<\CDUEDUE\leq\CDUE$ are given constants. In our main results we will also assume~$K$ to be~$C^{2}$ in~$\R^n\setminus\{0\}$ and to satisfy
\begin{equation}\label{KC1}
|z|\,| \nabla K(z)|
+ |z|^2\,|D^2 K(z)|
\le \CUNO \, K(z)\qquad{\mbox{ for all }}\;z\in\R^n\setminus\{0\},\end{equation}
for some constant~$\CUNO>0$. This is the class ${\mathcal{L}}_{2}$ of kernels introduced in \cite{caff-silv}.
We consider the linear operator
\begin{equation} \label{OP:L:BIS}
{\mathcal{L}}_K u(x)\,:=\int_{\R^n}\big( u(x)-u(y)\big)\,K(x-y)\,dy
\end{equation}
defined, as before, in the principal value sense. The operator is well defined on $W^{2,\infty}$ functions which are bounded in all of $\R^n$. When $K(z)=c_{n,s} |z|^{-n-2s}$, it is the fractional Laplacian.

The assumptions in~\eqref{EVEN}-\eqref{KC1}
are satisfied by the class of general stable symmetric
operators, where the kernels are
given by
\begin{equation}\label{SETA} K(z) :=\frac1{2 \,|z|^{n+2s}}\left(
a\left(\frac{z}{|z|}\right)+
a\left(-\frac{z}{|z|}\right) \right),\end{equation}
under appropriate hypotheses on the positive function $a$. See, for instance, (1.3) in~\cite{MR3482695}.

Integro-differential operators with such kernels naturally arise in the L\'evy-Khintchine probabilistic formula, to take into
account Poisson processes with jumps; see e.g. Section~2.2 in~\cite{MR854867}.
They possess applications in several fields; see e.g. Sections~1.2 and~1.3
in~\cite{MR854867}.
In spite of many similarities with the case of the fractional Laplacian, they also
present some important differences, in terms of regularity results, with respect to
the fractional Laplacian.
 
We establish first derivative
estimates for fully nonlinear equations involving general fractional kernels. To state our result for maximal type operators (other fully nonlinear equations are treated in next subsection), 
we consider a compact set of indexes~${\mathcal{B}}$, as well as
kernels~$K_B$ and  continuous functions~$g_B$ in $B_1$ for $B\in
{\mathcal{B}}$, satisfying
\begin{equation}\label{Gcon-BIS}{\mbox{if~$B_j\to B$ as~$j\to+\infty$, then }}\lim_{j\to+\infty}g_{B_j}(x)= g_B(x) {\mbox{ for all }}x\in B_1.
\end{equation}

\begin{theorem}\label{PUCCI TYPE BIS}
Let ${\mathcal{B}}$ be a compact set and $\{ K_B\}_{B\in{\mathcal{B}}}$ be kernels satisfying~\eqref{EVEN},
\eqref{KLIM}, and~\eqref{KC1} $($all with the same structural constants~$s$, $\CDUEDUE$,
$\CDUE$, and~$\CUNO)$. For~$B\in{\mathcal{B}}$,
let~$g _B\in W^{1,\infty}( B_1)$ and assume that~\eqref{Gcon-BIS}
is satisfied.

Let~$u\in C^\infty(\R^{n})\cap W^{1,\infty}(\R^n)$ be a
solution of
\begin{equation}\label{1.23BIS} \sup_{B\in{\mathcal{B}}}\Big(
{\mathcal{L}}_{K_B}u(x)-g_B(x)\Big)
= 0\quad {\mbox{for all }}x\in B_1.\end{equation}

Then,
\begin{equation*}
\sup_{B_{1/2}} |\nabla u|\le 
C\,\Big( \| u\|_{L^\infty(\R^n)}+
\sup_{B\in{\mathcal{B}}}\|g_B\|_{W^{1,\infty}( B_1)}
\Big)
\end{equation*}
for some constant~$C$ depending only on~$n$, $s$, $\CDUEDUE$,
$\CDUE$, and~$\CUNO$.
\end{theorem}

Recall that this result applies to a large number of equations, even if it assumes $u$ to be smooth; see the comments in footnote~\ref{FOOTNOTE4}.

To prove Theorem~\ref{PUCCI TYPE BIS}, no extension technique is available. Therefore, the
Bernstein method must rely solely on integral computations made 
``downstairs'', that is, in~$\R^{n}$. This turns out to be a very 
delicate issue. In fact, the validity of the key inequality 
\eqref{VE:FRA-new} for the fractional Laplacian remains unknown in the 
case of the operator ${\mathcal{L}}_{K}$ (see Open problems~\ref{OP1} and \ref{OP2} 
below). Our main contribution is to establish the inequality with an error 
term $E$ which will be absorbable (by scaling properties) at the end of 
the proof of first derivative estimates.
To establish the inequality with an error term, we will use the following criterium.

\begin{proposition} \label{SUPERT-Intro}
Let~$K$ satisfy~\eqref{EVEN}
and~\eqref{KLIM}, and let ${\mathcal{L}}_K$ be defined by~\eqref{OP:L:BIS}. Given a function $u\in C^\infty(\R^{n})\cap W^{1,\infty}(\R^n)$,
$\eta\in C^\infty(\R^n)\cap L^{\infty}(\R^n)$,
$e\in\R^{n}$ with $|e|=1$, and $\sigma>0$, consider
\begin{equation}\label{COM0129pq} \varphi:=\eta^2 (\partial_e u)^2 +\sigma u^2.\end{equation}

Then, given $\CALERRE\in\R$, the inequality
\begin{equation}\label{MAY18-1}
{\mathcal{L}}_K \varphi\le 2\eta^2\,\partial_e u\,{\mathcal{L}}_K \partial_e u+2\sigma u\,{\mathcal{L}}_K u+\CALERRE
\end{equation}
holds at a point~$x\in\R^n$ if and only if
\begin{equation}\label{MAY18-2}
\begin{split}&
2\int_{\R^n}\eta(x)\,\big( \eta(x)-\eta(y)\big)\,\partial_e u(x)\,\partial_e u(y)\,K(x-y)\,dy
\\ &\qquad\le\,\int_{\R^n}\big|\eta(x)\,\partial_e u(x)-\eta(y)\,\partial_e u(y)\big|^2\,K(x-y)\,dy
\\&\qquad\qquad+\sigma\int_{\R^n}\big|u(x)- u(y)\big|^2\,K(x-y)\,dy
+\CALERRE.\end{split}\end{equation}
\end{proposition}

It is simple to check that all integrals in \eqref{MAY18-2} are well defined in the principal value sense; see the comments following Proposition~\ref{SUPERT}. A useful aspect of \eqref{MAY18-2} is that all terms in its right-hand side are nonnegative, which is not necessarily the case in \eqref{MAY18-1}.

Using this result, we will establish the following key inequality for the operator ${\mathcal{L}}_K$.
It differs from the (still unknown) optimal inequality by a ``small error or remainder''.
Its proof will contain several quite surprising weighted integral cancellations.

\begin{theorem}\label{CONTO TRACCIA}
Let~$K$
satisfy~\eqref{EVEN},
\eqref{KLIM}, and~\eqref{KC1}, and let 
${\mathcal{L}}_K$ be defined by~\eqref{OP:L:BIS}. Let~$u\in C^\infty(\R^n)\cap W^{1,\infty}(\R^n)$ and $\eta\in C^\infty(\R^n)\cap W^{2,\infty}(\R^n)$. 

Then, for every~$\e>0$ there exists a constant~$\sigma_\e>0$
depending only on $\e$, $\|\eta\|_{C^2(\R^n)}$, and
the structural constants~$n$, $s$, $\CDUEDUE$,
$\CDUE$, and~$\CUNO$ in~\eqref{KLIM} and~\eqref{KC1}, such that
\begin{equation}\label{T6}
{\mathcal{L}}_K \big( \eta^2(\partial_e u)^2+\sigma_\e u^2\big)\le
2\eta^2\partial_e u \, {\mathcal{L}}_K\partial_e u+
2\sigma_\e \, u\, {\mathcal{L}}_Ku+\e^2\,\|\partial_e u\|^2_{L^\infty(B_3)}
\quad{\mbox{everywhere in }}B_2,
\end{equation}
for every~$e\in\R^n$ with~$|e|=1$.
\end{theorem}

Note that the error depends on the $L^{\infty}$-norm of the first derivative in a larger ball than the ball where the inequality is claimed.

Theorem~\ref{CONTO TRACCIA}
will serve as the cornerstone to prove
the first derivative estimates of Theorem~\ref{PUCCI TYPE BIS}.

The following are some intriguing open problems on the inequalities satisfied by the auxiliary functions in the
Bernstein technique. They concern either the fractional Laplacian or operators with general kernels.

\begin{openpb}\label{OP1}{\rm
Our proof of~\eqref{VE:FRA-new} 
heavily relies on the extension method.
It would be very interesting to prove inequality~\eqref{VE:FRA-new}
without using the extension. We only know how to do this when the function $u$ is assumed to be $s$-harmonic; see Lemma~\ref{849027674843hf38} and its proof.\footnote{Note that our proof of Lemma~\ref{849027674843hf38} (which does not use the extension) is uniform as $s$ tends to 1. This is also the case for the extension proof of \eqref{VE:FRA-new}. Instead, Theorem~\ref{CONTO TRACCIA} (and as a consequence, Theorem~\ref{PUCCI TYPE BIS} below) are not uniform. In this respect, it would be very interesting to find a proof of Theorem~\ref{PUCCI TYPE BIS} which is uniform as $s\nearrow 1$.}
In view that we know \eqref{MAY18-2} to be true with $E=0$ when $K(z)=|z|^{-n-2s}$, we still find intriguing not to be able to prove it directly in $\R^{n}$ (for this kernel and with $E=0$) without using the extension. 
Finding such a proof could shed light into the following question.}
\end{openpb}

\begin{openpb}\label{OP2}{\rm
Given a kernel $K$ as above, does Theorem~\ref{CONTO TRACCIA}
hold true without the additional small remainder? That is,
does~\eqref{T6} hold true with~$\e=0$ and $\sigma_\e$ large enough?
}
\end{openpb}

\begin{openpb}\label{OP2BIS}{\rm
We do not know whether one-sided
second derivative bounds hold true for general kernels, in the setting of Theorem~\ref{PUCCI TYPE BIS}. 
Recall that they do hold, by Theorem~\ref{MUCCI}, for the Bellman operator built from affine transformations of the fractional Laplacian.

To establish such a result, one would need to prove an inequality similar to \eqref{T6}, but with~$u$ and~$\partial_{e} u$ replaced by~$v$ and~$(\partial_{e} v)_{+}$, respectively, as explained in \eqref{eqn-v-pos} (here the arbitrary function $v$ plays the role of $\partial_{e} u$). Recall that the positive part comes from the fact that we only expect one-sided estimates for second derivatives. For the fractional Laplacian we know that such inequality holds, without an error term, by Proposition~\ref{PROP-66glo}. A corresponding inequality for general kernels, even with an absorbable error term~$E$,  is unknown.  For possible future use, in Proposition~\ref{SUPERT} we state
the analogue of criterium \eqref{MAY18-2} for the auxiliary function involving the positive part.}
\end{openpb}

\subsection{Other fully nonlinear equations and operators
of indefinite order}\label{SEct:EXTE}

In this subsection we use our methods in the setting of
superposition of fractional operators of different orders but having an extension property.
Equations of indefinite order describe phenomena in which more than a single diffusion regime takes place.

We consider convex fully nonlinear equations of the following form.
Given a positive integer $J$, let~$F\in C(\R^J)$. 
We assume that there exist functions\footnote{Even though the functions~$
\alpha_j$ are not required to be continuous, they are assumed to be defined
everywhere, and not almost everywhere. This is consistent with the general setting of
the article, in which the equations are supposed to be satisfied
everywhere in a given domain ---indeed, our terminology ``for every'' has to be taken literally, and
not in the meaning of ``for almost every''. This will be
important when proving the maximum principle, since it will require to evaluate the equation at a maximum point. This framework coincides, for instance,
with the one of Chapter~3 in the monograph~\cite{MR1814364}.}
$\alpha_1,\dots,\alpha_J$ defined in $\R^{J}$ and constants~$\Theta_0\ge\vartheta_0>0$
such that
\begin{equation}\label{CONV:POSITIVA}
\alpha_j(p)\ge0\qquad{\mbox{for every $p\in
\R^J$ and $j\in\{1,\dots,J\}$,}}\end{equation}
\begin{equation}\label{MON:F}
\Theta_0\ge
\sum_{j=1}^J\alpha_j(p)\ge\vartheta_0\qquad{\mbox{for every $p\in
\R^J$}},
\end{equation}
and
\begin{equation}\label{CONV:F}
F(q)-F(p)\ge\sum_{j=1}^J\alpha_j(p)(q_j-p_j)
\qquad{\mbox{for every~$q$ and $p$ in $\R^J$.}}
\end{equation}

By \eqref{CONV:F}, $F$ is convex. Note that the hypotheses on $F$ represent,
all three together, convexity and a quantification of ellipticity.
They are satisfied, for instance, by Bellman-type equations built from a finite number of linear operators, which will correspond (see \eqref{EQ} below) to $F(p)=\max\{p_{1},\ldots,p_{J}\}$.\footnote{
Indeed, it suffices to define $\alpha_{j}(p)=1$ if $j$ is the smallest index for which~$p_{j}=F(p)$ and $\alpha_{j}(p)=0$ for all other indexes.}
However, our setting here is more general since we also include operators $F$ of class $C^1$.\footnote{ 
Notice that, since here we only want to involve a finite number of linear operators, the associated Bellman-type equations are not, tipically, of class $C^1$.  This is in contrast with the class of Bellman operators built from infinitely many linear operators (as in Theorems~\ref{MUCCI} and \ref{PUCCI TYPE BIS}), which recovers all convex operators (since any convex function can be written as the supremum of linear functions).
}
Indeed, the three assumptions on $F$ are also satisfied if~$F$ is~$C^1$, convex,
nondecreasing
in each of its coordinate variables, and satisfies
$ \Theta_0\geq \Sigma_{j=1}^J\partial_{p_j}F(p)\ge\vartheta_0$ for every $p\in \R^J$ ---here
we take~$\alpha_j(p):=\partial_{p_j}F (p)$.

We deal with the superposition of operators of different orders.
Given a probability measure~$\mu$ on~$[0,1]$, i.e.,
\begin{equation}\label{1.0} \mu\ge0\qquad{\mbox{and}}\qquad\mu([0,1])=\mu(\R)=1,\end{equation}
we define
\begin{equation} \label{OP:L}
{\mathcal{L}}_\mu u(x):=\int_0^1 (-\Delta)^s u(x)\,d\mu(s).\end{equation}
In case of~$\mu$ being a Dirac's delta
at some~$s\in[0,1]$, 
$ {\mathcal{L}}_\mu$~reduces to the fractional Laplacian~$(-\Delta)^s$
(in particular, to the classical Laplacian if $s=1$). For $s=0$, $(-\Delta)^{0}$ denotes the Identity, $(-\Delta)^{0}u=u$.
The interest of including $s=0$ is to allow a unified treatment
of fully nonlinear equations and obstacle problems; see Corollary~\ref{92OBS} below.

The operators ${\mathcal{L}}_\mu$ have been studied in~\cite{MR3485125}, in relation with Allen-Cahn type equations, through local extension methods.

Given a positive integer $J\in\N$, let $F$ be as above, $\mu_1,\dots,\mu_J$ be probability measures on~$[0,1]$,
and~$g_1,\dots,g_J$ be continuous functions in~$B_1$.
We consider solutions~$u:\R^n\to\R$ of
\begin{equation}\label{EQ}
F\big( {\mathcal{L}}_{\mu_1}u(x)- g _1(x), \;\cdots,\; {\mathcal{L}}_{\mu_J}u(x)- g _J(x) \big)=0\qquad{\mbox{for all }}x\in B_1.
\end{equation}

Note that we are dealing with a very general class of equations of indefinite order, which includes the model equation
\begin{equation}\label{FORM83} F\big( (-\Delta)^{s_1} u,
\;\dots,\; (-\Delta)^{s_J} u \big)=f(x),\end{equation}
with~$s_j\in[0,1]$. 

A more general framework consists of making affine changes of variables for each index~$j$. This establishes a connection between the equations of Subsection~\ref{sub:prima} and those of the current setting; 
see Remark~\ref{affine-indefinite} for more details.
In addition, in such generality equation~\eqref{EQ} would include the classical extremal equations built from a finite number of second order linear operators; see \cite{CC}.
Therefore, since we will establish one-sided second derivative bounds, the convexity assumption on $F$ in the following theorem cannot be dropped, in view of the classical counterexamples to $W^{2,\infty}$ regularity by
Nadirashvili and Vl\u adu\c t~\cite{NV}
for nonconvex fully nonlinear equations of second order.\footnote{Two comments are in order here. First, once a one-sided second derivative bound for a second order fully nonlinear uniformly elliptic equation is established, it automatically leads to full second derivative estimates (by using the equation itself; see the
Bernstein technique described in Chapter~9 of \cite{CC},
and in particular inequality~(9.5) combined with Lemma~6.4
in~\cite{CC}). Second, recall that Isaacs equations cover all possible fully
nonlinear elliptic equations of second order (see Remark~1.5 in \cite{CC-paper})
and that our estimates are independent of the number of operators $J$.}

Our result establishes first and one-sided second derivative bounds for solutions of~\eqref{EQ}. The estimates are uniform in the number $J$ of operators. 

\begin{theorem}\label{THM:1}
Given $F$ satisfying \eqref{CONV:POSITIVA}, \eqref{MON:F}, and  \eqref{CONV:F}, $\mu$ satisfying \eqref{1.0}, and functions $g _j\in W^{1,\infty}(B_1)$
for $j=1,\ldots,J$, let~$u \in C^{\infty}(\R^{n})\cap W^{2,\infty}(\R^{n})$ be a
solution of~\eqref{EQ}. 

Then,
\begin{equation}\label{EST-MAIN-1}\begin{split}
\sup_{B_{1/2}} |\nabla u|\le\,& C\,\Big(
\|u\|_{L^\infty(\R^n)}+ |F(0)|+ 
\sup_{j\in\{1,\dots,J\}}\| g_j\|_{W^{1,\infty}(B_1)}
\Big)
\end{split}\end{equation}
for some constant~$C$ depending only
on~$n$, $\vartheta_0$, and $\Theta_0$.

If in addition $g _j\in W^{2,\infty}(  B_1)$
for $j=1,\ldots,J$, then we have
\begin{equation}\label{EST-MAIN-2}\begin{split}
\sup_{B_{1/2}} \partial^2_e u\le\,& C\,\Big(
\|u\|_{L^\infty(\R^n)}+ |F(0)|+ 
\sup_{j\in\{1,\dots,J\}}\| g_j\|_{W^{2,\infty}(B_1)}
\Big)
\end{split}\end{equation}
for every~$e\in\R^n$ with~$|e|=1$, where~$C$ is as before.
\end{theorem}

We remark that the estimates of
Theorem~\ref{THM:1} are new, to the best of our
knowledge, even in the case when $F$ is linear,
even for~${\mathcal{L}}_{\mu_j}=(-\Delta)^{s_j}$,
and even when all the functions $g_{j}$ are taken to be zero.

As in Subsection~\ref{sub:prima}, the one-sided second derivative estimate~\eqref{EST-MAIN-2} is somehow surprising since, for operators which could be of order smaller than two,
second derivative estimates are not expected to hold.

Since the operator in Theorem~\ref{THM:1} is of indefinite order, considering the equation
in a ball~$B_R$ instead of~$B_1$ produces
an unusual dependence on the radius~$R$
of the right-hand side of the corresponding estimates; see Remark~\ref{Rk_on_postive_parts}
for more details.

Theorem~\ref{THM:1} includes, as a particular case, the obstacle problem for the fractional 
Laplacian,
also called ``thin obstacle problem'' or ``Signorini problem''. For this, we take the measures
to be $\mu_{1}=\delta_{s}$ for some $s\in [0,1]$ and $\mu_{2}=\delta_{0}$, and $F$ to be the $\max$ operator. 

\begin{corollary}\label{92OBS}
Let~$s\in[0,1]$, $f$ and~$\phi$ be $W^{1,\infty}( B_1)$ functions,
and~$u\in C^{\infty}(\R^{n})\cap W^{2,\infty}(\R^{n})$ be a 
solution of
\begin{equation}\label{934rtiI234}
\max\Big\{
(-\Delta)^{s}u, \;u- \phi
\Big\}=f \quad {\mbox{everywhere in }} B_1.
\end{equation}

Then,
\begin{equation}\label{KSM-P8SKNDN}
\sup_{B_{1/2}} |\nabla u|\le C\,\Big(\| u\|_{L^\infty(\R^n)}
+\| f\|_{W^{1,\infty}(B_1)}+\| \phi \|_{W^{1,\infty}(B_1)}
\Big)\end{equation}
for some constant~$C$ depending only on~$n$.

If in addition $f$ and $\phi$ belong to $W^{2,\infty}( B_1)$, then we have
\begin{equation}\label{SEMI}\begin{split}&
\sup_{B_{1/2}} \partial^2_e u\le C\, \Big(\| u\|_{L^\infty(\R^n)}
+\|f\|_{W^{2,\infty}( B_1)}+\|\phi\|_{W^{2,\infty}( B_1)}\Big)
\end{split}\end{equation}
for every~$e\in\R^n$ with~$|e|=1$,
where~$C$ depends only on~$n$.
\end{corollary}

The gradient estimate of Corollary~\ref{92OBS} applies to a large number of equations. Namely,  as in footnote~\ref{FOOTNOTE4}, given any smooth function $u$ in $ W^{2,\infty}(\R^{n})$, take $\phi=0$ and define $f:=\max\{
(-\Delta)^{s}u, \;u\}$; note that $f$ is a Lipschitz function.

The bound~\eqref{SEMI} recovers the semiconcavity
estimate for the thin obstacle problem, first proved by Athanasopoulos and Caffarelli~\cite{AC04}
and later extended by Fern\' andez-Real \cite{FR} to the fully nonlinear thin obstacle problem.
In these papers the Bernstein technique was already used, but with a less flexible
auxiliary function than
in the current work: their auxiliary function is linear in
the second derivatives, while ours is quadratic. The quadratic structure has already
allowed further applications in obstacle problems. Indeed, in private communication
with the authors of \cite{FR-J} (an article that cites ours),
Fern\' andez-Real and Jhaveri have used our method in a situation where a
polynomial solving the thin obstacle problem is subtracted to the solution and
have gotten, in this way, estimates independent of the polynomial. This required
the use of the quadratic auxiliary function, as well as the use of incremental
quotients. 

\subsection*{Organization of the paper}\label{ORGA}

The rest of this paper is organized as follows. 

Section~\ref{PARTE2} is devoted to the arguments needed
to treat the operators defined
``downstairs'' in Subsections~\ref{sec:seconda}
and~\ref{SEct:EXTE}. Subsection~\ref{8i8996999} contains
a criterium for the key inequality for auxiliary functions, Proposition~\ref{SUPERT}, which will complement Proposition~\ref{SUPERT-Intro}.
In Subsection~\ref{934itgfoeojtthewi34848488484}
we prove Theorem~\ref{CONTO TRACCIA}, while Subsection~\ref{9o23939398iww3456iw}
contains a proof of the key inequality
of Proposition~\ref{Berns-frac} without using the extension
but with the additional assumption that~$u$
is $s$-harmonic.

In Section~\ref{61}
we present the necessary material on
linearized operators and 
maximum principles needed for the proofs of our main results.

In
Section~\ref{P98765vgyuE1}
we state and prove a general statement
(namely, Theorem~\ref{sec:XT}) which will be pivotal to obtain
the main results of this paper. 

Section~\ref{PARTE1} contains
the proofs of those results presented in 
Subsections~\ref{sub:prima}
and~\ref{SEct:EXTE} which deal with
operators ``with extensions''. More precisely,
in Subsection~\ref{sec:uno} we 
discuss Proposition~\ref{Berns-frac}
and its variants needed for the proof of the main results,
while
Subsection~\ref{MUCCI:S} contains the proofs of
Theorem~\ref{MUCCI}, Theorem~\ref{THM:1},
and Corollary~\ref{92OBS}.

In Section~\ref{934itgfoeojtthewi34848488484BOS} we deal with operators without an extension property.
By suitable scaled estimates,
we will be able to ``reabsorb'' our error or remainder term in Theorem~\ref{CONTO TRACCIA}
and complete the proof of Theorem~\ref{PUCCI TYPE BIS}. 

The three first appendices concern results needed in the paper;
Appendix \ref{AP-maxple} is of special interest since it establishes a maximum principle for the extension problem which is new, up to our knowledge. On the other hand, Appendix~\ref{sec:regularity-new} is of informative nature and discusses
existence and regularity issues for the equations of the paper. 

\section{The key inequalities for general integro-differential operators}\label{PARTE2}

\subsection{Equivalent formulations of
the key inequalities}\label{8i8996999}

Here we 
provide the proof of Proposition~\ref{SUPERT-Intro}, which will be used in next subsection to establish first derivative estimates. 
The proof of Proposition~\ref{SUPERT-Intro}
will also establish the following criterium, a variant
of the proposition which involves the positive part of the derivative. If one could prove  that inequality~\eqref{MAY18-4} appearing below holds for an appropriate error $E$, then one-sided second derivative estimates for operators with general kernels would follow;  see Open problem~\ref{OP2BIS}.

\begin{proposition} \label{SUPERT}
Let~${\mathcal{L}}_K$ be
as in~\eqref{OP:L:BIS}, with~$K$ satisfying~\eqref{EVEN}
and~\eqref{KLIM}.
Given~$v\in C^\infty(\R^n)\cap W^{1,\infty}(\R^n)$,
$\overline\eta \in
C^\infty(\R^n)\cap L^\infty(\R^n)$,
$e\in \R^n$ with~$|e|=1$, $\tau>0$, and~$\CALERRE\in\R$,
consider
\begin{equation}\label{8uj8ij6yhb6ytg8uj} \psi:=\overline\eta^2 \big(
\partial_e v\big)^2_{+} +\tau v^2.\end{equation}

Then, the inequality
\begin{equation}\label{MAY18-3}
{\mathcal{L}}_K \psi\le 2\overline\eta^2\,(\partial_e v)_{+}\,
{\mathcal{L}}_K\big(( \partial_e v)_{+}\big)
+2\tau v\,{\mathcal{L}}_K v+\CALERRE
\end{equation}
holds at a point~$x\in\R^n$ if and only if
\begin{equation}\label{MAY18-4}\begin{split}&
2\int_{\R^n}\overline\eta(x)\,\big( \overline\eta(x)-\overline\eta(y)\big)\,(\partial_e v)_{+}(x)\,(\partial_e v)_{+}(y)\,K(x-y)\,dy
\\ &\qquad\le\,\int_{\R^n}\big|\overline\eta(x)\,(\partial_e v)_{+}(x)-\overline\eta(y)\,(\partial_e v)_{+}(y)\big|^2\,K(x-y)\,dy
\\&\qquad\qquad+\tau\int_{\R^n}\big|v(x)-v(y)\big|^2\,K(x-y)\,dy
+\CALERRE.\end{split}\end{equation}
\end{proposition}

Note that the integrals in \eqref{MAY18-4} are finite since $\overline\eta$ is smooth, locally, and
bounded at infinity. At the same time, as in the Introduction, ${\mathcal{L}}_K \psi$ is well defined everywhere in~$\R^n$ since~$(\partial_e v)_{+}^2$ is a locally $W^{2,\infty}$ function which is bounded in all of~$\R^n$.
However, $(\partial_e v)_{+}$ in the right-hand side
of~\eqref{MAY18-3}
is only a $C^{1,1}$ function from below (locally). Recall that one says that~$\varphi\in C(B_1)$ is ``$C^{1,1}$
from below'' in~$B_1$ if for every~$x_0\in B_1$ there exists~$w\in C^{1,1}(B_1)=W^{2,\infty}(B_1)$ such that~$w\le \varphi$ everywhere in~$B_1$ and~$w(x_0)=\varphi(x_0)$.
This setting is sufficient to define the operator pointwise everywhere by having values
in~$\{-\infty\}\cup\R$. In addition, we make the convention~$
0\cdot(-\infty)=0$ in the expression~$(\partial_e v)_{+}\,
{\mathcal{L}}_K\big(( \partial_e v)_{+}\big)$ in \eqref{MAY18-3}.

Furthermore, we notice that whenever~\eqref{MAY18-3} holds true, then also
\begin{equation}\label{chiamare}
{\mathcal{L}}_K \psi\le 2\overline\eta^2\,(\partial_e v)_{+}\,
{\mathcal{L}}_K  \partial_e v
+2\tau v\,{\mathcal{L}}_K v+\CALERRE,
\end{equation}
since this would follow from the above convention when~$\partial_e v(x)\le0$
and from the fact that
$${\mathcal{L}}_K\big(( \partial_e v)_{+}\big)(x)\le{\mathcal{L}}_K\partial_e v(x)\;\mbox{ when }\;\partial_e v(x)>0.$$
This observation
is relevant since
Lemma~\ref{SUBSOLU:EXTE} will give control of~${\mathcal{L}}_K\partial_e v$ from above.
This is why, within Section~\ref{PARTE1} on operators with an extension, the inequality is stated
as~\eqref{chiamare} (with~$E=0$), and not as~\eqref{MAY18-3}.

We also point out that the integrals in~\eqref{MAY18-2}
and~\eqref{MAY18-4} are all
well defined, due to the regularity of
the functions involved. First, the integral in the left-hand side of~\eqref{MAY18-2}
is well defined in the principal value sense, since, for small~$z$,
\begin{eqnarray*}&&
\big( \eta(x)-\eta(x+z)\big)\,\partial_e u(x+z)\,K(z)=
\big( -\nabla\eta(x)\cdot z+O(|z|^2)\big)\,
\big(\partial_e u(x)+O(|z|)
\big)\,K(z)\\&&\qquad=-\partial_e u(x)\nabla\eta(x)\cdot z\,K(z)
+O(|z|^2)\,K(z),
\end{eqnarray*}
and the term~$-\nabla\eta(x)\cdot z\,K(z)$
provides a null contribution to the principal value of the integral
over~$z\in B_1$,
thanks to the symmetry assumption~\eqref{EVEN}
(the decay assumption~\eqref{KLIM}
will then make the term~$O(|z|^2)\,K(z)$ integrable for~$z\in B_1$).
A similar argument applies to the first integral in~\eqref{MAY18-4} where one may assume~$(\partial_e v)_+(x)>0$.

The second integral in~\eqref{MAY18-2} is instead a classical
Lebesgue integral,
since, for small~$z$,
$$ \big|\eta(x)\,\partial_e u(x)-\eta(x+z)\,\partial_e u(x+z)\big|^2\,K(z)
\le \|\eta\,\partial_e u\|_{W^{1,\infty}(B_1(x))}^2\,|z|^2\,K(z).$$
The same argument applies to
the second integral in~\eqref{MAY18-4}, since~$(\partial_e v)_+$ is locally a $W^{1,\infty}$ function.
A simpler argument gives that also the last integrals in~\eqref{MAY18-2}
and~\eqref{MAY18-4} are well defined.

We now establish both Propositions~\ref{SUPERT-Intro}
and~\ref{SUPERT} in a unified manner.

\begin{proof}[Proof of Propositions~\ref{SUPERT-Intro} and~\ref{SUPERT}]
Here, we take~${\mathcal{G}}(t)$ to be\footnote{We point out that
formally (that is, when all the integrals make sense) the arguments
presented here are valid for all nonlinear functions~${\mathcal{G}}$.}
either~$t$ or~$t_+$. We also adopt the notation~${\mathcal{G}}^2(t):=\big({\mathcal{G}}(t)\big)^2$. We write the proof for~$u$, $\sigma$, and~$\eta$
to address Proposition~\ref{SUPERT-Intro}, with~${\mathcal{G}}(t)=t$ in this case. By replacing
these choices by~$v$, $\tau$, and~$\overline\eta$, respectively,
and with~${\mathcal{G}}(t)=t_+$ now, we will conclude Proposition~\ref{SUPERT}.

Given a kernel~$K$, we have
\begin{eqnarray*}
&& {\mathcal{L}}_K \Big( \eta^2\,{\mathcal{G}}^2\big(\partial_e u\big)+\sigma
\,u^2\Big)(x)
-2\eta^2(x)\, {\mathcal{G}}\big(\partial_e u(x)\big)\,{\mathcal{L}}_K\,{\mathcal{G}}\big(\partial_e u\big)(x)
\\&&\qquad\qquad-2\sigma\,u(x)\,{\mathcal{L}}_K u(x)+\sigma\int_{\R^n}\big|u(x)-u(y)\big|^2\,K(x-y)\,dy\\
&&\qquad= \int_{\R^n} 
\Big( \eta^2(x)\,{\mathcal{G}}^2\big(\partial_e u(x)\big)
-\eta^2(y)\,{\mathcal{G}}^2\big(\partial_e u(y)\big)\Big)\,K(x-y)\,dy\\&&\qquad\qquad
+\sigma\int_{\R^n} \big(u^2(x)-u^2(y)\big)\,K(x-y)\,dy\\
&&\qquad\qquad-2\eta^2(x)\, {\mathcal{G}}\big(\partial_e u(x)\big)\,
\int_{\R^n}\Big( {\mathcal{G}}\big(\partial_e u(x)\big)-{\mathcal{G}}\big(\partial_e u(y)\big)\Big)
\,K(x-y)\,dy\\
&&\qquad\qquad-2\sigma\,u(x)\,\int_{\R^n} \big(u(x)-u(y)\big)\,K(x-y)\,dy
+\sigma\int_{\R^n}\big|u(x)-u(y)\big|^2\,K(x-y)\,dy
\\&&\qquad= \int_{\R^n} 
\Big( \eta^2(x)\,{\mathcal{G}}^2\big(\partial_e u(x)\big)
-\eta^2(y)\,{\mathcal{G}}^2\big(\partial_e u(y)\big)\Big)\,K(x-y)\,dy
\\&&\qquad\qquad-2\eta^2(x)\, {\mathcal{G}}\big(\partial_e u(x)\big)\,
\int_{\R^n}\Big( {\mathcal{G}}\big(\partial_e u(x)\big)-{\mathcal{G}}\big(\partial_e u(y)\big)\Big)
\,K(x-y)\,dy\\
&&\qquad= \int_{\R^n}\Big( 2\eta^2(x)\, {\mathcal{G}}\big(\partial_e u(x)\big)\, {\mathcal{G}}\big(\partial_e u(y)\big)
\\&&\qquad\qquad-\eta^2(x)\, {\mathcal{G}}^2\big(\partial_e u(x)\big)-
\eta^2(y)\, {\mathcal{G}}^2\big(\partial_e u(y)\big)
\Big)\,K(x-y)\,dy
\\ &&\qquad=2\int_{\R^n} \eta(x)\,\big(\eta(x)-\eta(y)\big)\, {\mathcal{G}}\big(\partial_e u(x)\big)\, {\mathcal{G}}\big(\partial_e u(y)\big)
\,K(x-y)\,dy
\\&&\qquad\qquad-
\int_{\R^n}\Big|
\eta(x)\, {\mathcal{G}}\big(\partial_e u(x)\big)-
\eta(y)\, {\mathcal{G}}\big(\partial_e u(y)\big)
\Big|^2\,K(x-y)\,dy.
\end{eqnarray*}
As a consequence, the inequality
$$ {\mathcal{L}}_K \big(\eta^2
{\mathcal{G}}^2 (\partial_e u )
+\sigma u^2\big)\le 2\eta^2\,{\mathcal{G}}\big(\partial_e u\big)\,{\mathcal{L}}_K\,
{\mathcal{G}}\big(\partial_e u\big) +2\sigma u\,{\mathcal{L}}_K u+\CALERRE
$$
is pointwise equivalent to 
\begin{equation*}
\begin{split}&
2\int_{\R^n}\eta(x)\,\big( \eta(x)-\eta(y)\big)\,{\mathcal{G}}\big(
\partial_e u(x)\big)\,{\mathcal{G}}\big(\partial_e u(y)\big)\,K(x-y)\,dy
\\ &\qquad\le\,\int_{\R^n}\Big|\eta(x)\,{\mathcal{G}}\big(\partial_e u(x)\big)
-\eta(y)\,{\mathcal{G}}\big(\partial_e u(y)\big)\Big|^2\,K(x-y)\,dy
\\&\qquad\qquad+\sigma\int_{\R^n}\big|u(x)- u(y)\big|^2\,K(x-y)\,dy
+\CALERRE.\end{split}\end{equation*}

{F}rom this, as explained above, one deduces Propositions~\ref{SUPERT-Intro} and~\ref{SUPERT}.
\end{proof}

\subsection{Proof of the first key inequality with a remainder}\label{934itgfoeojtthewi34848488484}

This subsection contains the proof
of our main inequality for the auxiliary function in the case of general
integro-differential operators.

\begin{proof}[Proof of Theorem~\ref{CONTO TRACCIA}]
By the translation invariance of the problem, we see that, to establish~\eqref{T6}
in~$B_2$,
it suffices to prove that
\begin{equation}\label{T6.3}\begin{split}&
2\int_{\R^n}\eta(0)\,\big( 
\eta(0)-\eta(y)\big)\partial_e u(0)\,\partial_e u(y)\,K(y)\,dy\\ &\qquad\qquad\leq\,
\int_{\R^n} \big|
\eta(0)\partial_e u(0)-\eta(y)\partial_e u(y)\big|^2\,K(y)\,dy\\&\qquad\qquad\qquad+
\sigma_\e\,\int_{\R^n}\big|u(0)-u(y)\big|^2\,K(y)\,dy+\e^2\,\| \partial_e u\|^2_{L^\infty(B_1)}.
\end{split}
\end{equation}
Once this is proved, and using Proposition~\ref{SUPERT-Intro} at~$x=0$
with~$\CALERRE:=\e^2\,\|\partial_e u\|^2_{L^\infty(B_1)}$,
the right-hand side of inequality~\eqref{T6}
in~$B_2$ will become~$\e^2\,\| \partial_e u\|^2_{L^\infty(B_3)}$.

To prove~\eqref{T6.3}, we exploit an appropriate cutoff procedure
on the gradient of~$u$, to suitably remove the singularity
of the integrand near the origin in the left-hand side of~\eqref{T6.3}, (without spoiling the estimates
at infinity). 
Namely, 
we consider an odd function~$\xi\in C^\infty_c((-2,2))$
such that~$\xi(t)=t$ if~$t\in(-1,1)$, and
$|\xi(t)|\le2$ for every~$t\in\R$.
We also consider~$\delta\in\left(0,\frac12\right]$, and set~$\xi_\delta(t):=\delta\xi(t/\delta)$.
We observe that
\begin{equation}\label{QUI90}
\xi_\delta(0)=0,\quad \xi_\delta'(0)=1, \quad{\mbox{ and }}\quad
\|\xi_\delta\|_{C^2(\R)}\le \frac{C}{\delta},
\end{equation}
for some universal constant~$C$.

We also introduce the map~$
\R^n\ni y=(y_1,\dots,y_n)\mapsto {\mathcal{Z}}_\delta(y):=\big( \xi_\delta(y_1),\dots,\xi_\delta(y_n)\big)$.
Since~$\xi_\delta$ is odd, we have that
\begin{equation}\label{QUI2091} \int_{\R^n} \nabla \eta(0)\cdot {\mathcal{Z}}_\delta(y)\,K(y)\,dy
=0,\end{equation}
in the principal value sense, thanks to the symmetry of the kernel~\eqref{EVEN}.
Consequently, we have that
\begin{equation}\label{QUI91}\begin{split}&
\int_{\R^n}\eta(0)\,\big( 
\eta(0)-\eta(y)\big)\partial_e u(0)\,\partial_e u(y)\,K(y)\,dy\\ &\qquad=\,
\int_{\R^n}\Big(\eta(0)\,\big( 
\eta(0)-\eta(y)
\big)\partial_e u(0)\,\partial_e u(y)+\eta(0)|\partial_e u(0)|^2\,
\nabla\eta(0)\cdot {\mathcal{Z}}_\delta(y)\Big)
\,K(y)\,dy.
\end{split}\end{equation}

Furthermore, we set 
\begin{equation}\label{QUI92}
I_1(y)\, :=\, \eta(0)\,\Big(\eta(0)-\eta(y)
+\nabla\eta(0)\cdot {\mathcal{Z}}_\delta(y)
\Big)\partial_e u(0)\,\partial_e u(y),
\end{equation}
\begin{equation}\label{QUI92BIS}
I_2 (y)\,:=\,
\partial_e u(0)\,\big(\eta(0)\partial_e u(0)-\eta(y)\partial_e u(y)\big)\,
\nabla\eta(0)\cdot {\mathcal{Z}}_\delta(y),\end{equation}
and
\begin{equation}\label{I3}
I_3(y)\,:=\, 
\big(\eta(y)-\eta(0)\big)\,
\partial_e u(0)\,\partial_e u(y)\,
\nabla\eta(0)\cdot {\mathcal{Z}}_\delta(y),
\end{equation}
and we
point out that
\begin{equation}\label{0x-10133}\begin{split}
& \eta(0)\,\big(\eta(0)-\eta(y)\big)\partial_e u(0)\,\partial_e u(y)
+\eta(0)|\partial_e u(0)|^2\,
\nabla\eta(0)\cdot {\mathcal{Z}}_\delta(y)\\ &\qquad\qquad
=\;
\eta(0)\,\Big(\eta(0)-\eta(y)
+\nabla\eta(0)\cdot {\mathcal{Z}}_\delta(y)
\Big)\partial_e u(0)\,\partial_e u(y)\\
&\qquad\qquad\qquad+\eta(0)\partial_e u(0)\,\big(\partial_e u(0)-\partial_e u(y)\big)\,
\nabla\eta(0)\cdot {\mathcal{Z}}_\delta(y)\\&\qquad\qquad
=\; I_1(y)+
\partial_e u(0)\,\big(\eta(0)\partial_e u(0)-\eta(y)\partial_e u(y)\big)\,
\nabla\eta(0)\cdot {\mathcal{Z}}_\delta(y)\\&\qquad\qquad\qquad
+\partial_e u(0)\,\big(\eta(y)\partial_e u(y)-\eta(0)\partial_e u(y)\big)\,
\nabla\eta(0)\cdot {\mathcal{Z}}_\delta(y)
\\ &\qquad\qquad=\; I_1(y)+I_2(y)+I_3(y).
\end{split}\end{equation}
Hence, substituting into~\eqref{QUI91}, we obtain
\begin{equation}\label{QUI220}
\int_{\R^n}\eta(0)\,\big( 
\eta(0)-\eta(y)\big)\partial_e u(0)\,\partial_e u(y)\,K(y)\,dy =
\int_{\R^n} \big( I_1(y)+I_2(y)+I_3(y)\big)\,K(y)\,dy.
\end{equation}

Now, we set
\begin{equation}\label{QUI99} \varphi_\delta(y):=\eta(0)-\eta(y)
+\nabla\eta(0)\cdot {\mathcal{Z}}_\delta(y),\end{equation}
and we observe that
\begin{equation}\label{QUI100}
\varphi_\delta(0)=0,\quad\nabla\varphi_\delta(0)=0,\quad{\mbox{ and }}\quad
\|\varphi_\delta\|_{C^2(\R^n)}\le C_\delta\,\|\eta\|_{C^2(\R^n)},
\end{equation}
for some constant~$C_\delta>0$, depending only on~$n$ and~$\delta$,
thanks to the properties~\eqref{QUI90} of~$\xi_\delta$.

We also set 
\begin{equation}\label{QUI101}
{ J_1 }\, :=\,\int_{\R^n}\partial_e\big(\varphi_\delta(y) K(y)\big)\,\big(
\eta(y)\,\partial_e u(y)-\eta(0)\,\partial_e u(0) \big)\big(u(y)-u(0)\big)\,dy\end{equation}
and
\begin{equation}\label{QUI101E}
{ J_2 }\,:=\,
\frac12\,
\int_{\R^n}\partial_e\,\Big(\partial_e \big(\varphi_\delta(y) K(y)\big)\eta(y)\Big)
\,\big|u(y)-u(0)\big|^2\,dy.
\end{equation}

We now perform some integration by parts in~$B_R$.
To this end, we use that the integrals involved in the computation are finite and that the boundary terms
on~$\partial B_R$ converge to zero as $R\to+\infty$,
thanks to
the decay of the kernel and of its derivatives assumed in~\eqref{KLIM} and~\eqref{KC1}.
More precisely, from~\eqref{QUI92} and~\eqref{QUI99}, and integrating by parts
twice, we find that
\begin{equation}\label{QUI107}
\begin{split}
\int_{\R^n} I_1(y)\,K(y)\,dy\,&
=\,\int_{\R^n}\eta(0)\,\varphi_\delta(y)\,\partial_e u(0)\,\partial_e u(y)\,K(y)\,dy\\
&=\,\int_{\R^n}\varphi_\delta(y)\,\eta(0)\,\partial_e u(0)\,\partial_e \big(u(y)-u(0)\big)\,K(y)\,dy
\\&=\,-
\int_{\R^n}\partial_e \big(\varphi_\delta(y) K(y)\big)\,\eta(0)\,\partial_e u(0) \big(u(y)-u(0)\big)\,dy
\\ &=\,
\int_{\R^n}\partial_e \big(\varphi_\delta(y) K(y)\big)\big(
\eta(y)\,\partial_e u(y)-\eta(0)\,\partial_e u(0) \big)\big(u(y)-u(0)\big)\,dy
\\&\qquad\qquad-
\int_{\R^n}\partial_e\big(\varphi_\delta(y) K(y)\big)\eta(y)\,\partial_e u(y) \big(u(y)-u(0)\big)\,dy
\\&
=\,
{ J_1 }-\frac12\,
\int_{\R^n}\partial_e \big(\varphi_\delta(y) K(y)\big)\,\eta(y)\,\partial_e 
\big|u(y)-u(0)\big|^2\,dy
\\&
=\,
{ J_1 }+\frac12\,
\int_{\R^n}{\partial_e}\,\Big( \partial_e \big(\varphi_\delta(y) K(y)\big)\eta(y)\Big)
\,\big|u(y)-u(0)\big|^2\,dy\\&
=\, { J_1 }+{ J_2 }.
\end{split}
\end{equation}

Furthermore, recalling again
the bound~\eqref{KC1} on the first and second derivatives
of the kernel and~\eqref{QUI100},
we point out that
\begin{equation}\label{QUI103}
\begin{split}&
\Big| \partial_e \big(\varphi_\delta(y) K(y)\big) \Big|\le
\big| \partial_e \varphi_\delta(y) \big|\,K(y)+\big|\varphi_\delta(y)\big|\,\big|
\partial_e K(y)\big|\\ &\qquad\le
C_{\delta,\eta} \big( K(y)+|y|\,\big|\partial_e K(y)\big|\big)
\le C_{\delta,\eta}\,K(y),\end{split}
\end{equation}
for some constant~$C_{\delta,\eta}>0$, possibly varying from line to line,
and depending 
only on~$n$, $\delta$,
$\|\eta\|_{C^2(\R^n)}$,
and on the structural constant~$\CUNO$
in~\eqref{KC1}
(for convenience, in what follows, we will rename~$C_{\delta,\eta}$
allowing dependences also
on~$s$ and on the constants~$\CDUEDUE$ and~$\CDUE$ in~\eqref{KLIM}).

In addition, using again~\eqref{KC1} and~\eqref{QUI100},
\begin{eqnarray*}
\Big|\partial^2_e\big(\varphi_\delta(y) K(y)\big)\eta(y)\Big|&\le&
C_{\delta,\eta}\,\left(
\Big|\partial_e^2\varphi_\delta(y) K(y)\Big|+
\Big|\partial_e \varphi_\delta(y)\,\partial_e K(y)\Big|+
\Big|\varphi_\delta(y) \,\partial_e^2 K(y)\Big|
\right)\\
&\le& C_{\delta,\eta}\,
\Big(K(y)+|y|\,\big|\partial_e K(y)\big|+|y|^2\big|D^2 K(y)\big|
\Big)\\
&\le& C_{\delta,\eta}\,K(y).
\end{eqnarray*}
Hence, from the latter estimate and~\eqref{QUI103},
\begin{equation}\label{QUI105}
\Big|\partial_e \,\Big(\partial_e \big(\varphi_\delta(y) K(y)\big)\eta(y)\Big)\Big|
\le C_{\delta,\eta}\,K(y).
\end{equation}

Thus, by the definitions~\eqref{QUI101} and~\eqref{QUI101E}
of~$J_1$ and~$J_2$, and the estimates~\eqref{QUI103} and~\eqref{QUI105},
using an appropriate Cauchy-Schwarz inequality
we get that
\begin{equation}\label{VECCHIA J1J2}\begin{split}&
|{ J_1 }|+|{ J_2 }|\\
&\qquad\le\, C_{\delta,\eta}\,\left(
\int_{\R^n}
\big|
\eta(y)\,\partial_e u(y)-\eta(0)\,\partial_e u(0) \big|\,\big|u(y)-u(0)\big|\,K(y)\,dy\right.\\&\qquad\qquad\qquad\left.+
\int_{\R^n} \big|u(y)-u(0)\big|^2\,K(y)\,dy\right)\\&\qquad
\le\,\frac18\,
\int_{\R^n}
\big|\eta(y)\,\partial_e u(y)-\eta(0)\,\partial_e u(0) \big|^2\,K(y)\,dy+C_{\delta,\eta}\,
\int_{\R^n} \big|u(y)-u(0)\big|^2\,K(y)\,dy.
\end{split}\end{equation}
This and~\eqref{QUI107} give that
\begin{equation}\label{QUI114}\begin{split}&
\int_{\R^n} I_1(y)\,K(y)\,dy\\&\qquad
\le\;
\frac18\,
\int_{\R^n}
\big|\eta(y)\,\partial_e u(y)-\eta(0)\,\partial_e u(0) \big|^2\,K(y)\,dy+C_{\delta,\eta}\,
\int_{\R^n} \big|u(y)-u(0)\big|^2\,K(y)\,dy.
\end{split}\end{equation}

Next, we notice that~$|{\mathcal{Z}}_\delta(y)|\le C\,|y|$, for every~$y\in\R^n$, for some constant~$C$ depending only
on~$n$. 
Thus, we have that
\begin{equation} \label{QUI109} J_3:=
\int_{B_1} |y|\, |{\mathcal{Z}}_\delta(y)|\,K(y)\,dy<+\infty.\end{equation}
In view of the definition~\eqref{QUI92BIS} of~$I_2$, 
noticing that~${\mathcal{Z}}_\delta$ is supported in~$B_1$,
and using a suitable Cauchy-Schwarz inequality,
we have that
\begin{equation}\label{QUI111}\begin{split}
&
\int_{\R^n} I_2(y)\,K(y)\,dy\\ &\qquad\le\,C_{\eta}\,
\int_{B_1} \big| \partial_e u(0)\big|
\big|\eta(0)\partial_e u(0)-\eta(y)\partial_e u(y)\big|\;|{\mathcal{Z}}_\delta(y)|\,K(y)\,dy\\&\qquad
\le\,
\frac18\,
\int_{\R^n} 
\big|\eta(0)\partial_e u(0)-\eta(y)\partial_e u(y)\big|^2\,K(y)\,dy
+
C_{\eta}\,
\int_{B_1} \big| \partial_e u(0)\big|^2\;|{\mathcal{Z}}_\delta(y)|^2\,K(y)\,dy\\&\qquad
\le\, \frac18\,
\int_{\R^n} 
\big|\eta(0)\partial_e u(0)-\eta(y)\partial_e u(y)\big|^2\,K(y)\,dy
+C_{\eta}\,\|\partial_e u\|_{L^\infty(B_1)}^2\,J_3
\end{split}\end{equation}
for some constant~$C_\eta>0$ depending only
on~$n$ and~$\| \eta\|_{C^2(\R^n)}$.
Similarly, recalling the definition~\eqref{I3}
of~$I_3$,
\begin{equation}\label{QUI112}\begin{split}
\int_{\R^n} I_3(y)\,K(y)\,dy\,\le\,&C_{\eta}\,
\int_{B_1}\big|\eta(y)-\eta(0)\big|\,
|\partial_e u(0)|\,|\partial_e u(y)|\;|
{\mathcal{Z}}_\delta(y)|\,K(y)\,dy\\
\le\,&C_{\eta}\,
\int_{B_1}|y|\,
|\partial_e u(0)|\,|\partial_e u(y)|\;|
{\mathcal{Z}}_\delta(y)|\,K(y)\,dy\\
\le\,&C_{\eta}\,\|\partial_e u\|_{L^\infty(B_1)}^2\,J_3.
\end{split}\end{equation}

Now we provide a bound on~$J_3$. To this aim, we split
the integral computation inside~$B_\delta$, where~${\mathcal{Z}}_\delta(y)=y$,
and in~$B_1\setminus B_\delta$, where~$|{\mathcal{Z}}_\delta(y)|\le 2\delta\,n^{1/2}$.
In this way, recalling the definition~\eqref{QUI109}
of~$J_3$, we get that
\begin{eqnarray*} |J_3|
&\le& C\,\left(
\int_{B_\delta} |y|^2\,K(y)\,dy
+\delta\int_{B_1\setminus B_\delta} |y|\, K(y)\,dy
\right),
\end{eqnarray*}
for some constant~$C$ depending only on~$n$.
The latter quantity tends to zero as~$\delta\searrow0$,
thanks to the bounds~\eqref{KLIM} on the kernel.
Given~$\e>0$, we can therefore
take~$\delta\in(0,1)$ sufficiently small,
in such a way that
\begin{equation*} |J_3|\le\e^2.\end{equation*}
After this choice of~$\delta$,
the constants~$C_{\delta,\eta}$ above
will be written accordingly,
with a slight abuse of notation, as~$C_{\e,\eta}$.
As a consequence, collecting the estimates in~\eqref{QUI114},
\eqref{QUI111}, and~\eqref{QUI112}, we conclude that
\begin{eqnarray*}&&
\int_{\R^n} \big(I_1(y)+I_2(y)+I_3(y)\big)\,K(y)\,dy\\&&\qquad\qquad
\le \frac14\,
\int_{\R^n}
\big|\eta(y)\,\partial_e u(y)-\eta(0)\,\partial_e u(0) \big|^2\,K(y)\,dy\\
&&\qquad\qquad\qquad+C_{\e,\eta}\,
\int_{\R^n} \big|u(y)-u(0)\big|^2\,K(y)\,dy+
C_{\eta}\,\e^2\,\|\partial_e u\|_{L^\infty(B_1)}^2.
\end{eqnarray*}
Therefore, recalling~\eqref{QUI220}, we deduce that
\begin{eqnarray*}&&
\int_{\R^n}\eta(0)\,\big( 
\eta(0)-\eta(y)\big)\partial_e u(0)\,\partial_e u(y)\,K(y)\,dy
\\&&\qquad\qquad
\le \frac14\,
\int_{\R^n}
\big|\eta(y)\,\partial_e u(y)-\eta(0)\,\partial_e u(0) \big|^2\,K(y)\,dy\\
&&\qquad\qquad\qquad+C_{\e,\eta}\,
\int_{\R^n} \big|u(y)-u(0)\big|^2\,K(y)\,dy+
C_{\eta}\,\e^2\,\|\partial_e u\|_{L^\infty(B_1)}^2.
\end{eqnarray*}
This establishes~\eqref{T6.3}, up to renaming~$\e$ and the constants,
by choosing~$\sigma_\e$ large enough.
As a consequence, also~\eqref{T6} follows.
\end{proof}

\subsection{``Downstairs''
proof of the first key inequality for $s$-harmonic functions}\label{9o23939398iww3456iw}

The next result is the particular
case of Proposition~\ref{Berns-frac} where $u$ is $s$-harmonic.
While the proof of Proposition~\ref{Berns-frac}
will rely on extension methods, we provide here a
proof of this particular case
without using the extension. Note that, as in Proposition~\ref{Berns-frac}, $\sigma_0$ is independent of~$s$.

\begin{lemma}\label{849027674843hf38}
Let~$\eta\in C^\infty(\R^n)\cap W^{2,\infty}(\R^n)$,
and assume that~$u\in W^{1,\infty}(\R^n)$
is a (weak) solution of
\begin{equation}\label{92i-9393}
(-\Delta)^s u =0 \quad{\mbox{ in }}B_1.
\end{equation}

Then, there exists~$\sigma_0>0$,
depending only on~$n$ and~$\| \eta\|_{C^2(\R^n)}$, such that
\begin{equation}\label{10-10-XT6}
(-\Delta)^s \big( \eta^2|\nabla u|^2+\sigma u^2\big)\le 0\quad{\mbox{ in }}B_1\quad
\text{ if } \sigma\ge \sigma_0.
\end{equation}\end{lemma}

\begin{proof}[Proof $($without using the extension problem$)$] 
Though the statement
of Lemma~\ref{849027674843hf38} is specific for the fractional Laplacian,
we perform the initial part of the proof arguing for a general kernel~$K$,
to isolate the only point where we will use that~$K(z)=c_{n,s}|z|^{-n-2s}$.

The proof relies on several integrations by
parts, which carefully take into account oscillations
and compensations inside the integrals.

We observe that, by regularity results for~\eqref{92i-9393}, $u\in C^\infty(B_1)$.
Exploiting Proposition~\ref{NUOVA:P}, in order
to prove~\eqref{10-10-XT6},
it suffices to show, by translation invariance, that
\begin{equation}\label{01234-11-T6.3}\begin{split}&
2\int_{\R^n}\eta(0)\,\big( 
\eta(0)-\eta(y)\big)\nabla u(0)\cdot\nabla u(y)\,K(y)\,dy\\ &\qquad\quad\leq\,
\int_{\R^n} \big|
\eta(0)\nabla u(0)-\eta(y)\nabla u(y)\big|^2\,K(y)\,dy+
\sigma\,\int_{\R^n}\big|u(0)-u(y)\big|^2\,K(y)\,dy
\end{split}
\end{equation}
knowing that
\begin{equation}\label{STARR}
(-\Delta)^s u (0)=0.
\end{equation}
To prove this, we call~$I_1$ the left-hand side of~\eqref{01234-11-T6.3}.
Integrating by parts, we have
\begin{equation}\label{01p10}
\begin{split}
I_1\,&=
2\int_{\R^n}\eta(0)\,\big( 
\eta(0)-\eta(y)\big)\nabla u(0)\cdot\nabla \big(u(y)-u(0)\big)\,K(y)\,dy\\
&=
-2\int_{\R^n}\eta(0)\,\nabla u(0) \big(u(y)-u(0)\big)\cdot
\nabla\Big( \big( 
\eta(0)-\eta(y)\big)\,K(y)\Big)\,dy.
\end{split}
\end{equation}
We remark that, to obtain this integration by parts identity,
one must argue in balls~$B_R$ and use that the boundary terms
on~$\partial B_R$ go to zero as~$R\to+\infty$ (as well as the integrability in~$\R^n$ of the
above functions), thanks to
the decay of the kernel
and of its derivatives assumed in~\eqref{KLIM} and~\eqref{KC1}.

Moreover, from the bound~\eqref{KC1}
on the first derivative of the kernel,
\begin{equation}\label{RIAS} \Big|
\nabla\Big( \big( 
\eta(0)-\eta(y)\big)\,K(y)\Big)\Big|\le
\big|\nabla\eta(y)\big|\,K(y)+
\big| 
\eta(0)-\eta(y)\big|\,\big|\nabla K(y)\big|
\le CK(y),\end{equation}
where, from now on, 
$C$ denotes different constants depending only on~$n$ and~$\|\eta\|_{C^2(\R^n)}$
(in particular, independent of~$s$ in our case, that is, when
the kernel~$K$ is that of the fractional Laplacian).
Therefore, we can bound~\eqref{01p10} as
\begin{equation*}
\begin{split}
I_1\,&=
2\int_{\R^n}\big(\eta(y)\,\nabla u(y)-\eta(0)\,\nabla u(0)\big) \big(u(y)-u(0)\big)\cdot
\nabla\Big( \big( 
\eta(0)-\eta(y)\big)\,K(y)\Big)\,dy\\&\quad\quad
-2\int_{\R^n}\eta(y)\,\nabla u(y) \big(u(y)-u(0)\big)\cdot
\nabla\Big( \big( 
\eta(0)-\eta(y)\big)\,K(y)\Big)\,dy
\\ &\le CI_2-2\int_{\R^n}\eta(y)\,\nabla u(y) \big(u(y)-u(0)\big)\cdot
\nabla\Big( \big( 
\eta(0)-\eta(y)\big)\,K(y)\Big)\,dy,
\end{split}
\end{equation*}
with
\begin{equation*}
I_2:=
\int_{\R^n}\big|\eta(y)\,\nabla u(y)-\eta(0)\,\nabla u(0)\big| \big|u(y)-u(0)\big|
K(y)\,dy.
\end{equation*}
Now, integrating by parts and using again~\eqref{RIAS}, we have
\begin{equation*}
\begin{split}
& 2\left| \int_{\R^n}\eta(y)\,\nabla u(y) \big(u(y)-u(0)\big)\cdot
\nabla\Big( \big( 
\eta(0)-\eta(y)\big)\,K(y)\Big)\,dy\right|\\&\qquad\quad
=\;\left| \int_{\R^n}\eta(y)\,\nabla \big|u(y)-u(0)\big|^2\cdot
\nabla\Big( \big( 
\eta(0)-\eta(y)\big)\,K(y)\Big)\,dy\right| \\&\qquad\quad=\;\left|
\int_{\R^n}\big|u(y)-u(0)\big|^2\, {\rm div}\Big(\eta(y)
\nabla\Big( \big( 
\eta(0)-\eta(y)\big)\,K(y)\Big)\Big)\,dy\right|\\&\qquad\quad \le\;
C
\int_{\R^n}\big|u(y)-u(0)\big|^2\,|\nabla\eta(y)|\,K(y)\,dy
\\&\qquad\quad\qquad\quad+\left|
\int_{\R^n}\big|u(y)-u(0)\big|^2\,\eta(y)\,
\Delta\Big( \big( 
\eta(0)-\eta(y)\big)\,K(y)\Big)\,dy\right|\\ &\qquad\quad\le\;
CI_3+2|T_1|
+|T_2|,\end{split}\end{equation*}
with
\begin{equation*}
I_3:=\int_{\R^n}\big|u(y)-u(0)\big|^2\,K(y)\,dy,
\end{equation*}
\begin{equation*}
T_1:=
\int_{\R^n}\big|u(y)-u(0)\big|^2\,\eta(y)\,
\nabla\eta(y)\cdot\nabla K(y) \,dy,\end{equation*}
and 
\begin{equation}\label{DTT}
T_2:=
\int_{\R^n}\big|u(y)-u(0)\big|^2\,\eta(y)\,
\big( 
\eta(0)-\eta(y)\big)\,\Delta K(y) \,dy.
\end{equation}

Clearly, $I_2$ and~$I_3$ are ``good terms'' which are controlled by the
right-hand side of~\eqref{01234-11-T6.3}. Hence, to bound~$I_1$ it remains to control~$|T_1|$ and~$|T_2|$.

To estimate~$T_1$, we observe that
\begin{equation*} 
\big|\nabla\eta(y)-\nabla\eta(0)\big|\,|\nabla K(y)|
\le C\,|y|\,|\nabla K(y)|
\le CK(y),\end{equation*}
thanks to the bound~\eqref{KC1}
on the first derivative of the kernel, and therefore, integrating by parts,
\begin{equation*}
\begin{split}
|T_1|\;&\le\left|
\int_{\R^n}\big|u(y)-u(0)\big|^2\,\eta(y)\,
\nabla\eta(0)\cdot\nabla K(y) \,dy\right|\\&\qquad
+\left|
\int_{\R^n}\big|u(y)-u(0)\big|^2\,\eta(y)\,
\big(\nabla\eta(y)-\nabla\eta(0)\big)\cdot\nabla K(y) \,dy\right|\\&\le\left|
\int_{\R^n}\nabla\Big(\big|u(y)-u(0)\big|^2\,\eta(y)\Big)\cdot
\nabla\eta(0)\, K(y) \,dy\right|
+CI_3\\&\le\left|
\int_{\R^n}\nabla\big|u(y)-u(0)\big|^2\,\eta(y)\cdot
\nabla\eta(0)\, K(y) \,dy\right|
+CI_3\\&\le 2\left|
\int_{\R^n}\big(u(y)-u(0)\big)\,\eta(y)\,\nabla u(y)\cdot
\nabla\eta(0)\, K(y) \,dy\right|
+CI_3\\
&\le2\left|
\int_{\R^n}\big(u(y)-u(0)\big)\,\eta(0)\,\nabla u(0)\cdot
\nabla\eta(0)\, K(y) \,dy\right|+C(I_2+I_3)\\
&=C(I_2+I_3),
\end{split}
\end{equation*}
where \eqref{STARR} has been used in the last line.

Finally, we estimate~$T_2$. To this end, we take~$\zeta\in C^\infty_c(B_1)$
with~$\zeta=1$ in~$B_{1/2}$, and we define
$$ T_3:=\int_{\R^n}\big|u(y)-u(0)\big|^2\,\eta(y)\,
\nabla\eta(0)\cdot y\,\zeta(y)\,\Delta K(y) \,dy.$$
Then,
by the second derivative bound~\eqref{KC1} on the kernel and~\eqref{DTT}, we have
\begin{equation}\label{0203487ow023939934}
\begin{split}
|T_2|\,&=\left|
\int_{\R^n}\big|u(y)-u(0)\big|^2\,\eta(y)\,
\big( 
\eta(0)-\eta(y)+\nabla\eta(0)\cdot y\,\zeta(y)\big)\,\Delta K(y) \,dy-T_3\right|\\
&\le C
\int_{\R^n}\big|u(y)-u(0)\big|^2\,|y|^2\,|D^2 K(y)| \,dy
+|T_3|\\&\le
CI_3+|T_3|.
\end{split}\end{equation}

Thus, it only remains to bound~$T_3$.
For this, when the kernel~$K$ is that of the fractional Laplacian, we have that
\begin{eqnarray*}&& 2(s+1)\nabla K(y)+y\Delta K(y)\\&&\qquad=-
c_{n,s} \,(n+2s)\,y\,\left( \frac{2(s+1)}{|y|^{n+2s+2}}
+\sum_{i=1}^n \left( \frac1{|y|^{n+2s+2}}-\frac{(n+2s+2)\,y_i^2}{|y|^{n+2s+4}}\right)
\right)=0.\end{eqnarray*}
Consequently, using once more
the first derivative bound~\eqref{KC1}
on the kernel, and integrating again by parts, we find that
\begin{eqnarray*}
|T_3|&=& 2(s+1)\,\left|
\int_{\R^n}\big|u(y)-u(0)\big|^2\,\eta(y)\,
\nabla\eta(0)\,\zeta(y)\cdot\nabla K(y) \,dy\right|
\\&\le& 2(s+1)\,\left|
\int_{\R^n}\big|u(y)-u(0)\big|^2\,\eta(y)\,
\nabla\eta(0)\,\zeta(0)\cdot\nabla K(y) \,dy\right|\\&&\quad
+C
\int_{\R^n}\big|u(y)-u(0)\big|^2\,\eta(y)\,|
\nabla\eta(0)|\,|y|\,|\nabla K(y)| \,dy
\\ &\le& 
2(s+1)\,\left|
\int_{\R^n} \nabla\Big(\big|u(y)-u(0)\big|^2\,\eta(y)\Big)\cdot
\nabla\eta(0)\,\zeta(0)\, K(y) \,dy\right|
+CI_3\\
\\ &\le& 
2(s+1)\,\left|
\int_{\R^n} \nabla \big|u(y)-u(0)\big|^2\,\eta(y)\cdot
\nabla\eta(0)\,\zeta(0)\, K(y) \,dy\right|
+CI_3\\
&=& 
4(s+1)\,\left|
\int_{\R^n} \big(u(y)-u(0)\big)\,\eta(y)\,
\nabla u(y)\cdot
\nabla\eta(0)\,\zeta(0)\, K(y) \,dy\right|
+CI_3\\&\le&
4(s+1)\,\left|
\int_{\R^n} \big(u(y)-u(0)\big)\,\eta(0)\,
\nabla u(0)\cdot
\nabla\eta(0)\,\zeta(0)\, K(y) \,dy\right|
\\&&\quad+
4(s+1)\,\left|
\int_{\R^n} \big(u(y)-u(0)\big)\,\big(\eta(y)\,
\nabla u(y)-\eta(0)\,\nabla u(0)\big)\cdot
\nabla\eta(0)\,\zeta(0)\, K(y) \,dy\right|\\&&\quad
+CI_3\\
&\le & 0+C(I_2+I_3),
\end{eqnarray*}
where we have exploited~\eqref{STARR} once more in the last step.
\end{proof}

\section{Linearized operator and a maximum estimate}\label{61}

\subsection{The linearized operator}\label{sub41s}

Here we study the linearized equation
associated with the nonlinear problem \eqref{EQ}. To simplify notation,
we denote by~${\mathcal{L}}_1,\dots,{\mathcal{L}}_J$ the linear operators~${\mathcal{L}}_{\mu_1},\dots,
{\mathcal{L}}_{\mu_J}$.
We will always assume that the convexity and ellipticity conditions~\eqref{CONV:POSITIVA},
\eqref{MON:F},
and~\eqref{CONV:F} are satisfied,
for some constants~$\Theta_0\ge\vartheta_0>0$. 

Given a function~$u$,
we use the short notation
\begin{equation}\label{defalphak} \alpha_j:=
\alpha_j
\big({\mathcal{L}}_1u(x)-g_1(x),\dots,{\mathcal{L}}_Ju(x)-g_J(x)\big)\end{equation}
and
we consider the operator
\begin{equation}\label{L:EQ:OP} 
Lv := \sum_{j=1}^J \alpha_j \,{\mathcal{L}}_{j}v.\end{equation}
It will be important that the functions $\alpha_{j}$ are defined at all points of our domain ---and not only almost everywhere--- since we will need to evaluate them at a maximum point of an auxiliary function. This will always be
possible since ${\mathcal{L}}_j u$ and $g_{j}$ will be finite and well defined at all points, by the regularity assumed on $u$ and since  $g_{j}$ are continuous functions.

The relevance of the linearized operator~$L$
is given by the fact that the solution and its derivatives
satisfy suitable inequalities with respect to~$L$, as stated in
the following result. Here, we remark that since~$g_j$ is not better than semiconcave, its first and second derivatives only exist almost everywhere.

\begin{lemma}\label{SUBSOLU:EXTE}
Let~$g_j$
be Lipschitz functions in~$B_1$ for~$j=1,\dots,J$,
$u\in C^\infty(B_1)\cap W^{1,\infty}(\R^n)$ be a solution of~\eqref{EQ}
everywhere
in~$B_1$, and $e\in\R^n$ satisfy~$|e|=1$.

Then,
\begin{eqnarray*}
&& Lu\ge -F(-g_1,\dots,-g_J)}
\quad{\mbox{ everywhere in }B_1\end{eqnarray*}
and
\begin{eqnarray*}
&& L\,\partial_{e} u
=
\sum_{j=1}^J \alpha_j\partial_e g_j
\quad{\mbox{ almost everywhere in }}B_1.
\end{eqnarray*}
In particular,
$$ Lu\ge -|F(0)|-\Theta_0\sup_{j\in\{1,\dots,J\}}\|g_j\|_{L^\infty(B_1)}\quad{\mbox{ everywhere in }}B_1.$$

If in addition the functions~$g_j$
are semiconcave in~$B_1$ and~$u\in W^{2,\infty}(\R^n)$, then
\[ L\,\partial_e^2 u \le 
\sum_{j=1}^J \alpha_j\partial_e^2 g_j
\quad{\mbox{ almost everywhere in }}B_1.
\]
\end{lemma}

\begin{proof} We prove the first statement
by using equation~\eqref{EQ}, in combination
with the convexity assumption~\eqref{CONV:F}, used here
with~$q:=(-g_1,\dots,-g_J)$ and~$p:=({\mathcal{L}}_1 {u}-g_1,
\dots,{\mathcal{L}}_J {u}-g_J)$. In this way,
we have that, everywhere in~$B_1$,
\begin{eqnarray*} 
&& F(-g_1,\dots,-g_J)-F({\mathcal{L}}_1 {u}-g_1,
\dots,{\mathcal{L}}_J {u}-g_J)
\ge -\sum_{j=1}^J\alpha_j({\mathcal{L}}_1 {u}-g_1
,\dots,{\mathcal{L}}_J {u}-g_J)\,
{\mathcal{L}}_j{u},
\end{eqnarray*}
which establishes the first statement.

Now, by~\eqref{CONV:POSITIVA}, \eqref{MON:F} and~\eqref{CONV:F} (exploited 
 with~$q:=0$ and~$p:=(-g_1,\dots,-g_J)$), we see that
\begin{eqnarray*}&&
|F(0)|-F(-g_1,\dots,-g_J)\ge
F(0)-F(-g_1,\dots,-g_J)\ge\sum_{j=1}^J\alpha_j(-g_1,\dots,-g_J)g_j
\\&&\qquad\ge
-\sum_{j=1}^J \alpha_j(-g_1,\dots,-g_J) \|g_j\|_{L^\infty(B_1)}\ge
-\sum_{j=1}^J \alpha_j(-g_1,\dots,-g_J) \sup_{\ell\in\{1,\dots,J\}}\|g_\ell\|_{L^\infty(B_1)}\\&&\qquad
\ge-\Theta_0\sup_{\ell\in\{1,\dots,J\}}\|g_\ell\|_{L^\infty(B_1)}.
\end{eqnarray*}
This observation and the first statement of the lemma lead to the third one.

Now we prove the second statement. To this end,
we let~$\e>0$ and exploit the convexity assumption~\eqref{CONV:F}
at~$x\in B_1$,
with
\begin{equation}\label{qep:1} q:=
({\mathcal{L}}_1u(x\pm \e e)-g_1(x\pm \e e)
,\dots,{\mathcal{L}}_Ju(x\pm \e e)-g_J(x\pm \e e))
\end{equation}
and
\begin{equation}\label{qep:2}
p:=
({\mathcal{L}}_1u(x)-g_1(x),\dots,
{\mathcal{L}}_Ju(x)-g_J(x)).\end{equation}
Using~\eqref{EQ}, this leads to
\begin{eqnarray*}
0
&=&\frac{1}{\e}
\Big\{
F\big({\mathcal{L}}_1u(x\pm \e e)
-g_1(x\pm \e e),\dots,{\mathcal{L}}_Ju(x\pm \e e)-g_J(x\pm \e e)\big)
\\&&\qquad-F\big({\mathcal{L}}_1u(x)
-g_1(x),\dots,{\mathcal{L}}_Ju(x)-g_J(x)\big)
\Big\}
\\
&\ge&\sum_{j=1}^J
\alpha_j
\big({\mathcal{L}}_1u(x)-g_1(x),\dots,{\mathcal{L}}_Ju(x)-g_J(x)\big)\\
&&\qquad\cdot\frac{
{\mathcal{L}}_j u(x\pm \e e)-{\mathcal{L}}_j u(x)
-g_j(x\pm \e e)+g_j(x)
}{\e}.\end{eqnarray*}
Taking the limit as~$\e\searrow0$, we thereby find that,
for almost
every~$x\in B_1$,
\begin{eqnarray*} 
0&\ge&\pm\sum_{j=1}^J
\alpha_j
\big({\mathcal{L}}_1u(x)-g_1(x),\dots,{\mathcal{L}}_Ju(x)-g_J(x)\big)
\big({\mathcal{L}}_j\partial_e u(x)-\partial_eg_j(x)\big)\\&=&
\pm L\,\partial_eu(x)\mp\sum_{j=1}^J
\alpha_j \partial_eg_j (x).
\end{eqnarray*}
{F}rom this, we deduce the second statement.

Finally, we prove the last statement.
For this,
exploiting again~\eqref{CONV:F}
with~$q$ and~$p$ as above,
we have that
\begin{eqnarray*}
0&=&\frac{1}{\e^2}\,\Bigg\{\Big(
F\big({\mathcal{L}}_1u(x+ \e e)-g_1(x+ \e e),\dots,{\mathcal{L}}_Ju
(x+ \e e)-g_J(x+ \e e)\big)\\&&\qquad\qquad
-F\big({\mathcal{L}}_1u(x)-g_1(x),\dots,{\mathcal{L}}_Ju(x)-g_J(x)\big)\Big)\\&&\qquad\qquad+
\Big( 
F\big({\mathcal{L}}_1u(x-\e e)-g_1(x-\e e),\dots,{\mathcal{L}}_J
u(x- \e e)-g_J(x-\e e)\big)\\&&\qquad\qquad
-F\big({\mathcal{L}}_1u(x)-g_1(x),\dots,{\mathcal{L}}_Ju(x)-g_J(x)\big)\Big)\Bigg\}
\\
&\ge&\sum_{j=1}^J
\alpha_j
\big({\mathcal{L}}_1u(x)-g_1(x),\dots,{\mathcal{L}}_Ju(x)-g_J(x)\big)\\
&&\qquad\qquad\cdot\Bigg\{ \frac{
\big(
{\mathcal{L}}_ju(x+\e e)-{\mathcal{L}}_ju(x)\big)+
\big( {\mathcal{L}}_ju(x-\e e)-{\mathcal{L}}_ju(x)\big)
}{\e^2}\\
&&\qquad\qquad\qquad-\frac{\big(g_j(x+\e e)-g_j(x)\big)+\big(g_j(x-\e e)-g_j(x)\big)}{\e^2}\Bigg\}
\\
&=&\sum_{j=1}^J
\alpha_j
\;\Bigg\{
\frac{
{\mathcal{L}}_ju(x+\e e)+
{\mathcal{L}}_ju(x-\e e)-2{\mathcal{L}}_ju(x)
}{\e^2}\\&&\qquad\qquad\qquad-
\frac{
g_j(x+\e e)+
g_j(x-\e e)-2g_j(x)
}{\e^2}
\Bigg\}.
\end{eqnarray*}
Sending~$\e\searrow0$, we conclude that~$0\ge\sum_{j=1}^J
\alpha_j
\,\Big( {\mathcal{L}}_j\partial^2_e u(x)-\partial^2_e g_j(x)\Big)$
for almost
every~$x\in B_1$,
as desired.
\end{proof}

\subsection{A maximum estimate}

Here we show that the linearized operator~$L$ in~\eqref{L:EQ:OP}
satisfies the maximum principle, as well as a quantitative maximum estimate in the case
of nonzero right-hand sides.
This is the content of the following result, which will be proved using a barrier function. As customary, we use the notation~$C_{\rm b}(\R^n):=C(\R^n)\cap L^\infty(\R^n)$ to denote the space of bounded and continuous
functions over all~$\R^n$.

\begin{proposition}\label{ujMAX}
Let~$\gamma_0\ge0$.
Let~$\varphi\in C_{\rm b}(\R^n)\cap {W^{2,\infty}(B_1)}$ be a
nonnegative
function in~$\R^n$ such that
\begin{equation}\label{98suidvyfzsdaq}
L\varphi\le\gamma_0\quad{\mbox{ everywhere in }}B_1,
\end{equation}
where~$L$ is given by~\eqref{defalphak}--\eqref{L:EQ:OP} for some~$u\in C^\infty(B_1)\cap W^{1,\infty}(\R^n)$
and continuous functions~$ g _1,\dots, g _J$ in~$B_1$.

Then,
\begin{equation} \label{DClap}\sup_{{B_1}}\varphi \le
\sup_{\R^n\setminus{B_1}}\varphi +C\,\gamma_0,\end{equation}
where $C$ depends only on $n$, $\vartheta_0$, and~$\Theta_0$.
Recall that~$\vartheta_0$ and $\Theta_0$ are the constants in
hypothesis~\eqref{MON:F}.
\end{proposition}

\begin{remark}\label{R35}
{\rm 
Concerning the statement~\eqref{98suidvyfzsdaq},
we stress that, in our applications, we will need to use Proposition~\ref{ujMAX}
for~$C_{\rm b}(\R^n)\cap {W^{2,\infty}(B_1)}$ functions~$\varphi$
which are not~$C^2$. For instance, in some cases the auxiliary function $\varphi$ will contain a term involving~$(\partial_e^2u)_+^2$ (note that this function is locally~$W^{2,\infty}$ when~$u$ is smooth, even if~$(\partial_e^2u)_+$ is not~$W^{2,\infty}$
in general). Thus, for a $C_{\rm b}(\R^n)\cap {W^{2,\infty}(B_1)}$ function~$\varphi$,
we now define
a precise meaning to~\eqref{98suidvyfzsdaq} at {\it all} points of~$B_1$,
obtaining a finite value for~$L\varphi$.

First, we use the usual
principal value definition for all the integro-differential operators considered through the paper, since the regularity of $\varphi$ is sufficient
to compute the integrals (involved in $L\varphi$) in the principal value sense, obtaining a finite value.

Secondly, when $L\varphi$ involves computing the classical Laplacian ---as it may be the case for ${\mathcal{L}}_\mu$ in \eqref{OP:L}---, we define~$-\Delta\varphi$
through second incremental quotients as
\begin{equation}\label{CAV}
-\Delta\varphi(x):=\liminf_{h\searrow0}\sum_{i=1}^n\frac{2\varphi(x)-\varphi(x+he_i)-\varphi(x-he_i)}{h^2},
\end{equation}
which is finite since~$\varphi\in W^{2,\infty}$.
As a matter of fact, in our applications, inequality~\eqref{98suidvyfzsdaq}
will be satisfied even when writing~\eqref{CAV} with a $\limsup$ instead of~$\liminf$.
}
\end{remark}

To establish Proposition~\ref{ujMAX}, we start giving an auxiliary barrier function.

\begin{lemma}\label{02-3-45}
Let~$s\in(0,1)$ and
\begin{equation}\label{BETA:DEF} \beta(x):= \left\{ \begin{matrix}
|x|^2 -2 & {\mbox{ if $x\in B_{10}$, }}\\
98 & {\mbox{ otherwise.}}
\end{matrix}
\right.\end{equation}

Then, if $K$ satisfies~\eqref{EVEN}
and~\eqref{KLIM}, we have that
\begin{equation}\label{CASO:K}
{\mathcal{L}}_K \beta\le -c\quad{\mbox{everywhere in }}B_1,\end{equation}
for some constant~$c>0$ depending only on~$n$ and
the structural constant~$\CDUEDUE$ in~\eqref{KLIM}.

Similarly, with~${\mathcal{L}}_\mu$ defined by~\eqref{1.0}--\eqref{OP:L},
we have that
\begin{equation}\label{CASO:MU}
{\mathcal{L}}_\mu \beta\le -c\quad{\mbox{everywhere in }}B_1,\end{equation}
for some constant~$c>0$ depending only on~$n$.

In particular, for every~$s\in[0,1]$,
\begin{equation}\label{CASO:MU:LA}
(-\Delta)^s \beta\le -c\quad{\mbox{everywhere in }}B_1,\end{equation}
for some constant~$c>0$ depending only on~$n$.
\end{lemma}

\begin{proof} Let us start by proving~\eqref{CASO:K}.
Let~$x\in B_1$.
We observe that if~$z\in B_{9}$ then~$|x\pm z|\le|x|+|z|<10$, and hence~$
\beta(x\pm z)
=|x\pm z|^2-2
$.
As a consequence, for every~$x\in B_1$ and~$z\in B_{9}$,
\begin{equation}\label{HneIJA-a}
\begin{split}& \beta(x+z)+\beta(x-z)-2\beta(x)
\\&\qquad\quad =\big(
|x|^2+|z|^2+2x\cdot z-2\big)+\big(
|x|^2+|z|^2-2x\cdot z-2\big)-2\big(
|x|^2-2\big)= 2|z|^2.\end{split}\end{equation}
Accordingly
\begin{equation}\label{891-234}
\int_{B_{9}} \Big( \beta(x+z)+\beta(x-z)-2\beta(x)\Big)\,K(z)\,dz=
2\int_{B_{9}} |z|^2\,K(z)\,dz.
\end{equation}

On the other hand, if~$z\in \R^n\setminus B_{9}$,
we have that~$|x\pm z|\ge|z|-|x|\ge8$.
Therefore, for every~$z\in \R^n\setminus B_{9}$,
we have that~$\beta(x\pm z)\ge62$.
Since~$\beta(x)\le-1$, we find in this case that
\begin{equation*}
\int_{\R^n\setminus B_9} \Big( \beta(x+z)+\beta(x-z)-2\beta(x)\Big)\,K(z)\,dz
\ge 126\int_{\R^n\setminus B_9} K(z)\,dz.
\end{equation*}
This inequality 
and~\eqref{891-234},
together with the symmetry~\eqref{EVEN} of the kernel
and the lower bound~\eqref{KLIM},
lead to
\begin{eqnarray*}
-{\mathcal{L}}_K \beta(x)&=&
\frac12\,\int_{\R^n}\big( \beta(x+z)+\beta(x-z)-2\beta(x)\big)\,K(z)\,dz\\
&\ge& \int_{B_{9}} |z|^2\,K(z)\,dz
+\int_{\R^n\setminus B_9} K(z)\,dz\\&\ge&
\CDUEDUE\;s\,(1-s)\,\left(
\int_{B_{9}} \frac{dz}{|z|^{n+2s-2}}
+\int_{\R^n\setminus B_9} \frac{dz}{|z|^{n+2s}}
\right)\\&\ge&
\widetilde{C}\,\CDUEDUE,\end{eqnarray*}
for some constant~$\widetilde C$ only depending on~$n$. This gives the desired result~\eqref{CASO:K}.

Now we prove~\eqref{CASO:MU}.
For this, we let~$s\in(0,1)$
and apply~\eqref{CASO:K} with~$K(z):=c_{n,s}\,|z|^{-n-2s}$, obtaining~\eqref{CASO:MU:LA}
with~$c>0$ depending only on~$n$
(since it is well known that for this kernel one may take~$C_1$ to depend only on~$n$).
We also notice that~$-\Delta \beta =-2n$
and~$\beta\le -1$ in~$ B_1$, and therefore~\eqref{CASO:MU:LA}
is satisfied also for~$s=0$
and~$s=1$, up to renaming~$c$. Consequently, the claim~\eqref{CASO:MU} follows
from
the definition~\eqref{OP:L} of~${\mathcal{L}}_\mu$.

Note that claim~\eqref{CASO:MU:LA},
that we have already proved,
can also be regarded as a particular case of~\eqref{CASO:MU} by taking~$\mu$ to be a Dirac's delta.
\end{proof}

With the above barrier at hand, we are in the position of proving the maximum principle
with estimate which
is suitable for our goals.

\begin{proof}[Proof of Proposition~\ref{ujMAX}]
We take~$\beta$ as in Lemma~\ref{02-3-45}
and define
\begin{equation}\label{09p29} \varphi_*(x):=\varphi(x)+ \frac{
C\,\gamma_0}{100}\,\beta(x).
\end{equation}
We claim that
\begin{equation}\label{8u8g8fia82}
\sup_{ \R^n} \varphi_*
=\sup_{ \R^n\setminus B_1} \varphi_*\,.
\end{equation}
Once this is proved, from
the fact that~$-2\le\beta\le98$, we deduce
$$ \sup_{ B_1} \varphi\le\sup_{ B_1} \varphi_* 
+\frac{2C\,\gamma_0}{100}\leq
\sup_{ \R^n\setminus B_1} \varphi_* 
+\frac{2C\,\gamma_0}{100}\le
\sup_{ \R^n\setminus B_1} \varphi +\frac{98\,C\,\gamma_0}{100}
+\frac{2C\,\gamma_0}{100}
$$
and conclude~\eqref{DClap}.

To prove claim~\eqref{8u8g8fia82}, it is enough to establish that
\begin{equation}\label{8u8g8fia82-PRE}
\sup_{ \R^n} \varphi_*
=\sup_{ \R^n\setminus B_{\rho}} \varphi_*\,
\end{equation}
for every~$\rho\in(0,1)$, since~$\varphi_*$ is continuous in~$\R^n$.
For this,
we argue by contradiction and
assume that
there exists~$x_*\in\overline{ B}_{\rho}\subset B_1$ such that~$ \varphi_*(x_*)=\sup_{\R^n}\varphi_*$.
This leads to $(-\Delta)^s\varphi_*(x_*)\ge0$ for all~$s\in(0,1)$.
We notice that~$(-\Delta)^s\varphi_*(x_*)\ge0$ when~$s=1$,
in view of~\eqref{CAV}, and also when~$s=0$, since
$$ \varphi_*(x_*)=\sup_{\R^n}\varphi_*\ge\sup_{x\in\R^n\setminus
B_{\sqrt{2}}}\left(\varphi(x)
+ \frac{
C\,\gamma_0}{100}\,\beta(x)\right) \geq0.$$
We therefore obtain that
\begin{equation*}
{\mathcal{L}}_{\mu_j} \varphi(x_*)\ge0\end{equation*}
for all~$j\in\{1,\dots,J\}$.
Finally, recalling
that~$\alpha_j\ge0$ by~\eqref{CONV:POSITIVA}
and the definition~\eqref{L:EQ:OP}
of~$L$,
we deduce
\begin{equation}\label{9574u5iy84uyi45jtk}
L\varphi(x_*)\ge0.\end{equation}

Now, in light of~\eqref{CASO:MU}, we have that, for all~$j\in\{1,\dots,J\}$,
$$ {\mathcal{L}}_{\mu_j}\beta\le-c 
\qquad {\mbox{everywhere in~$B_1$,}}$$
and consequently, recalling the lower bound in~\eqref{MON:F} on~$\alpha_j$,
$$ L\beta \le-c\sum_{j=1}^J\alpha_j\le-c\,\vartheta_0\qquad {\mbox{everywhere in~$B_1$.}}$$
Here~$c$ is the constant in Lemma~\ref{02-3-45}.
Making use of this inequality, \eqref{09p29}, and~\eqref{9574u5iy84uyi45jtk},
and observing that~$x_*\in\overline{B}_{\rho}\subset B_1$,
we conclude
$$ L\varphi(x_*)=L\varphi_*(x_*)- \frac{
C\,\gamma_0}{100}\,L\beta(x_*)\ge \frac{c\,C\,\vartheta_0\,\gamma_0}{100}.
$$
This, combined with~\eqref{98suidvyfzsdaq}, gives that~$
c\,C\,\vartheta_0/{100}\le1$,
which provides a contradiction if~$C$ is taken large enough, depending only on~$n$
and~$\vartheta_0$.

This argument establishes
claim~\eqref{8u8g8fia82-PRE} and finishes the proof.\end{proof}

\section{A unified approach towards derivative estimates}\label{P98765vgyuE1}

Most of our theorems
will follow from our next result, Theorem~\ref{sec:XT}.
It deals with solutions of ``linear operators
possibly varying from point to point and satisfying a maximum principle
with estimate''.
To state our result precisely, we consider a family of
linear operators~$\{L^{(x)}\}_{x\in B_1}$,
of any of the types that we have considered previously in the paper,
and we set the following terminology.

\begin{definition}\label{BENE}
We say that the family~$\{
L^{(x)}\}_{x\in B_1}$ satisfies the maximum principle with
estimate in~$B_1$ if
there exists a constant~$C$
such that,
for every~$\gamma_0\ge0$ and every nonnegative function~$
\varphi\in C_{\rm b}(\R^n)\cap {W^{2,\infty}(B_1)}$,
the following statement holds true:
if 
\begin{equation}\label{MAXPART1} 
\inf_{y\in B_1}\{ L^{(y)}\varphi(x)\}\le\gamma_0
\qquad{{\mbox{for all }}}\;x\in B_1,
\end{equation}
then
\begin{equation}\label{MAXPART}
\sup_{{B_1}}\varphi \le
\sup_{\R^n\setminus{B_1}}\varphi +C\gamma_0.
\end{equation}
\end{definition}

We recall that the computation of~$L^{(y)}\varphi$
with~$\varphi$ only of class~$C_{\rm b}(\R^n)\cap {W^{2,\infty}(B_1)}$ is intended in the light of Remark~\ref{R35}
(in particular, when~$L^{(y)}$ involves classical second order operators,
the incremental quotient setting
in~\eqref{CAV} must be adopted).
This comment concerns~\eqref{MAXPART1} and also the left-hand side
of~\eqref{L-EQAU-B1} below. Note, instead, that
the right-hand side of~\eqref{L-EQAU-B1}
is well defined and finite since the
functions~$\partial^2_eu$, $\partial_eu$, and~$u$
are assumed in the theorem to be
smooth and globally bounded.

We now
provide a result that will 
establish first and one-sided second derivative
bounds in the case of operators possessing a local extension. It will serve both for the Pucci-type operators as well as for our other fully nonlinear equations
defined through the function~$F$. In addition, the same method, properly modified, will be used for equations with no extension property.

We introduce the auxiliary function
\begin{equation}\label{78AJJQoeAAiak0IIA}
\varphi(x) :=\overline\eta^2(x)\,\big(\partial_e^2u(x)\big)^2_+
+ \tau \,{{\eta}}^2(x)\big(\partial_e u(x)\big)^2+
\sigma \,\left(u(x)-\sup_{B_1} u\right)^2,\end{equation}
where~$\overline\eta$ and~$\eta
$ are smooth functions with compact support in~$B_{1/2}$
and~$B_1$, respectively. In addition, we will need that~$\eta=1$
in~$B_{1/2}$ in order to be able to verify
the key inequality~\eqref{L-EQAU-B1} below.
As we will see in the proof of the following theorem, the first
and second derivative bounds~\eqref{LAPRIMAGENE}
and~\eqref{LASECOGENE} will follow from making
two different appropriate choices of the
pair of cutoff functions~$(\overline\eta,{\eta})$.

Let us explain why we are forced,
in this nonlocal theory, to consider
the auxiliary function~\eqref{78AJJQoeAAiak0IIA}
involving both~$\partial_e u$ and~$
\partial_e^2u$ at the same time, as well as two different cutoff functions. This
has not been considered in the local theory. Indeed,
for local equations one works first with~$\eta^2(\partial_e u)^2+\sigma
u^2$ as in~\eqref{FACIL0}, obtaining first derivative
estimates in a ball. One considers next~$\overline\eta^2(\partial_e v)_+^2+\sigma
v^2$ with~$v=\partial_e u$,
as in~\eqref{eqn-v-pos}. This allows us
to control~$\partial_e^2 u$
from above in a smaller ball by the previous control of~$\partial_e u$
in the larger ball. However, since in the nonlocal setting the maximum
principle involves the whole exterior datum
(and not only the boundary datum), this second
choice would require to control the first derivatives~$v=\partial_e u$
in all space, which cannot be achieved for an equation posed
in a ball. We solve this trouble by introducing the new
auxiliary function~\eqref{78AJJQoeAAiak0IIA}.
Note anyway that the simpler choice~\eqref{eqn-v-pos}
would still work if our nonlocal equation were posed in all of~$\R^n$.
This is why we have chosen to refer to this simpler test function
in the Introduction, especially in order to make the Open problem~\ref{OP2BIS}
as simple as possible.

\begin{theorem}\label{sec:XT} Let~$\sigma\ge1$ and $\tau\ge1$.
Assume that~$\{
L^{(x)}\}_{x\in B_1}$ is a family of linear operators
satisfying the maximum principle with estimate in~$B_1$,
with constant~$C$, according to Definition~\ref{BENE}.

Given any~$e\in\R^n$ with~$|e|=1$,
$\overline\eta\in C^\infty_c(B_{1/2})$,
${\eta}\in C^\infty_c(B_1)$
with~$\eta=1$ in~$B_{1/2}$,
and~$u\in C^\infty(\R^n)\cap W^{2,\infty}(\R^n)$,
consider the function~$\varphi$ in~\eqref{78AJJQoeAAiak0IIA}.

Assume that, for every such choice of~$\sigma$, $\tau$, $e$, $\overline\eta$, $\eta$, and~$u$,
we have
\begin{equation} \label{L-EQAU-B1}\begin{split}
L^{(x)}\varphi(x)
\le\,&
2\overline\eta^2(x)\,(\partial_e^2 u(x))_+\,L^{(x)}\partial_e^2 u(x)\\&\quad
+ 2\tau\,{\eta}^2(x)\,\partial_e u(x)\,
L^{(x)}\partial_e u(x)+2\sigma \,\left(u(x)-\sup_{B_1}u\right)\,
L^{(x)} {\left(u-\sup_{B_1}u\right)}(x)
\end{split}\end{equation}
for
every~$x\in B_1$.

We then have,
for every~$u\in C^\infty(\R^n)\cap W^{2,\infty}(\R^n)$,
\begin{equation}\label{LAPRIMAGENE}
\sup_{B_{1/2}} |\partial_e u|\le 
C\,\Big(C a_1+ \big(C a_0\|u\|_{L^\infty(B_1)}\big)^{1/2}+
\|u\|_{L^\infty(\R^n)}
\Big)\end{equation}
and
\begin{equation}\label{LASECOGENE}
\sup_{B_{1/4}} \partial^2_e u\le C\,\Big(
C a_2+C a_1
+\big(C a_0 \| u\|_{L^\infty(B_1)}\big)^{1/2}+ \| u\|_{L^\infty(\R^n)}
\Big),\end{equation}
where
\begin{equation}\label{1NOTA a 012}
a_0:=\sup_{x\in B_1} \left(L^{(x)}{\left(u-\sup_{B_1}u\right)}(x)\right)_- \, ,
\end{equation}
\begin{equation}\label{2NOTA a 012}
a_1:=\sup_{x\in B_1} \big|L^{(x)}\partial_e u( x)\big|,\end{equation} 
\begin{equation}\label{3NOTA a 012}
a_2:=\sup_{x\in B_1}\big(
L^{(x)}\partial_e^2 u( x)\big)_+\,,
\end{equation}
and~$C$ is a constant depending only on~$n$,
$\sigma$, and~$\tau$.
\end{theorem}

Note that the functions~$\varphi$ in~\eqref{78AJJQoeAAiak0IIA}
belong to $C_{\rm b}(\R^n)\cap {W^{2,\infty}(B_1)}$, as required in the previous section,
by basic properties
of the positive part and the square power
appearing
in~$\big(\partial_e^2u\big)^2_+$, and since~$\partial_e^2u\in C^\infty(\R^n)$.

\begin{proof}[Proof of Theorem~\ref{sec:XT}]
We let
\[
\Phi(x):=a_2\,
\overline\eta^2(x)\,(\partial_e^2 u(x))_+
+a_1  {\eta}^2(x)\,|\partial_e u(x)|
+a_0 \| u\|_{L^\infty(B_1)}.\]
We claim that
\begin{equation}\label{9OA-CLAI}
\sup_{{B_1}}\varphi
\le C_\sharp\left(C
\sup_{B_1} \Phi+\| u\|_{L^\infty(\R^n)}^2
\right)\end{equation}
for some constant~$C_\sharp\ge1$ depending only on~$n$,
$\sigma$, and~$\tau$.
To check this, we use~\eqref{L-EQAU-B1} to find that, for all~$x\in B_1$,
\begin{eqnarray*} 
\inf_{y\in B_1}\{ L^{(y)}\varphi(x)\}&\le&
L^{(x)}\varphi(x)
\\&\le&
2\overline\eta^2(x)\,(\partial_e^2 u( x))_+
\,L^{( x)}\partial_e^2 u( x)
+ 2\tau \, {\eta}^2( x)\,\partial_e u( x)\,
L^{( x)} \partial_e u( x)\\&&\qquad+2\sigma \,\left(u( x)-
\sup_{B_1}u\right)L^{( x)} {\left(u-\sup_{B_1}u\right)}(x)\\&\le&
\sup_{B_1}\Big(
2a_2\overline\eta^2\,(\partial_e^2 u)_+
+ 2a_1\,\tau \,{\eta}^2\,|\partial_e u|\Big)
+4a_0\,\sigma \,\|u\|_{L^\infty(B_1)}
\\ &\le&C\sup_{B_1}\Phi,\end{eqnarray*}
for some constant~$C$ depending only on~$n$, $\sigma$, and~$\tau$.
Hence, \eqref{MAXPART1} is satisfied with~$
\gamma_0:=
C\sup_{B_1}\Phi$, and thus~\eqref{MAXPART} yields that~$
\sup_{{B_1}}\varphi \le
\sup_{\R^n\setminus{B_1}}\varphi +C\gamma_0$.
As a consequence, recalling~\eqref{78AJJQoeAAiak0IIA},
we have that~$
\sup_{{B_1}}\varphi \le 4\sigma\,\|u\|_{L^\infty(\R^n)}^2
+C\gamma_0$.
This establishes~\eqref{9OA-CLAI}, as desired.

We stress that the constant~$C_\sharp$ in~\eqref{9OA-CLAI}
does not depend on~$\eta$ and $\overline\eta$, and hence
we can now take appropriate choices for these two functions.
Note first that, from~\eqref{9OA-CLAI},
\begin{equation*}\begin{split}
&\sup_{{B_1}} \Big( \overline\eta^2(\partial_e^2 u)_+^2+
\tau \,\eta^2|\partial_e u|^2\Big)
\\
&\qquad\le\,
\sup_{{B_1}}\varphi\\&\qquad
\le\, C_\sharp\left(C
\sup_{B_1} \Phi+\| u\|_{L^\infty(\R^n)}^2
\right)\\ &\qquad
\le\, C_\sharp\left(C
\Big( \sup_{B_1}\big( a_2\overline\eta^2 \,(\partial_e^2 u)_+
+ a_1{\eta}^2\,|\partial_e u|\big)
+ a_0 \| u\|_{L^\infty(B_1)}\Big)+ \| u\|_{L^\infty(\R^n)}^2
\right)
\\ &\qquad
\le\, C_\sharp\Bigg(\frac1{2C_\sharp}\sup_{B_1}\Big(\overline\eta^2 \,(\partial_e^2 u)_+^2 +
\tau {\eta}^2\,|\partial_e u|^2\Big)
\\&\qquad\qquad+C_\sharp C^2 a_2^2\sup_{B_1}\overline\eta^2
+C_\sharp C^2 a_1^2\sup_{B_1}\eta^2
+C a_0 \| u\|_{L^\infty(B_1)}+ \| u\|_{L^\infty(\R^n)}^2
\Bigg),
\end{split}\end{equation*}
where we have used a
Cauchy-Schwarz inequality in the last step,
and therefore
\begin{eqnarray*}&& \frac12\,
\sup_{{B_1}} \Big( \overline\eta^2(\partial_e^2 u)_+^2+\tau \,
\eta^2|\partial_e u|^2\Big)\\&&\qquad\le
C_\sharp^2\left(
C^2 a_2^2\sup_{B_1} \overline\eta^2
+C^2 a_1^2\sup_{B_1}\eta^2
+C a_0 \| u\|_{L^\infty(B_1)}+ \| u\|_{L^\infty(\R^n)}^2
\right).\end{eqnarray*}
We now make our two choices of cutoff
functions. First, choosing~$\overline\eta:=0$
and~$\eta$ with~$|\eta|\le1$ and~$\eta=1$
in~$B_{1/2}$, from the last inequality
we obtain~\eqref{LAPRIMAGENE}.
Next, with the same choice of~$\eta$, we now choose~$\overline\eta\in C^\infty_c(B_{1/2})$
satisfying~$\overline\eta=1$ in~$B_{1/4}$. In this way, we infer~\eqref{LASECOGENE}.
\end{proof}

\section{Operators having local extension}\label{PARTE1}

In this section, we prove the main results of
Subsections~\ref{sub:prima} and~\ref{SEct:EXTE},
namely
Theorem~\ref{MUCCI}, Theorem~\ref{THM:1}, and
Corollary~\ref{92OBS}.
To this end,
we will exploit the key inequality of Proposition~\ref{Berns-frac}
in a more general version, as given
in the forthcoming
Proposition~\ref{PRO:AZZ}. To state it,
we consider a direction~$e\in\R^n$, with~$|e|=1$,
and two cutoff functions~$\eta$ and~$\overline\eta$ such that
\begin{equation}\label{118}
\eta\in C^\infty_c (B_{1})
\text{ and } \overline{\eta}\in C^\infty_c (B_{1/2})
\text{ with }
{\mbox{$\eta(x)=1$ for all~$x\in B_{1/2}$.}}\end{equation}
Note that this setting covers the situation required in Theorem~\ref{sec:XT}.

Recall the comments after Definition~\ref{BENE}
for the precise meaning of the left-hand side of~\eqref{65GHA:98a124gH} below (when~$s=1$),
and note that the proposition is uniform in~$s$ (in fact, we include the cases~$s=0$
and~$s=1$).

\begin{proposition}\label{PRO:AZZ}
Let~$\eta$ and~$\overline\eta$
be as in~\eqref{118},
$s\in[0,1]$, and~$\kappa
\in\R$. Let~$u\in C^\infty(\R^n)\cap W^{2,\infty}(\R^n)$.
For every~$x\in\R^n$, let
$$ \varphi(x) :=\overline\eta^2(x)\,\big( \partial_e^2u(x)\big)^2_+ 
+ \tau\,\eta^2(x)\big(\partial_e u(x)\big)^2+
\sigma\,\big(u(x)-\kappa\big)^2.$$

Then, there exist positive constants~$\sigma_0$ and~$\tau_0$, depending only on~$n$,
$\|\eta\|_{C^2(\R^n)}$,
and~$\|\overline{{{\eta}}}\|_{C^2(\R^n)}$, such that
\begin{equation}\label{65GHA:98a124gH}
\begin{split}
& (-\Delta)^{s} \varphi \le
2\overline\eta^2\,(\partial_e^2 u)_+\,
(-\Delta)^s \partial_e^2 u+ 2\tau\,{\eta}^2\,\partial_e u\,(-\Delta)^s\partial_e u\\&\qquad\qquad\qquad
+2\sigma \,(u-\kappa)\,(-\Delta)^s {(u-\kappa)}\qquad\qquad\qquad
{\mbox{if~$\tau\ge \tau_0$ and~$\sigma\ge \sigma_0 \tau$,}}\end{split}\end{equation}
everywhere in all of~$\R^n$.
\end{proposition}

The proof of Proposition~\ref{PRO:AZZ} will be given after Corollary~\ref{COR:A:ZZ}
and will rely on the forthcoming auxiliary calculations.

\subsection{Computations in the extended space}\label{sec:uno}

Throughout this section,
we consider~$s\in(0,1)$ and~$a:=1-2s\in(-1,1)$.
Given~$u\in C^\infty(\R^n)\cap W^{2,\infty}(\R^n)$,
we take~$U$ to be the $s$-harmonic extension of~$u$
in~$\R^{n+1}_+:=
\R^n\times(0,+\infty)=\R^n\times\R_+$, i.e.,
the unique
bounded solution of the problem
\begin{equation}\label{EQ:EXT:U}
\left\{
\begin{matrix}
{\rm div}\, \big(y^a \nabla U(x,y)\big)=0 & {\mbox{ for all }}(x,y)\in\R^{n+1}_+,\\
U(x,0)=u(x) & {\mbox{ for all }}x\in\R^n.
\end{matrix}
\right.\end{equation}
Such solution can be found, for instance, by convolving~$u$ against
the fractional Poisson kernel; see~\cite{caffarelli-silvestre}.
As a result, since~$u$ is continuous,
we have that~$U$ is continuous
in~$\overline{\R^{n+1}_+}$.
It is also clear that~$U$ is smooth in~$\R^{n+1}_+$. Note that the solution~$U$
can also be obtained by direct minimization of
an energy functional; see Section~3
in~\cite{MR3165278}. 

The uniqueness of the bounded extension~$U$ follows from Lemma~\ref{0okmokm6}, which is a more
general result than uniqueness that we will need later in this section. See
Corollary~3.5 in~\cite{MR3165278} for an alternative proof of uniqueness.

For notational convenience, we also set 
\begin{equation}\label{defla}
L_a U :=- {\rm div}\, \big(y^a \nabla U\big).
\end{equation}
Of course, no confusion should arise with the
linearized operator~$L$
introduced in~\eqref{L:EQ:OP}. For further reference, we point out that
\begin{equation}\label{AUS:fg}
L_a (VW) = (L_a V)\,W+V\,(L_a W) - 2y^a \nabla V\cdot \nabla W.
\end{equation}
With a slight abuse of notation, we identify the direction~$e\in\R^n$
with the same
direction in~$\R^{n+1}$
when we consider the directional derivatives along~$e$. 
In this setting, we have the
following result.

\begin{lemma}\label{:2BIS}
Let~$\eta$ be as in~\eqref{118} and~$\kappa\in\R$.
For~$(x,y)\in\R^{n+1}_+$, let
$$\Psi_0(x,y):=
\big( U(x,y)-\kappa\big)^2$$
and
$$ \Psi_1(x,y):={{\eta}}^2(x)\,
\big( \partial_e U(x,y)\big)^2.$$

Then,
\begin{equation}
\label{LE:AUS:2} L_a \Psi_0 =- 2y^a \,|\nabla U|^2
\end{equation}
and 
\begin{equation} \label{87t58-0dst}
L_a \Psi_1 \le- y^a\,{{\eta}}^2\,\big| \nabla\partial_e U\big|^2
+C\chi_{B_1}\,y^a\big( \partial_e U\big)^2
\end{equation}
everywhere
in~$\R^{n+1}_+$,
for some constant~$C$
depending only on~$n$ and~$\|\eta\|_{C^2(\R^n)}$.
\end{lemma}

The statement and proof of Lemma~\ref{:2BIS}
concern a
general constant~$\kappa\in\R$. Later it will be convenient
to choose~$\kappa$ to be the supremum of the solution.

\begin{proof}[Proof of Lemma \ref{:2BIS}] By~\eqref{AUS:fg}
we have that~$ L_a \Psi_0=2(L_a U)\,(U-\kappa)-2y^a \,|\nabla U|^2$,
whence~\eqref{LE:AUS:2} plainly follows.

Now, to prove~\eqref{87t58-0dst},
since~$L_a U=0$, and thus~$L_a \partial_eU=0$,
we infer that
\begin{equation} \label{LE:AUS:2:ZZ}
L_a( \partial_e U)^2 = -2y^a \,|\nabla \partial_e U|^2.
\end{equation}
Furthermore, by~\eqref{AUS:fg},
\begin{equation}\label{HG:A1}
L_a {{\eta}}^2 = 2{{\eta}}\, (L_a 
{{\eta}} )- 2y^a |\nabla{{\eta}}|^2 =-
2y^a {{\eta}}\,\Delta_x{{\eta}}- 2y^a |\nabla_x{{\eta}}|^2.
\end{equation}
Now, we notice that~$\Psi_1={{\eta}}^2( \partial_e U)^2 $.
Accordingly, we
use~\eqref{AUS:fg}, \eqref{LE:AUS:2:ZZ}, and~\eqref{HG:A1}
to see that
\begin{eqnarray*}
L_a\Psi_1 &=&(
L_a {{\eta}}^2)\,( \partial_e U)^2 +{{\eta}}^2\,(L_a ( \partial_e U)^2  )-
2y^a \nabla_x {{\eta}}^2\cdot \nabla_x( \partial_e U)^2  \\
&=& -2y^a
({{\eta}}\,\Delta_x{{\eta}}
+|\nabla_x{{\eta}}|^2)\,\big( \partial_e U\big)^2
-2y^a \,{{\eta}}^2\,|\nabla \partial_e U|^2
\\&&\qquad- 8y^a \,{{\eta}}\, \partial_e U\,\nabla_x{{\eta}}
\cdot \nabla_x \partial_e U.
\end{eqnarray*}
Also, we observe that, by the Cauchy-Schwarz inequality,
$$ 8y^a\, \big| {{\eta}}\,\partial_e U\,\nabla_x{{\eta}}
\cdot \nabla_x \partial_e U\big|\le
y^a \, {{\eta}}^2\,|\nabla\partial_e U|^2+
16\, y^a\,|\nabla_x{{\eta}}|^2\, (\partial_e U)^2.$$
Consequently, we obtain that
$$ L_a\Psi_1 \le- y^a \,{{\eta}}^2\, |\nabla\partial_e U|^2
+2y^a\,(
|{{\eta}}\Delta_x{{\eta}}|
+7\,|\nabla_x{{\eta}}|^2)\,\big( \partial_e U\big)^2.$$
Hence, the desired inequality~\eqref{87t58-0dst} follows, 
since~${\eta}$ is supported in~$B_1$,
$|\nabla_x{{\eta}}|\le \|\eta\|_{C^1(\R)}$,
and~$|\Delta_x{{\eta}}|\le n\,\|\eta\|_{C^2(\R)}$.
\end{proof}

We now start estimating the operator~$L_a$
acting on the second derivatives of~$U$.
The following (nonoptimal)
inequality is all what we will need subsequently.

\begin{lemma}\label{LE:AUS:1:ZZ}
Let~$\overline\eta$ be as in~\eqref{118}.
For~$(x,y)\in\R^{n+1}_+$,
let $$\Psi_2(x,y):=\overline\eta^2 (x)\,\big( \partial^2_e U(x,y)\big)^2_+ .$$

Then, there exists a constant~$C$,
depending only on~$n$
and~$\|\overline{{{\eta}}}\|_{C^2(\R^n)}$, such that
\begin{equation}\label{DR:ZZ}
L_a\Psi_2\le
C\,\chi_{B_{1/2}}\,y^a \,(\partial^2_e U)^2
\qquad{\mbox{ everywhere in }}\R^{n+1}_+.
\end{equation}
\end{lemma}

Concerning claim~\eqref{DR:ZZ},
we point out that,
since~$U\in C^\infty(\R^{n+1}_+)$,
we have that~$\partial_e^2U\in C^\infty(\R^{n+1}_+)$
and accordingly
\begin{equation}\label{ksd-we}
\big( \partial^2_e U\big)^2_+\in W^{2,\infty}_{\rm loc}(\R^{n+1}_+).\end{equation}
This is sufficient to compute the left-hand side of~\eqref{DR:ZZ}
everywhere by writing~$L_a$
in nondivergence form and using
the incremental quotient limit definition in~\eqref{CAV}. As an alternative,
one could also interpret~\eqref{DR:ZZ} in  the weak sense,
since~\eqref{ksd-we} yields that~$\big( \partial^2_e U \big)^2_+\in H^1_{\rm loc}(\R^{n+1}_+)$.

\begin{proof}[Proof of Lemma~\ref{LE:AUS:1:ZZ}] Recalling~\eqref{AUS:fg},
we have that
\begin{equation}\label{HG:A0:ZZ}
L_a\Big(\overline\eta^2 \,\big( \partial^2_e U\big)^2\Big)=
L_a \overline\eta^2\,\big( \partial^2_e U\big)^2+
\overline\eta^2\,L_a \big( \partial^2_e U\big)^2-
8y^a \overline\eta\,\partial^2_e U\,
\nabla \overline\eta\cdot\nabla \partial^2_e U
\end{equation}
and, since~$L_a U=0$, 
\begin{equation}\label{HG:A1:11}
L_a\big((\partial^2_e U)^2\big) = 2\partial^2_e U
\, L_a\partial^2_e U
- 2y^a |\nabla \partial^2_e U|^2=- 2y^a |\nabla \partial^2_e U|^2.
\end{equation}
Moreover,
$ L_a {{\overline\eta^2}} 
=-
2y^a {{\overline\eta}}\,\Delta_x{{\overline\eta}}-
2y^a |\nabla_x{{\overline\eta}}|^2$.
By inserting this and~\eqref{HG:A1:11} into~\eqref{HG:A0:ZZ}, we conclude that
\begin{equation}\label{9:AH:ZZ}
\begin{split}
L_a \Big(\overline\eta^2 \,\big( \partial^2_e U\big)^2\Big) \,&=\, -
2y^a \overline\eta\,\Delta_x\overline\eta\,(\partial^2_e U)^2
- 2y^a |\nabla_x\overline\eta|^2\,(\partial^2_e U)^2
- 2y^a \overline\eta^2\,|\nabla \partial^2_e U|^2\\&\qquad\qquad
-8 y^a \overline\eta \,\partial_e^2 U\,
\nabla_x \overline\eta\cdot\nabla_x\partial^2_e U\,.\end{split}\end{equation}
Now we use the
Cauchy-Schwarz inequality, finding that
$$ -8 y^a \overline\eta \,\partial_e^2 U
\nabla_{x} \overline\eta\cdot\nabla_x \partial^2_e U \le
y^a \,\overline\eta^2\, |\nabla \partial^2_e U|^2
+16\, y^a\, |\nabla_x\overline\eta|^2 \,(\partial_e^2 U)^2.$$
Since~$\overline\eta$ is supported in~$B_{1/2}$,
$|\nabla_x\overline\eta|\le \|\overline{{{\eta}}}\|_{C^1(\R)}$,
and~$|\Delta_x\overline\eta|\le n\,\|\overline{{{\eta}}}\|_{C^2(\R)}$,
the latter estimate and~\eqref{9:AH:ZZ} give
\begin{equation}\label{CH964}
L_a \Big(\overline\eta^2 \,\big( \partial^2_e U\big)^2\Big)\le
C\,\chi_{B_{1/2}}\,y^a \,(\partial^2_e U)^2
- y^a \overline\eta^2\,|\nabla \partial^2_e U|^2\le
C\,\chi_{B_{1/2}}\, y^a \,(\partial^2_e U)^2\end{equation}
everywhere in~$\R^{n+1}_+$, for a suitable constant~$C$
depending only on~$n$
and~$\|\overline{{{\eta}}}\|_{C^2(\R^n)}$.

Now, we use this inequality to prove~\eqref{DR:ZZ}.
For this, let~$(x,y)\in\R^{n+1}_+$ and
we observe that if~$\Psi_2(x,y)>0$, then~$\Psi_2=\overline\eta^2 \,\big( \partial^2_e U\big)^2$
in a small neighborhood of~$(x,y)$ and thus~\eqref{DR:ZZ}
follows in this case directly from~\eqref{CH964}.

If instead~$(x,y)\in\R^{n+1}_+$ is such that~$\Psi_2(x,y)=0$,
the second incremental quotient
definition~\eqref{CAV} gives, since~$\Psi_2\ge0$, that~$
-\Delta\Psi_2(x,y)\le0$.
Also, since~$\Psi_2\in C^1_{\rm loc}(\R^{n+1}_+)$ by~\eqref{ksd-we},
the fact that~$\Psi_2\ge0=\Psi_2(x,y)$
leads to~$\nabla\Psi_2(x,y)=0$
and,
as a result,
\begin{eqnarray*}
L_a\Psi_2(x,y)=-y^a\Delta\Psi_2(x,y)-ay^{a-1}\partial_y\Psi_2(x,y)\le0.
\end{eqnarray*}
In particular, also
in this case~$L_a\Psi_2(x,y)\le
C\,\chi_{B_{1/2}}\, y^a \,(\partial^2_e U)^2$. The proof
of~\eqref{DR:ZZ} is thereby complete.
\end{proof}

By combining Lemmata~\ref{:2BIS} and~\ref{LE:AUS:1:ZZ} we obtain:

\begin{corollary}\label{COR:A:ZZ}
Let~$\eta$ and~$\overline\eta$
be as in~\eqref{118} and~$
\sigma$, $\tau$, $\kappa\in\R$.
For~$(x,y)\in\R^{n+1}_+$, let
\begin{equation}\label{DEF:FI}
\Phi(x,y):=\overline\eta^2 (x)\,\big(\partial^2_e U(x,y) \big)_+^2 +
\tau\, {{\eta}}^2(x)\big( \partial_e U(x,y)\big)^2 +
\sigma\, \big(U(x,y)-\kappa\big)^2.\end{equation}

Then, there exist
positive constants~$\sigma_0$ and~$\tau_0$, depending only
on~$n$, $\|\eta\|_{C^2(\R^n)}$,
and~$\|\overline{{{\eta}}}\|_{C^2(\R^n)}$, such that~$L_a \Phi\le0$
everywhere in~$\R^{n+1}_+$
if~$\tau\ge\tau_0$ and~$\sigma\ge \sigma_0 \tau$.
\end{corollary}

\begin{proof} In the notation of
Lemmata~\ref{:2BIS} and~\ref{LE:AUS:1:ZZ}
we have that~$
\Phi= \Psi_2+\tau \Psi_1+\sigma \Psi_0$, and consequently
\begin{equation*}
\begin{split}
&L_a\Phi\;\le\;
C\,\chi_{B_{1/2}}\, y^a \,(\partial^2_e U)^2
-\tau\,  y^a\,{{\eta}}^2\,\big| \nabla\partial_e U\big|^2\\ &\qquad\qquad\qquad
+C\,\tau\,\chi_{B_1}\, y^a\big( \partial_e U\big)^2
-2\sigma\, y^a \,|\nabla U|^2.\end{split}\end{equation*}
Now, we notice that~$ |\partial^2_e U|
\le |\nabla\partial_e U|$ and that, by \eqref{118}, $
\chi_{B_{1/2}}\le {{\eta}}^2$.
Therefore, if~$\tau$ is sufficiently large
as stated in the corollary, we deduce
that
\begin{equation*}
L_a\Phi\;\le\; C\tau\, y^a\big( \partial_e U\big)^2
-2\sigma\, y^a \,|\nabla U|^2.\end{equation*}
Thus, if~$\sigma$ is sufficiently large,
we conclude~$L_a\Phi\le0$,
which proves the desired result.\end{proof}

The previous computations allow us to prove
the main inequality of this section.

\begin{proof}[Proof of Proposition~\ref{PRO:AZZ}] We first
consider the case~$s=0$. We have that \label{0-SOTTPA}
\begin{eqnarray*}
&& \varphi -
2\overline\eta^2\,(\partial_e^2 u)_+^2-
2\tau\,{\eta}^2\,(\partial_e u)^2
-2\sigma \, (u-\kappa)^2\\
&&\qquad\qquad=
-\overline\eta^2\,( \partial_e^2u)^2_+ 
- \tau\, \eta^2 (\partial_e u )^2-
\sigma\, (u-\kappa)^2\le0,
\end{eqnarray*}
which establishes the desired result in this case.

Now, we focus on the case~$s\in(0,1)$. {F}rom this,
the case~$s=1$ can be obtained in the limit\footnote{It is interesting to point out that while
we can
obtain~\eqref{65GHA:98a124gH} with~$s=1$ by a limit argument as~$s\nearrow1$,
when~$s=0$ we needed to perform a direct ---though simple--- computation
and we could not argue by sending~$s\searrow0$. Indeed, the limit as~$s\searrow0$ of the fractional Laplacian
involves a term of the type~$s\int_{\R^n\setminus B_1}u(y)\,|y|^{-n-2s}\,dy$
which goes to zero if~$u$ is compactly supported or decreases
sufficiently fast at infinity, but not in general.
In particular, it is not always true that~$ \lim_{s\searrow0}(-\Delta)^su(x)=u(x)$,
since the left-hand side is invariant if one replaces~$u$ by~$u-\kappa$,
for~$\kappa\in\R$, while the right-hand side is not;
see Proposition~4.4 in~\cite{MR2944369}
for additional details.}
(also, one could proceed
by
direct computations, as in~\cite{CC}).
For~$s\in(0,1)$,
we let~$U$ be the $s$-harmonic extension of~$u$ in~$\R^{n+1}_+$,
and we recall that~$a= 1-2s$. By~\cite{caffarelli-silvestre}, up to a positive multiplicative constant depending
only on~$s$ ---that we do not write since it plays no role to establish~\eqref{65GHA:98a124gH}--- we have
\begin{equation}\label{A1.13p:ZZ}
-\lim_{y\searrow0} y^{a} \partial_y \partial_e U
= (-\Delta)^{s} \partial_e u
\end{equation}
and 
\begin{equation}\label{B1.13p:ZZ}
-\lim_{y\searrow0} y^{a} \partial_y \partial^2_{e} U=
(-\Delta)^{s} \partial_{e}^2 u.\end{equation}
Moreover, if~$\Phi^*$ is
the $s$-harmonic extension of~$\varphi$ in~$\R^{n+1}_+$,
we also have that
\begin{equation}\label{h67io:ZZ}
-\lim_{y\searrow0} y^a \partial_y \Phi^*
= (-\Delta)^{s} \varphi\end{equation}
(recall that~$\varphi$ is locally~$W^{2,\infty}$ and bounded in~$\R^n$).

The key point is to consider the function
$$ \Phi(x,y):=
\overline\eta^2(x)\,\big(\partial_e^2 U(x,y)\big)^2_+
+ \tau\,\eta^2(x)\big(\partial_e U(x,y)\big)^2+
\sigma\,\big(U(x,y)-\kappa\big)^2.$$
Now, by Corollary~\ref{COR:A:ZZ}, we know that~$
L_{a} \Phi\le0$ almost everywhere in~$\R^{n+1}_+$,
as long as~$\sigma$ and~$\tau$ are taken
as in the statement
of Proposition~\ref{PRO:AZZ}.
Thus, since~$L_{a}\Phi^*=0$,
\begin{equation*}
L_{a} (\Phi-\Phi^*)\le0\quad{\mbox{ everywhere in }}\R^{n+1}_+.\end{equation*}
Notice also that~$\Phi-\Phi^*$ 
vanishes in~$\R^n \times\{0\}$,
and hence the maximum principle in Lemma~\ref{0okmokm6}
gives that~$\Phi-\Phi^*\le0$ in~$\R^{n+1}_+$.
As a consequence, since this function vanishes in~$\R^n\times\{0\}$,
we find that~$\lim_{y\searrow0} y^{a}\, \partial_y (\Phi-\Phi^*)\le0$.
Therefore, by~\eqref{h67io:ZZ},
the definition of~$\Phi$, and recalling again that~$(\partial^2_e U)_+^2$ is a~$W^{2,\infty}_{\rm loc}$ function,
\begin{equation*}
\begin{split}
 (-\Delta)^{s} \varphi \;&\le\;-
\lim_{y\searrow0} y^{a} \partial_y \Phi\\
&=\;-\lim_{y\searrow0}\Big( 
2y^{a}\,\overline\eta^2\,(\partial_e^2 U)_+\,
\partial_y\partial_e^2 U
+ 2y^{a}\,\tau\, \eta^2\,\partial_e U\,\partial_y\partial_e U
\\&\qquad\qquad+2y^{a}\,\sigma\, (U-\kappa)\,\partial_y 
(U-\kappa)\Big).
\end{split}\end{equation*}

We conclude, since~$\partial_e U(x,0)=\partial_e u(x)$ and~$\partial_e^2U(x,0)=
\partial_e^2u(x)$, that
\begin{equation}\label{g68:BIS:ZZ}\begin{split}
(-\Delta)^{s} \varphi\; \le\;&-
\lim_{y\searrow0}\Big(
2y^{a}\,\overline\eta^2\,(\partial_e^2 u)_+\,
\partial_y \partial_e^2 U
+ 2y^{a}\,\tau\, {\eta}^2\,\partial_e u\,\partial_y \partial_e U
\\&\qquad+2y^{a}\,\sigma\, (u-\kappa)\,\partial_y (U-\kappa)\Big)
\\ =\;&2\overline\eta^2\,(\partial_e^2 u)_+
(-\Delta)^s \partial_e^2 u
+ 2\tau\, \eta^2\,\partial_e u\,(-\Delta)^s\partial_e u
\\&\qquad
+2\sigma\, (u-\kappa)\,(-\Delta)^s {(u-\kappa)},\end{split}\end{equation}
where~\eqref{A1.13p:ZZ} and~\eqref{B1.13p:ZZ}
were used in the last step.
\end{proof}

The following result is a simple, but
useful, improvement of Proposition~\ref{PRO:AZZ},
in which we obtain that~$\varphi$ is a suitable subsolution
with respect to the linearized operator~\eqref{L:EQ:OP}
(in our current setting, the linear operators~$
{\mathcal{L}}_1,\dots,{\mathcal{L}}_J$
in~\eqref{L:EQ:OP} boil down to the linear operators~$
{\mathcal{L}}_{\mu_{1}},\dots,{\mathcal{L}}_{\mu_J}$).

\begin{corollary}\label{CRT:AZZ}
Let~$\eta$ and~$\overline\eta$
be as in~\eqref{118}.
Let~$\kappa\in\R$ and~$u\in C^\infty(\R^n)\cap W^{2,\infty}(\R^n)$.
For every~$x\in\R^n$, let
\begin{equation}\label{9203-139pq} \varphi(x) :=\overline\eta^2(x)\,\big(
\partial_e^2u(x)\big)^2_+
+ \tau\, \eta^2(x)\big(\partial_e u(x)\big)^2+
\sigma\, \big(u(x)-\kappa\big)^2.\end{equation}

Then, there exist positive constants~$\sigma_0$ and~$\tau_0$, depending only on~$n$,
$\|\eta\|_{C^2(\R^n)}$,
and~$\|\overline{{{\eta}}}\|_{C^2(\R^n)}$, such that,
if~$\tau\ge \tau_0$ and~$\sigma\ge \sigma_0 \tau$, then
in all of~$\R^n$ we have
\begin{equation}\label{9203-139pq-BISB} L\varphi \le
2\overline\eta^2\,(\partial_e^2 u)_+
\,
L\,\partial_e^2 u 
+ 2\tau\, {\eta}^2\,\partial_e u\,
L\,\partial_e u+2\sigma\, (u-\kappa)\,L{(u-\kappa)}.\end{equation}
\end{corollary}

\begin{proof} 
We write~\eqref{65GHA:98a124gH}
for every~$s\in[0,1]$,
and we integrate with respect to the measure~$\mu_j$,
for every given~$j\in\{1,\dots,J\}$.
We find that
\begin{equation*}
{\mathcal{L}}_{\mu_j} \varphi\le
2\overline\eta^2\,(\partial_e^2 u)_+
\,{\mathcal{L}}_{\mu_j}(\partial_e^2 u)
+ 2\tau\, \eta^2\,\partial_e u\,{\mathcal{L}}_{\mu_j}
\partial_e u
+2\sigma\, (u-\kappa)\,{\mathcal{L}}_{\mu_j} {(u-\kappa)}
.\end{equation*}
Now, we
multiply the above expression by~$\alpha_j$, as defined
in~\eqref{defalphak}, using~\eqref{CONV:POSITIVA},
and then we sum the inequality over~$j\in\{1,\dots,J\}$. In this way,
recalling also the definition of~$L$ in~\eqref{L:EQ:OP}, we conclude the proof.
\end{proof}

We remark that Proposition~\ref{Berns-frac}
is the simplified version of Proposition~\ref{PRO:AZZ} 
which already leads to first derivative estimates. On the other hand,
Proposition~\ref{Berns-frac} is stated for a more general cutoff function than the one in
Proposition~\ref{PRO:AZZ}. The proof of Proposition~\ref{Berns-frac} follows the same lines
as that of Proposition~\ref{PRO:AZZ}, and only requires Lemma~\ref{:2BIS}.

Though not explicitly used in this article, we next state the simplest inequality for an auxiliary function
which leads to one-sided derivative estimates (at least for global solutions). For more details see
the comments before and after~\eqref{eqn-v-pos}, in Open problem~\ref{OP2BIS},
and before Theorem~\ref{sec:XT}. Its proof also follows the same lines as that
of Proposition~\ref{PRO:AZZ}.

\begin{proposition}\label{PROP-66glo}
Let~$s\in(0,1)$,
$v\in C^\infty(\R^n)\cap W^{1,\infty}(\R^n)$, $\eta\in C^\infty(\R^n)\cap W^{2,\infty}(\R^n)$,
$e\in \R^n$ with~$|e|=1$, and~$\sigma>0$.
Let
$$ \psi:=\eta^2 \big(
\partial_e v\big)^2_{+} +\sigma v^2.$$

Then, in all of~$\R^n$ we have
\begin{equation*}
(-\Delta)^s \psi\le 2\eta^2\,(\partial_e v)_{+}\,
(-\Delta)^s\partial_e v
+2\sigma v\,(-\Delta)^s v\quad{\mbox{ if }}\sigma\ge\sigma_0,
\end{equation*}
for some constant~$\sigma_0$ depending only on~$n$ and~$\|\eta\|_{C^2(\R^n)}$ ---and,
in particular, independent of~$s$.
\end{proposition}

\subsection{Proofs of Theorem~\ref{MUCCI}, Theorem~\ref{THM:1}, and
Corollary~\ref{92OBS}}\label{MUCCI:S}

We can now prove these results using Theorem~\ref{sec:XT}, via an appropriate choice
of the linear operators~$\{L^{(x)}\}_{x\in B_1}$.

However, let us first make a remark that connects Theorems~\ref{MUCCI} and \ref{THM:1}. 

\begin{remark}\label{affine-indefinite}{\rm 
Theorem~\ref{THM:1} can be easily extended by changing
the definition~\eqref{OP:L} of ${\mathcal{L}}_\mu$
to involve more general operators of the form
\begin{equation}\label{STxWi} \int_0^1
\left( -\sum_{i,j=1}^n M_{ij}(s)\, \partial_{x_i x_j}^2 u\right)^s\,d\mu(s),
\end{equation}
containing fractions of second order elliptic operators
with constant coefficients instead of fractions of the Laplacian. Here, for every~$s\in[0,1]$ we are given a matrix~$M(s)\in {\mathcal{A}}_{\Lambda^{-2},\lambda^{-2}}$, where the classes $ {\mathcal{A}}$ were defined in \eqref{ELLMA}.

In this way one includes here those equations of Subsection~\ref{sub:prima} built from a finite number $J$ of linear operators. In fact, since our estimates will be independent of $J$, it is possible to
deduce Theorem~\ref{MUCCI} in that subsection from this extension of Theorem~\ref{THM:1},
by an approximation and limiting argument.

The proofs would remain almost unchanged by performing, for every $s\in [0,1]$,
the change of variables $\overline x=Ax$ (and then consider the function $\overline u (\overline x)= u(A^{-1}\overline x)$), where $A=A(s)\in
{\mathcal{A}}_{\lambda,\Lambda}$ is such that $A^{-2}=M:=M(s )$ ---as done later in~\eqref{TRANSFO} within the proof of Theorem~\ref{MUCCI}. In this way, one can see that
$$
\left( -\sum_{i,j=1}^n M_{ij}(s)\, \partial_{x_i x_j}^2 u\right)^s (x) = 
c_{n,s}\, \det A(s) \int_{\R^n}\frac{u(x)-u(y)}{|A(s)\,(x-y)|^{n+2s}} dy,
$$
which coincides with the operator ${\mathcal{L}}_{A}$ previously considered in \eqref{TUTTIS}, up to a multiplicative constant. These fractional operators thus possess the same type of
extension properties as the fractional Laplacian.

The previous equality can be alternatively checked through Fourier symbols since
\begin{eqnarray*}&&
\left(\sum_{i,j=1}^n M_{ij}\xi_i\xi_j \right)^s=
\left(\sum_{i,j,k=1}^n (A^{-1})_{ik}(A^{-1})_{kj}\, \xi_i\xi_j \right)^s=|\overline{\xi}|^{2s}
=\frac{c_{n,s}}{2}\int_{\R^n} \frac{2-e^{i\overline{z}\cdot\overline{\xi}}-e^{-i\overline{z}\cdot\overline{\xi}}}{|\overline{z}|^{n+2s}}\,d\overline{z}\\
&&\qquad 
=\frac{c_{n,s}}{2}\int_{\R^n} \frac{2-e^{i(A^{-1}\overline{z})\cdot\xi}-e^{-i(A^{-1}\overline{z})\cdot\xi}}{|\overline{z}|^{n+2s}}\,d\overline{z}=
\frac{c_{n,s}\,\det A}{2}\,
\int_{\R^n} \frac{2-e^{iz\cdot\xi}-e^{-iz\cdot\xi}}{|Az|^{n+2s}}\,dz,
\end{eqnarray*}
where $\overline\xi:=A^{-1}\xi$.
}
\end{remark}

\begin{proof}[Proof of Theorem~\ref{MUCCI}]
Note that,
since~${\mathcal{A}}$
is compact and we assume the continuity hypothesis~\eqref{Gcon}, observing also that~${\mathcal{L}}_Au(x)$
is continuous with respect to~$A$ by the dominated convergence theorem,
we have that, given~$x\in B_1$,
${\mathcal{M}}_{{\mathcal{A}}} u(x)={\mathcal{L}}_{A_x}u(x)-g_{A_x}(x)$
for some~$A_x\in{\mathcal{A}}$.

We start verifying the hypotheses of Theorem~\ref{sec:XT}
for the family of operators~$L^{(x)}={\mathcal{L}}_{A_x}$.
We first check that
the maximum principle
with estimate, as considered in Definition~\ref{BENE},
is satisfied in this case. Indeed,
assume that~\eqref{MAXPART1} holds true. By~\eqref{CASO:K}
applied to the kernel~$K(z)=c_{n,s}|Az|^{-n-2s}$, we know that
there exists
a function~$\beta$ with~$|\beta|\le100$ in~$\R^n$
and~${\mathcal{L}}_A\beta\le-c$
everywhere
in~$B_1$
for every~$A\in{\mathcal{A}}$.
Here~$c>0$ depends only on~$n$, $\lambda$, and~$\Lambda$.

Define~$\widetilde\varphi:=\varphi+\gamma_0\beta/c$.
By~\eqref{MAXPART1}, we have that~$
\inf_{y\in B_1}\{ {\mathcal{L}}_{A_y}\widetilde\varphi(x)\}\le 0$
for every~$x\in B_1$.
{F}rom this (and again the fact that the previous infimum will be
equal to~${\mathcal{L}}_{A^x}\widetilde\varphi(x)$ for some~$A^x\in{\mathcal{A}}$ ---$A^x$
perhaps different from~$A_x$), one sees that~$\sup_{\R^n}\widetilde\varphi$
cannot be achieved in~$B_1$ (except when~$\widetilde\varphi$ is constant),
and therefore in any case
\begin{equation*}
\max_{\overline{B}_1}\widetilde\varphi\le \sup_{\R^n\setminus B_1}\varphi+\frac{100\,
\gamma_0}c.
\end{equation*}
Consequently, we find that
$$ \sup_{B_1}\varphi\le \sup_{\R^n\setminus B_1}\varphi+\frac{200\,
\gamma_0}c,$$
and hence the maximum principle with estimate, as considered in Definition~\ref{BENE},
is satisfied.

Now we check that~\eqref{L-EQAU-B1} holds true in the current situation
if we take~$\tau$ and~$\sigma$ appropriately.
Indeed, 
by the definition~\eqref{TUTTIS} of~$\mathcal{L}_A$, we know that
\begin{equation}\label{TRANSFO} {\mathcal{L}}_A u(x)=\frac1{\det A}(-\Delta)^s
u_A(Ax),\qquad{\mbox{
where~$u_A( \overline{x}):=u(A^{-1}\overline{x})$.}}\end{equation}
Also~$\partial_{e_A} u_A(\overline{x})=\partial_{e} u(A^{-1}\overline{x})$
and~$\partial_{e_A}^2 u_A(\overline{x})=\partial_{e}^2 u(A^{-1}\overline{x})$, with~$e_A:=Ae$.
That is, setting~$e_A':=e_A/|e_A|$, we have that~$ \partial_{e} u(A^{-1}\overline{x})=
|e_A|\,\partial_{e_A'} u_A(\overline{x})$ and~$
\partial_{e}^2 u(A^{-1}\overline{x})=|e_A|^2\,
\partial_{e_A'}^2 u_A(\overline{x})$.
Thus, we have that
\begin{equation}\label{TRANS-90sccjer}
(-\Delta)^s \partial_{e'_A} u_A(\overline{x})=
\frac{\det A}{|e_A|}\,{\mathcal{L}}_A\partial_{e} u(x)\quad{\mbox{
and }}\quad
(-\Delta)^s \partial_{e'_A}^2u_A(\overline{x})=\frac{\det A}{|e_A|^2}\,{\mathcal{L}}_A\partial_{e}^2 u(x).\end{equation}
At the same time, if we let~$\varphi$ be as in~\eqref{78AJJQoeAAiak0IIA}, $\sigma_A:=|e_A|^{-4}\sigma$,
$\tau_A:=|e_A|^{-2}\tau$, $\eta_A(\overline{x}):=\eta(A^{-1}\overline{x})$,
$\overline\eta_A(\overline{x}):=
\overline\eta(A^{-1}\overline{x})$,
and
$$ \overline\varphi_A(\overline{x}):=
\overline\eta^2_A(\overline{x})\,\big(\partial_{e'_A}^2u_A(\overline{x})\big)^2_+
+\tau_A \,\eta^2_A(\overline{x})
\big(\partial_{e'_A} u_A(\overline{x})\big)^2+\sigma_A\,
\left( u_A(\overline{x})-\sup_{A(B_1)} u_A\right)^2,$$
we find that, if~$x\in B_1$, and thus~$\overline{x}=Ax\in A(B_1)$,
\begin{align*}
{\mathcal{L}}_A\varphi(x)
=\frac{|e_A|^4}{\det A}\,(-\Delta)^s\overline\varphi_A(\overline{x}).
\end{align*}
As a result, 
in view of\footnote{Strictly speaking, we are applying here a small variation
of Proposition~\ref{PRO:AZZ},
since the condition~\eqref{118}
on the cutoff functions reads here~$
\eta_A$, $\overline{\eta}_A\in C^\infty_c (A(B_{1/2}))$ and~$\eta_A=1$
in~$A( B_{1/4})$. It is easy to verify that
this fact does not alter the conclusion~\eqref{0o9o08ddf9}.
Notice also that~$\tau_A$ and~$\sigma_A$ are comparable
to~$\tau$ and~$\sigma$, respectively, with constants depending
only on~$n$, $\lambda$, and~$\Lambda$.}
Proposition~\ref{PRO:AZZ},
if~$\tau$ and~$\sigma$ are sufficiently large, we conclude that,
for every~$x\in B_1$
(and thus~$\overline{x}=Ax\in A(B_1)$),
\begin{align}\label{0o9o08ddf9}
{\mathcal{L}}_A\varphi(x)\,&\le
\frac{|e_A|^4}{\det A}\,\Bigg\{
2\overline\eta^2_A(\overline{x})\,
\big(\partial_{e'_A}^2u_A(\overline{x})\big)_+\,
(-\Delta)^s \partial_{e'_A}^2u_A(\overline{x})
\\&\qquad+ 2\tau_A\,\eta^2_A(\overline{x})\,\partial_{e'_A} u_A
(\overline{x})\,(-\Delta)^s\partial_{e'_A} u_A(\overline{x})\\&\qquad
+2\sigma_A \,
\left( u_A(\overline{x})-\sup_{A(B_1)} u_A\right)
\,(-\Delta)^s \left( u_A-\sup_{A(B_1)} u_A\right)(\overline{x})
\Bigg\}.
\end{align}
{F}rom
this and~\eqref{TRANS-90sccjer}, we get
\begin{align*}
{\mathcal{L}}_A\varphi(x)\,&\le
\Bigg\{
2\overline\eta^2_A(\overline{x})\,
\big(\partial_{e}^2u(x)\big)_+\,
{\mathcal{L}}_A \partial_{e}^2u(x)
+  2\tau\,\eta^2_A(\overline{x})\,
\partial_{e} u(x)\,{\mathcal{L}}_A\,\partial_{e} u(x)\\&\qquad
+ 2\sigma \,
\left( u(x)-\sup_{B_1} u\right)
\,{\mathcal{L}}_A \left( u-\sup_{B_1} u\right)(x)
\Bigg\}
\end{align*}
for all~$x\in B_1$.
Hence, inequality~\eqref{L-EQAU-B1} holds true with~$
L^{(x)}={\mathcal{L}}_{A_x}$, after choosing, given~$x$, $A=A_x$.

With this, we are in the position of applying Theorem~\ref{sec:XT}. To this end, we estimate the quantities~$a_0$,
$a_1$, and~$a_2$ given in its statement. The arguments here are similar to those of Subsection~\ref{sub41s}.
First, by equation~\eqref{MUCCI:EQ},
\begin{equation}\label{3.37BIS}
{\mathcal{L}}_{A_x} u(x)=g_{A_x}(x) 
\end{equation}
for every~$x\in B_1$.
Also, using again equation~\eqref{MUCCI:EQ},
for every~$x$, $y\in B_1$ we have that
\begin{equation}\label{X560199339999:1}\begin{split}&
{\mathcal{L}}_{A_x}u(y)-g_{A_x}(y)\le
{\mathcal{M}}_{{\mathcal{A}}} u(y)=0={\mathcal{M}}_{{\mathcal{A}}} u(x) 
={\mathcal{L}}_{A_x}u(x)-g_{A_x}(x)
.\end{split}
\end{equation}
Therefore, if~$e\in\R^n$ and~$|e|=1$, given~$h\in(0, 1-|x|)$,
we see that
\begin{eqnarray*}
{\mathcal{L}}_{A_x} \left(\frac{u(x\pm he)-u(x)}{h}\right)=
\frac{{\mathcal{L}}_{A_x}u(x\pm he)-{\mathcal{L}}_{A_x}u(x)}{h}\le\frac{
g_{A_x}(x\pm he)-g_{A_x}(x)
}h,
\end{eqnarray*}
and, as a consequence,~$
\pm{\mathcal{L}}_{A_x} \partial_e u(x)\le\pm
\partial_e g_{A_x}(x)$
for almost every~$x\in B_1$.
Thanks to the possible sign choice in this inequality, 
and to the fact that~$g_A$ is a Lipschitz function, and thus pointwise differentiable almost everywhere,
we thereby deduce that
\begin{equation}\label{X560199339999:2}
{\mathcal{L}}_{A_x} \partial_e u(x)= 
\partial_e  g_{A_x} (x)
\end{equation}
for almost every~$x\in B_1$.
In addition, exploiting~\eqref{X560199339999:1} once again,
\begin{equation*} \begin{split}&
{\mathcal{L}}_{A_x} \left(\frac{u(x+ he)+u(x- he)-2u(x)}{h^2}\right)\\&\qquad=
\frac{{\mathcal{L}}_{A_x}u(x+ he)+{\mathcal{L}}_{A_x}u(x- he)-
2{\mathcal{L}}_{A_x}u(x)}{h^2}\\
&\qquad\le\frac{
g_{A_x}(x+he)+g_{A_x}(x-he)-2g_{A_x}(x)
}{h^2} 
.\end{split}\end{equation*}
Thus, if~$ g_{A_x}$ is semiconcave, and hence pointwise twice
differentiable almost everywhere, we have that,
for almost every~$x\in B_1$,
\begin{equation}\label{DESDF23} {\mathcal{L}}_{A_x}\partial_e^2 u (x)\le\partial_e^2 g_{A_x} (x) 
.\end{equation}

In the notation of Theorem~\ref{sec:XT}
(with~$L^{(x)}$
there corresponding to~${\mathcal{L}}_{A_x}$ here),
\eqref{3.37BIS}, \eqref{X560199339999:2}, and~\eqref{DESDF23}
lead to
\begin{equation}\label{IN939ik3me8834yg}
a_0=\sup_{x\in B_1}\big(g_{A_x}(x)\big)_-
\le \sup_{{x\in B_1}\atop{A\in{\mathcal{A}}}} \big(
 g_{A}(x)
\big)_-
=\| (g_A )_-\|_{L^\infty({\mathcal{A}}\times B_1)}
,\end{equation}
\begin{equation}\label{IN939ik3me8834yg2}
a_1=
\sup_{x\in B_1}\big|\partial_e g_{A_x}(x)\big|
\le\sup_{{x\in B_1}\atop{A\in{\mathcal{A}}}}\big|
\partial_e g_{A}(x)
\big|=
\|\partial_e g_{A}  \|_{L^\infty({\mathcal{A}}\times B_1)},
\end{equation}
and 
\begin{equation}\label{IN939ik3me8834yg3}
a_2\le\sup_{x\in B_1}\big(\partial_e^2 g_{A_x}(x)
\big)_+
\le\sup_{{x\in B_1}\atop{A\in{\mathcal{A}}}}\big(
\partial_e^2 g_{A}(x)\big)_+=
\| (\partial_e^2 g_{A} )_+ \|_{L^\infty({\mathcal{A}}\times B_1)}
\,.
\end{equation}

These considerations and Theorem~\ref{sec:XT} lead to
\begin{equation}\label{NBCshnr:1}
\sup_{B_{1/2}} |\partial_e u|
\le C\,\Big( \|\partial_e g_{A} \|_{L^\infty({\mathcal{A}}\times B_1)}
+\big(
\| ( g_{A} )_-\|_{L^\infty({\mathcal{A}}\times B_1)}
\,\| u\|_{L^\infty(B_1)}
\big)^{1/2}+\| u\|_{L^\infty(\R^n)}
\Big)\end{equation}
and
\begin{equation}\label{NBCshnr:2}\begin{split}
\sup_{B_{1/4}} \partial^2_e u\le\,& C\,
\Big(
\| (\partial_e^2 g_{A} )_+ \|_{L^\infty({\mathcal{A}}\times B_1)}
+\|\partial_e g_{A}\|_{L^\infty({\mathcal{A}}\times B_1)}\\&\qquad+\big(
\| ( g_{A})_-\|_{L^\infty({\mathcal{A}}\times B_1)}
\,\| u\|_{L^\infty(B_1)}
\big)^{1/2}+\| u\|_{L^\infty(\R^n)}\Big)
,\end{split}\end{equation}
for some constant~$C$ depending only on~$n$, $\lambda$, and $\Lambda$.
This gives~\eqref{9ixsjXVaBAM-1},
and also~\eqref{9ixsjXVaBAM-2}
with~$B_{1/2}$ replaced by~$B_{1/4}$. {F}rom this result,
it is easy to deduce the second derivative bound
as stated in Theorem~\ref{MUCCI} (that is, in~$B_{1/2}$)
by a covering argument. \end{proof}

\begin{proof}[Proof of Theorem~\ref{THM:1}]
We use again
Theorem~\ref{sec:XT}, but now
taking, for all~$x\in B_1$, $L^{(x)}:=L$ to be
the linearized operator
introduced in~\eqref{L:EQ:OP}. Let us check that
the hypotheses of Theorem~\ref{sec:XT} are fulfilled.
First of all, by
Proposition~\ref{ujMAX},
if~$\varphi\in C_{\rm b}(\R^n)\cap {W^{2,\infty}(B_1)}$ is nonnegative
and satisfies~$L\varphi\le\gamma_0$ in~$B_1$, for some~$\gamma_0\ge0$,
then~$\sup_{{B_1}}\varphi \le
\sup_{\R^n\setminus{B_1}}\varphi +C\,\gamma_0$,
with~$C$ 
depending only on $n$, $\vartheta_0$, and~$\Theta_0$.
This says that~$L$ satisfies the maximum principle in~$B_1$ with constant~$C$, according to Definition~\ref{BENE}.
Moreover, choosing~$
\kappa:=\sup_{B_1}u$
in Corollary~\ref{CRT:AZZ}, we deduce that~\eqref{L-EQAU-B1} holds for~$\sigma$ and~$\tau$
large enough depending only\footnote{Here we use that the proof
of Theorem~\ref{sec:XT} requires only two choices of pairs~$(\overline\eta,\eta)$
for cutoff functions.}
on~$n$.

Therefore, from
Theorem~\ref{sec:XT} we obtain
that
\begin{equation}\label{7hb2qujjtui501-1}
\sup_{B_{1/2}} |\partial_e u|\le 
C\,\Big(a_1+ \big(a_0\|u\|_{L^\infty(B_1)}\big)^{1/2}+
\|u\|_{L^\infty(\R^n)}
\Big)\end{equation}
and
\begin{equation}\label{7hb2qujjtui501-2}
\sup_{B_{1/4}} \partial^2_e u\le C\,\Big(
a_2+ a_1
+\big( a_0 \| u\|_{L^\infty(B_1)}\big)^{1/2}+\| u\|_{L^\infty(\R^n)}
\Big),\end{equation}
where~$C$ is a constant depending only on~$n$, $\theta_0$, and~$\Theta_0$.

The quantities~$a_0$, $a_1$, and~$a_2$ in the statement of Theorem~\ref{sec:XT} are given by
\begin{equation*}
a_0:=\sup_{x\in B_1} \Big(L \big(u-\sup_{B_1}u\big)( x)\Big)_- \, ,\quad
a_1:=\sup_{x\in B_1} \big|L\,\partial_e u( x)\big|,\quad
{\mbox{and}}\quad
a_2:=\sup_{x\in B_1}\big(
L\,\partial_e^2 u( x)\big)_+\,.
\end{equation*}

We now use
Lemma~\ref{SUBSOLU:EXTE}. First, 
taking into account that~$L$ could contain
the operator~$(-\Delta)^0={\rm Id}$,
we have that
\begin{eqnarray*}&&
\Big(L\big(u-\sup_{B_1}u\big)\Big)_-=\max\Big\{-Lu+L\sup_{B_1}u,0\Big\}\\&&\qquad\qquad
\le
|F(0)|+\Theta_0\sup_{j\in\{1,\dots,J\}}\|g_j\|_{L^\infty(B_1)}}+\| u\|_{L^\infty(B_1)
\end{eqnarray*}
everywhere
in~$B_1$.
Hence, we conclude that
$$ a_0\le\,|F(0)|+\Theta_0\sup_{j\in\{1,\dots,J\}}\|g_j\|_{L^\infty(B_1)}+ \| u\|_{L^\infty(B_1)}.$$
This and the first derivative bound~\eqref{7hb2qujjtui501-1} lead to (using again Lemma~\ref{SUBSOLU:EXTE})
\begin{equation}\label{7hb2qujjtui501-1BOS}
\sup_{B_{1/2}} |\partial_e u|\,\le\, 
C\,\Big(
\| u\|_{L^\infty(\R^n)}
+|F(0)|+\sup_{j\in\{1,\dots,J\}}\|g_j\|_{W^{1,\infty}(B_1)} 
\Big). \end{equation}
Similarly, in light of the second derivative bound~\eqref{7hb2qujjtui501-2},
we see that
\begin{equation}\label{7hb2qujjtui501-2BOS}\begin{split}&
\sup_{B_{1/4}} \partial^2_e u\,\le\, C\,\Big( \| u\|_{L^\infty(\R^n)}+|F(0)|
+\sup_{j\in\{1,\dots,J\}}\|g_j\|_{W^{2,\infty}(B_1)}
\Big).\end{split}\end{equation}

These observations complete
the proof of
Theorem~\ref{THM:1}, after
using a covering argument as at the end of the proof of
Theorem~\ref{MUCCI}.
\end{proof}

Corollary~\ref{92OBS} follows immediately from Theorem~\ref{THM:1}, as we show
next. Alternatively, the corollary could also be established
following the lines of the proof of Theorem~\ref{MUCCI} ---hence, using the~$\max$
structure of the operator instead of introducing the linearized operator~$L$
as in Theorem~\ref{THM:1}.

\begin{proof}[Proof of Corollary~\ref{92OBS}]
We are in the setting of Theorem~\ref{THM:1}, with~$J=2$,
$\mathcal{L}_{\mu_1}=(-\Delta)^s$, $g_1=f$,
$\mathcal{L}_{\mu_2}=(-\Delta)^0$ (the identity), $g_2=f+\phi$, and~$F$
being the~$\max$ operator ---which satisfies~\eqref{CONV:POSITIVA}-\eqref{CONV:F}, with $\theta_0=\Theta_0=1$. Thus, 
the corollary follows immediately from Theorem~\ref{THM:1}.
\end{proof}

\begin{remark}\label{Rk_on_postive_parts}
{\rm
In view of the previous proofs, the estimates of Theorem~\ref{MUCCI}, Theorem~\ref{THM:1}, and
Corollary~\ref{92OBS} can be stated in a more refined manner.
Such statements can be found in version~2 of the arXiv preprint of this work
\url{https://arxiv.org/pdf/2010.00376v2.pdf} 
where the estimates are more precise in several respects: some quantities on the functions $g_A$ (or $g_j$)
involve only their positive or negative parts; the supremums of $g_A$ (or $g_j$) and of its derivatives are considered separately; finally, the equations and estimates are considered in a ball of arbitrary radius $R$, which is particularly interesting in the case of operators of indefinite order (since, for them, the way in which the equation changes after scaling is not as simple as in the definite order case).
}
\end{remark}

\section{General integro-differential operators: proof of Theorem~\ref{PUCCI TYPE BIS}}\label{934itgfoeojtthewi34848488484BOS}

Here, we exploit Theorem~\ref{CONTO TRACCIA}
to establish Theorem~\ref{PUCCI TYPE BIS}. Because of the presence of an error term in the inequality of  Theorem~\ref{CONTO TRACCIA}, the gradient estimates initially obtained will have a remainder or error term. But the fact that the estimates hold at any scale will allow us to reabsorb the error term and get rid of it.

\begin{proof}[Proof of Theorem~\ref{PUCCI TYPE BIS}]
The proof follows closely the lines of that of Theorem~\ref{MUCCI}. The main differences
are that here we only perform first derivative estimates, whence the choice
of the test function must be modified accordingly, and 
that the gradient estimate initially presents a small error that needs to be reabsorbed via scaled estimates.

The sketch of the proof of Theorem~\ref{PUCCI TYPE BIS}, emphasizing the technical differences with respect to the proof of Theorem~\ref{MUCCI}, goes as follows.

We assume the equation to hold in the ball~$B_2$ instead of~$B_1$, which we can do up to scaling.
The notation~$A_x\in{\mathcal{A}}$ in the proof of Theorem~\ref{MUCCI}
must be replaced by~$B_x\in{\mathcal{B}}$, as well as
the notation~${\mathcal{L}}_A$
must be changed here into~${\mathcal{L}}_{K_B}$. 
Also, the continuity of~$g_A$ with respect to~$A$ given by~\eqref{Gcon}
is replaced here by the continuity of~$g_B$ with respect to~$B$, as given by~\eqref{Gcon-BIS}.
In addition, now~$u$ is a solution of equation~\eqref{1.23BIS} in $B_2$ instead of equation~\eqref{MUCCI:EQ}.

In this way, we find that, given~$x\in B_2$, there exists~$B_x\in{\mathcal{B}}$ such that
$$ 0=\sup_{B\in{\mathcal{B}}} \Big( {\mathcal{L}}_{K_B}u(x)-g_{B}(x)\Big)=
{\mathcal{L}}_{K_{B_x}}u(x)-g_{B_x}(x).$$

Next, prooceding exactly as in the proof of Theorem~\ref{MUCCI}, we have that
the maximum principle
with estimate for the family of operators~$L^{(x)}={\mathcal{L}}_{K_{B_x}}$, as considered in Definition~\ref{BENE},
is satisfied in this case.

We now proceed as in the proof of Theorem~\ref{sec:XT} (namely, Theorem~\ref{sec:XT}
in its generality leads to second derivative estimates, while here we focus on the same setting but restricted to
first derivatives).
To check the hypotheses of Theorem~\ref{sec:XT}
for the family of operators~$L^{(x)}={\mathcal{L}}_{K_{B_x}}$, one needs to establish the validity of
a suitable analogue of~\eqref{L-EQAU-B1}, but here with~$\overline\eta:=0$ and~$\tau:=1$, since
we deal only with first derivatives.
To this end,
instead of using Proposition~\ref{PRO:AZZ} as in the proof of Theorem~\ref{MUCCI},
we utilize here Theorem~\ref{CONTO TRACCIA}. But then, in doing so,
the right-hand side of the analogue inequality to~\eqref{L-EQAU-B1}
(which holds now in $B_2$ by Theorem~\ref{CONTO TRACCIA}) presents an additional reminder
$\e^2\,\|\partial_e u\|^2_{L^\infty(B_3)}$.
Consequently,
the first derivative estimate in~\eqref{NBCshnr:1} also
presents an additional reminder $\e\,\|\partial_e u\|_{L^\infty(B_3)}$.

More precisely, the estimate corresponding to~\eqref{NBCshnr:1} here reads
\begin{equation}\label{THERES}\begin{split}&
\sup_{B_{1}} |\partial_e u|
\le C_\e\,\Big( \|\partial_e g_{B} \|_{L^\infty({\mathcal{B}}\times B_2)}
+\big(
\| ( g_{B} )_-\|_{L^\infty({\mathcal{B}}\times B_2)}
\,\| u\|_{L^\infty(B_2)}
\big)^{1/2}+\| u\|_{L^\infty(\R^n)}
\Big)\\&\qquad\qquad\qquad\qquad+\e\,\|\partial_e u\|_{L^\infty(B_3)}.\end{split}\end{equation}
The form of the right-hand side of this inequality originates from estimating the quantities~$a_0$ and~$a_1$ in the statement of Theorem~\ref{sec:XT}. For this,
in the present framework the analogue of~\eqref{3.37BIS} allows us to bound~$a_0$ by~$\| ( g_{B} )_-\|_{L^\infty({\mathcal{B}}\times B_2)}$ and the analogue of~\eqref{X560199339999:2} provides an estimate on~$a_1$ of the form~$\|\partial_e g_{B} \|_{L^\infty({\mathcal{B}}\times B_2)}$, leading to the right-hand side of~\eqref{THERES}.

Now we assume the function $u$ to be a solution in $B_5$ and we scale~\eqref{THERES} in a ball~$B_\rho(x)$, for each~$x\in B_2$ and~$\rho\in(0,1)$.
This can be performed by applying the previous estimate to the function~$ u_\rho(y):=u(x+\rho y)$. Using that $\rho\in(0,1)$, one finds that
\begin{equation*}\begin{split}
\sup_{B_{\rho}(x)} \rho |\partial_e u|&
\le C_\e\,\Big( \rho^{1+2s}\|\partial_e g_{B} \|_{L^\infty({\mathcal{B}}\times B_4)}
+\big(\rho^{2s}
\| ( g_{B} )_-\|_{L^\infty({\mathcal{B}}\times B_4)}
\,\| u\|_{L^\infty(B_4)}
\big)^{1/2}+\| u\|_{L^\infty(\R^n)}
\Big)\\&\qquad\qquad\qquad\qquad+\e\rho\,\|\partial_e u\|_{L^\infty(B_{3\rho}(x))}\\
&\le C_\e\sigma+\e\rho\,\|\partial_e u\|_{L^\infty(B_{3\rho}(x))},\end{split}\end{equation*}
where
$$ \sigma:=\|\partial_e g_{B} \|_{L^\infty({\mathcal{B}}\times B_4)}
+\big(
\| ( g_{B} )_-\|_{L^\infty({\mathcal{B}}\times B_4)}
\,\| u\|_{L^\infty(B_4)}
\big)^{1/2}+\| u\|_{L^\infty(\R^n)}.$$

Finally, we combine this
scaled estimates with Lemma~\ref{SCALED} (used with~$m:=1$) 
to reabsorb the remainder in the left-hand side. In this way we obtain a gradient estimate in $B_{1/2}$ for solutions in $B_5$. Now, a standard covering and scaling argument yields the estimate as stated in Theorem~\ref{PUCCI TYPE BIS}.
\end{proof}

\begin{appendix}

\section{A maximum principle in~$\R^{n+1}_+$}\label{AP-maxple}

We give here an elementary proof of a maximum principle for the extension operator
in the whole halfspace.
The result may have appeared somewhere else, but we could not find a reference.

\begin{lemma}\label{0okmokm6}
Let~$V$ be bounded from above in~$\R^{n+1}_+$ and
satisfy~$L_a V\le0$ weakly\footnote{We could instead assume that~$V$ is locally~$W^{2,\infty}$
in~$\R^{n+1}_+$ and that~$L_aV\le0$ is satisfied at every point in~$\R^{n+1}_+$ in the nondivergence
sense of Remark~\ref{R35} and Lemma~\ref{LE:AUS:1:ZZ} (see also the comments after the lemma). The proof of
Lemma~\ref{0okmokm6} in this setting is the same as the one that we give below. Simply notice that the
maximum principle in the half ball ${\mathcal{B}}_{4R}^+$ also holds for this notion of
subsolutions.}
in~$\R^{n+1}_+$. Assume that~$V(x,0)\le0$
for~$x\in\R^n$ (in the trace sense).

Then, $V\le0$ in~$\R^{n+1}_+$.
\end{lemma}

\begin{proof} Replacing~$V$ by~$V/\big(1+\|V^+\|_{L^\infty(\R^{n+1}_+)}\big)$, we can suppose that~$
V\le1$ in~$\R^{n+1}_+$.
For~$R>0$, we define~$
{\mathcal{B}}_R:=\{(x,y)\in\R^{n+1}
{\mbox{ s.t. }}|(x,y)|<R\}$ and~$
{\mathcal{B}}_R^+:={\mathcal{B}}_R\cap\R^{n+1}_+$.

Let~$\overline{W}\in C^\infty(\R^{n+1})$ be nonnegative and
such that~$\overline{W}=0$ in~${\mathcal{B}}_{2}$ and~$\overline{W}=1$
in~$\R^{n+1}\setminus {\mathcal{B}}_{3}$.
We define~$W:\R^{n+1}_+\to\R$ to be the minimizer of the extended
Dirichlet energy$$
\int_{{\mathcal{B}}_4^+} y^a|\nabla U(x,y)|^2\,dx\,dy$$
among the functions~$U$ having finite energy and such
that~$U=\overline{W}$ along~$\partial{\mathcal{B}}_4^+$.
We point out that this is a well posed problem in the sense
of Sobolev spaces with Muckenhoupt  weights
(which indeed possess a suitable notion of trace);
see~\cite{MR643158} and Theorem~3.2 of~\cite{MR3165278}.

Now, for every~$R>0$, we let~$W_R(x,y):=W\left(\frac{x}R,\frac{y}R\right)$. We note that~$W_R(x,y)=\overline{W}\left(\frac{x}R,\frac{y}R\right)=1\ge V(x,y)$
for~$(x,y)\in\partial {\mathcal{B}}_{4R}\cap \R^{n+1}_+$.
Furthermore, $W_R(x,0)=W\left(\frac{x}R,0\right)
=\overline{W}\left(\frac{x}R,0\right)\ge0\ge V(x,0)$ for~$x\in B_{4R}$.
As a consequence, by the maximum
principle (see~\cite{MR643158}) we infer
that
\begin{equation}\label{7Co7}
V(x,y)\le W_R(x,y)=W\left(\frac{x}R,\frac{y}R\right)\quad{\mbox{
for }} (x,y)\in{\mathcal{B}}_{4R}^+.\end{equation}

We also define~$W_{\!o}:{\mathcal{B}}_2\to\R$
to be the odd reflection of~$W$ in the variable~$y$. Thus,
we have that~$-{\rm div}(|y|^a\nabla W_{\!o})=0$ weakly
in~${\mathcal{B}}_2$.
Hence, we can apply the results in~\cite{MR643158}
(or Theorem~3.3 of~\cite{MR3165278})
and deduce that~$W_{\!o}$ is 
H\"older continuous, and thus so is~$W$ up to the boundary~$B_{2R}\times\{0\}$.
{F}rom this and~\eqref{7Co7}
the desired result
follows by sending~$R\to+\infty$.
\end{proof}

\section{Scaled inequalities}\label{APPA}

We recall a classical result about scaled inequalities which we have
used to prove Theorem~\ref{PUCCI TYPE BIS}.

\begin{lemma}\label{SCALED}
Let $m\ge0$ be an integer and~$u\in C^m(B_5)$, with~$B_5\subset\R^n$. 
Let~$\sigma_0\ge0$, and assume that for every~$\e>0$ there exists a constant~$C_\e\ge0$
such that
\begin{equation}\label{0380392410}
\sum_{k=0}^m \rho^k \| D^k u\|_{L^\infty(B_\rho(x))}
\le C_\e \,\sigma_0+\e\sum_{k=0}^m
\rho^k\| D^k u\|_{L^\infty(B_{3\rho}(x))}
\end{equation}
for every~$x\in B_2$ and every~$\rho\in(0,1)$.

Then, there exists a constant~$C$, depending only\footnote{It is important to
notice that, in the statement of Lemma~\ref{SCALED},
the constant~$C$ does not depend only
on~$m$ ---but on $m$ and on the value of the
function~$C_\varepsilon=:C(\varepsilon)$
at~$\epsilon:=\big(2(m+1)4^{m+1}\big)^{-1}$.
When we apply Lemma~\ref{SCALED}
to prove Theorem~\ref{PUCCI TYPE BIS}, this observation is crucial
in order to obtain the correct dependencies of the structural constants
in Theorem~\ref{PUCCI TYPE BIS} (instead of obtaining a universal constant).
}on~$m$
and on the value of the constant~$C_\epsilon$ 
with~$\epsilon:=\big(2(m+1)4^{m+1}
\big)^{-1}$, such that
\begin{equation}\label{9eudf3331144}
\sum_{k=0}^m \| D^k u\|_{L^\infty(B_{1/2})}\le C \sigma_0.
\end{equation}
\end{lemma}

\begin{proof}
Let
\begin{eqnarray*}&&
S(x):=\sum_{k=0}^m (1-|x|)^{k+1}\,|D^ku(x)|\qquad\quad
{\mbox{ and }}\qquad \quad M:=\max_{x\in\overline{B}_1} S(x).
\end{eqnarray*}
Since~$S$ vanishes along~$\partial B_1$, the maximum of~$S$
in~$\overline{B}_1$ is attained at an interior point: namely, there exists~$x_0$
with~$|x_0|<1$ such that~$S(x_0)=M$.

We set
$$ \rho_0:=\frac{1-|x_0|}{4}\in(0,1).$$
Then, we have that~$B_{3\rho_0}(x_0)\subset B_1$,
and, for all~$x\in B_{3\rho_0}(x_0)$, it holds that~$1-|x|\ge1-|x_0|-3\rho_0=\rho_0$.
As a consequence,
\begin{eqnarray*}&& \sum_{k=0}^m \rho_0^k\| D^k u\|_{L^\infty(B_{3\rho_0}(x_0))}\leq
\sum_{k=0}^m \rho_0^k \sup_{x\in B_{3\rho_0}(x_0)}
\frac{S(x)}{(1-|x|)^{k+1}}\\&&\qquad\qquad
\le \rho_0^{-1}\sum_{k=0}^m \sup_{x\in B_{3\rho_0}(x_0)} S(x)
\le  \rho_0^{-1} (m+1)M.\end{eqnarray*}
{F}rom this and~\eqref{0380392410}, we infer that
\begin{eqnarray*}
&&
C_\e \,\sigma_0+\e (m+1)M\ge\rho_0
C_\e \,\sigma_0+\rho_0\e\sum_{k=0}^m
\rho_0^k\| D^k u\|_{L^\infty(B_{3\rho_0}(x_0))}
\ge\rho_0
\sum_{k=0}^m \rho_0^k \| D^k u\|_{L^\infty(B_{\rho_0}(x_0))}
\\ &&\quad\ge
\sum_{k=0}^m 
\frac{(1-|x_0|)^{k+1}}{4^{k+1}}\, | D^k u(x_0)|
\ge
\sum_{k=0}^m 
\frac{(1-|x_0|)^{k+1}}{4^{m+1}}\, | D^k u(x_0)|
=\frac{S(x_0)}{4^{m+1}}=\frac{M}{4^{m+1}} .\end{eqnarray*}

Therefore, taking~$\e:=\frac{1}{2(m+1)\,4^{m+1}}$, we 
conclude that~$S\le M\le C\sigma_0$ in~$B_1$ for some constant~$C$ depending only on~$m$
and on the constant~$C_\epsilon$ 
with~$\epsilon:=\big(2(m+1)4^{m+1}
\big)^{-1}$.
{F}rom this, the bound in~\eqref{9eudf3331144}, in half the ball, follows immediately.
\end{proof}

\section{A variant of Proposition~\ref{SUPERT-Intro}}\label{RGiwq04}

We state and prove here a convenient modification of Proposition~\ref{SUPERT-Intro} (in which~$\partial_eu$
is replaced by~$\nabla u$)
that is used in the proof of Lemma~\ref{849027674843hf38}.
\begin{proposition}\label{NUOVA:P}
Let~$K$ satisfy~\eqref{EVEN}
and~\eqref{KLIM}, and let ${\mathcal{L}}_K$ be defined by~\eqref{OP:L:BIS}. Given a function $u\in C^\infty(\R^n)\cap W^{1,\infty}(\R^n)$,
$\eta\in C^\infty(\R^n)\cap L^{\infty}(\R^n)$, and~$\CALERRE\in\R$, consider
\begin{equation*} \varphi:=\eta^2 |\nabla u|^2 +\sigma u^2.\end{equation*}

Then, the inequality
\begin{equation*}
{\mathcal{L}}_K \varphi\le 2\eta^2\,\nabla u\cdot{\mathcal{L}}_K \nabla u+2\sigma u\,{\mathcal{L}}_K u+\CALERRE
\end{equation*}
holds at a point~$x\in\R^n$ if and only if
\begin{equation*}
\begin{split}&
2\int_{\R^n}\eta(x)\,\big( \eta(x)-\eta(y)\big)\,\nabla u(x)\cdot\nabla u(y)\,K(x-y)\,dy
\\ &\qquad\le\,\int_{\R^n}\big|\eta(x)\,\nabla u(x)-\eta(y)\,\nabla u(y)\big|^2\,K(x-y)\,dy
\\&\qquad\qquad+\sigma\int_{\R^n}\big|u(x)- u(y)\big|^2\,K(x-y)\,dy
+\CALERRE.\end{split}\end{equation*}
\end{proposition}

\begin{proof} We observe that
\begin{eqnarray*}
&& {\mathcal{L}}_K \Big( \eta^2\,|\nabla u|^2+\sigma
\,u^2\Big)(x)
-2\eta^2(x)\, \nabla u(x)\cdot{\mathcal{L}}_K\,\nabla u (x)
\\&&\qquad\qquad-2\sigma\,u(x)\,{\mathcal{L}}_K u(x)+\sigma\int_{\R^n}\big|u(x)-u(y)\big|^2\,K(x-y)\,dy\\
&&\qquad= \int_{\R^n} 
\Big( \eta^2(x)\,|\nabla u(x)|^2
-\eta^2(y)\,|\nabla u(y)|^2\Big)\,K(x-y)\,dy\\&&\qquad\qquad
+\sigma\int_{\R^n} \big(u^2(x)-u^2(y)\big)\,K(x-y)\,dy\\
&&\qquad\qquad-2\eta^2(x)\, \nabla u(x)\cdot
\int_{\R^n}\Big( \nabla u(x)-\nabla u(y)\Big)
\,K(x-y)\,dy\\
&&\qquad\qquad-2\sigma\,u(x)\,\int_{\R^n} \big(u(x)-u(y)\big)\,K(x-y)\,dy
+\sigma\int_{\R^n}\big|u(x)-u(y)\big|^2\,K(x-y)\,dy
\\
&&\qquad= \int_{\R^n}\Big( 2\eta^2(x)\, \nabla u(x)\cdot\nabla u(y)
-\eta^2(x)\, |\nabla u(x)|^2-
\eta^2(y)\, |\nabla u(y)|^2
\Big)\,K(x-y)\,dy
\\ &&\qquad=2\int_{\R^n} \eta(x)\,\big(\eta(x)-\eta(y)\big)\, \nabla u(x)
\cdot\nabla u(y)
\,K(x-y)\,dy
\\&&\qquad\qquad-
\int_{\R^n}\Big|
\eta(x)\, \nabla u(x)-
\eta(y)\, \nabla u(y)
\Big|^2\,K(x-y)\,dy.
\end{eqnarray*}
{F}rom this, the proposition follows readily.
\end{proof}

\section{Existence and regularity of solutions}\label{sec:regularity-new}

In order to keep our arguments as simple as possible, in this paper
the main computations have been performed
assuming that the solution~$u$ is smooth. Note that our auxiliary functions depend on the first or second derivatives of the solution $u$, and that we need the operator to act on them. Thus, in this article we need the solution to be at least $C^{1+2s}$ or $C^{2+2s}$,
respectively. Here we discuss known existence and regularity results for the equations that we cover.

Concerning the existence results
for concave nonlocal fully nonlinear equations (possibly including
also the case of rough kernels), J. Serra proved in
Theorem~1.3
of~\cite{MR3426087}
that if~${\mathcal{B}}$ is a family of indexes
and for all~$B\in{\mathcal{B}}$ the kernel~$K_B$
satisfies the evenness and ellipticity conditions~\eqref{EVEN}
and~\eqref{KLIM},
and~$g_B$
and~$u_0$ are H\"older continuous, with~$u_0$ bounded in~$\R^n\setminus B_1$, then
the problem
$$ \begin{cases}\displaystyle
\sup_{B\in{\mathcal{B}}}\Big(
{\mathcal{L}}_{K_B}u-g_B\Big)=0 & {\mbox{ in }}B_1,\\
u=u_0 & {\mbox{ in }}\R^n\setminus B_1,
\end{cases}$$
admits a unique viscosity solution which is continuous everywhere
and~$C^{2s+\beta}_{\rm loc}(B_1)$ for some~$\beta\in(0,1)$.
This setting applies to equations~\eqref{MUCCI:EQ} and~\eqref{1.23BIS}
presented in Theorems~\ref{MUCCI}
and~\ref{PUCCI TYPE BIS} of this paper.

See also~\cite{MR1081182, MR1081183, MR2400256, MR2735074, MR2911421}
for other existence and regularity results in related (but quite different) fractional settings.

A general approach to regularity is provided by the notion of
elliptic operators in nonlinear integro-differential equations arising from L\'evy processes,
as given by Caffarelli and Silvestre in Definition~3.1 of~\cite{caff-silv}. When our source terms~$g_A$, $g_B$, $g_1,\dots,g_J$ are constants,\footnote{The constancy of the source terms is needed to make the operator translation invariant,
as requested on page~603 of~\cite{caff-silv}.}
this setting
includes 
our Pucci-type equations~\eqref{MUCCI:EQ} and~\eqref{1.23BIS},
as well as
the fully nonlinear framework presented here in~\eqref{EQ} (when all the operators have
the same order)
when~$F:\R^J\to\R$ is differentiable
and~$\partial_{p_j} F\in [C^{-1},C]$ for some~$C\ge1$
and all~$j\in\{1,\dots,J\}$.
Under the assumptions
that the source terms~$g_A$, $g_B$, $g_1,\dots,g_J$ are constants,
that the operators have all the same order and satisfy~\eqref{EVEN},
\eqref{KLIM}, and a bound for~$|\nabla K|$ as in~\eqref{KC1}, in light of Theorem~13.1 in~\cite{caff-silv}
we know that the solutions in Theorems~\ref{MUCCI},
and~\ref{PUCCI TYPE BIS} are locally~$C^{1+\beta}$
for some~$\beta\in(0,1)$.
This also applies to the case dealt with
in Theorem~\ref{THM:1} when all
the nonlocal operators~${\mathcal{L}}_{\mu_j}$
reduce to the fractional Laplacian of some order~$s$ (with the same order~$s$
for all the operators). See also
Theorem~27 of~\cite{MR2781586} for related results
with uniform estimates.

Similarly, in the
case of vanishing source terms~$f$, $g_A$, and~$g_B$,
Theorem~1.1 of~\cite{caff-silv-2} leads to the local~$C^{2s+\beta}$ regularity for equations~\eqref{MUCCI:EQ}
and~\eqref{1.23BIS} of definite order $2s$ ---for
the latter, assuming that all the kernels satisfy
assumptions \eqref{EVEN},
\eqref{KLIM}, and~\eqref{KC1}.
In particular, since~$2s+\beta>2s$,
the corresponding equation is satisfied also in the pointwise sense
(see e.g. 
the comment before Theorem~1.1 in~\cite{caff-silv-2},
and Proposition~2.1.4 of~\cite{MR2707618}
for a detailed proof of this fact).

In any case, the regularity of type~$C^{1+\beta}$
or~$C^{2s+\beta}$ is not sufficient for the techniques discussed
in this paper, and this is the reason for which we are taking additional
regularity assumptions in our main results.

General results dealing with higher regularity for nonlocal
fully nonlinear equations in 
boun-\-ded 
domains need to address 
two difficulties, namely the nonlinearity of the equation and the possible
singularity that external data may induce due to the nonlocal structure of the problem.
For linear equations satisfied in the whole of~$\R^n$
the regularity theory
is well understood, and in this case solutions are $C^\infty$.
Indeed, in this situation
one can conclude via a bootstrap argument that
solutions are smooth even in case of integro-differential operators with
rough kernels,
since one can differentiate the equation without introducing errors that come from rough exterior data
(e.g., one could proceed as in Section~4 of~\cite{MR3744813}
without having to introduce additional cutoff functions).

To understand the effect that being a global solution has
on regularity, one can consider the toy model given by
\begin{equation}\label{TOY1}
F\big( (-\Delta)^s u(x), u(x)\big)=0 \qquad{\mbox{ for all }}x\in\R^n,
\end{equation}
with~$F$ smooth and with first derivatives bounded and bounded away from zero. Assume first
that~$s\in(1/2,1)$.
In this case, once~$u$ belongs to~$C^{2s+\beta}(B_1(x_0))$ for some~$\beta\in(0,1)$ ---a regularity which is known from~\cite{MR3426087}---
and
for all~$x_0\in\R^n$, given a direction~$e$,
one can differentiate~\eqref{TOY1} and obtain
$$ \partial_{p_1} F\big( (-\Delta)^s u(x), u(x)\big)\,(-\Delta)^s u_e(x)
+\partial_{p_2} F\big( (-\Delta)^s u(x), u(x)\big)\, u_e(x)=0\qquad{\mbox{ for all }}x\in\R^n.$$
Therefore
$$ \,(-\Delta)^s u_e(x)=f(x)\qquad{\mbox{ for all }}x\in\R^n,$$
with
$$ G(p_1,p_2):= -
\frac{ \partial_{p_2} F(p_1,p_2) }{
\partial_{p_1} F(p_1,p_2) }$$
and
$$ f(x):=G\big( (-\Delta)^s u(x), u(x)\big)\, u_e(x).$$
Thus, one can apply the local regularity theory for the fractional Laplacian in~$B_{1/2}(x_0)$
(see e.g. Proposition 2.1.11 in~\cite{MR2707618}) and obtain that
\begin{equation}\label{yhgbgfhfgjuF} \| u_e \|_{C^{1+\alpha}(B_{1/2}(x_0))}\le
C\,\Big( \|u_e\|_{L^\infty(\R^n)}+\|f\|_{L^\infty(B_1(x_0))}\Big)\le
C\,\|u_e\|_{L^\infty(\R^n)},\end{equation}
for some~$\alpha\in(0,1)$,
up to renaming~$C$. 
But, since~$s>1/2$,
for every~$x_0\in\R^n$,
$$ |u_e(x_0)|\le \| u\|_{C^1(B_1(x_0))}\le \| u\|_{C^{2s+\beta}(B_1(x_0))}.$$
Hence, we have a uniform bound on~$\|u\|_{C^{2+\alpha}(B_{1/2}(x_0))}$,
for all~$x_0\in\R^n$, and in particular a bound on~$\|D^2 u\|_{L^\infty(\R^n)}$.
Since this is a global bound, one can iterate the procedure and obtain that~$u\in C^\infty(\R^n)$.
When~$s\le1/2$, this method needs
to be modified, by iterating a H\"older regularity result
on the incremental quotients; see e.g.
pages~634-635 in~\cite{caff-silv}.
In both cases, we stress that this procedure only works for global solutions,
since~\eqref{yhgbgfhfgjuF}
requires~$x_0$ to be an arbitrary point in~$\R^n$.

The bootstrap method 
can also be applied to fully nonlinear operators
when the equation is satisfied in a bounded domain with good exterior data;
see Theorem 1.5 in~\cite{MR3579887}
and Theorem~6 in~\cite{MR3331523} to implement the bootstrap
regularity of Schauder type. 

For equations satisfied on bounded domains,
when the operator is built, roughly speaking, by the sum of fractional Laplacians of different orders
that include the classical Laplacian, then the solutions
are typically~$C^\infty$, in view of the regularizing effect of the
higher order operator. In this setting
one can include also some nonlinear terms. A simple example consists of~$u\in C^{2s+\beta}(B_1)$
being a solution of the equation
$$ -\Delta u(x)+\widetilde F\big( (-\Delta)^s u(x)\big)=0\qquad{\mbox{for all }}x\in B_1,$$
with~$\widetilde F\in C^\infty(\R)$.
We then have that the map~$ x\mapsto f(x):=-\widetilde F\big( (-\Delta)^s u(x)\big)$
is locally H\"older continuous
and hence we can apply
the classical Schauder theory to 
obtain that second derivatives of~$u$ are
H\"older continuous in~$B_{1/2}$.
Then, by
bootstrapping, we
conclude that~$u$ has as many derivatives as we wish.
We remark that this setting is a particular case of that in~\eqref{FORM83}
by choosing~$F(p_1,p_2):=p_1+\widetilde F(p_2)$ and~$s_1:=1$.

There are other special situations in which $C^\infty$
solutions in a bounded domain can be constructed, as established in
Theorem~1.1 of~\cite{MR3744813}.
This paper provides cases of nonlocal fully nonlinear
equations (rather concrete ones) whose Dirichlet problems
possess a
unique and
smooth solution. More specifically,
as detailed in Definitions~2.3 and~2.10 in~\cite{MR3744813},
one can consider ``nice weights''
which make the bootstrap regularity compatible with the convex
structure of the equation. This setting provides
smooth solutions for
concave elliptic operators acting
on integral expressions of the form
$$ \int_{\R^n}\big(u(x+y)+u(x-y)-2u(x)\big)\,\frac{y_i\,y_j\,\rho(x,y)}{|y|^{n+2s+2}}\,dy,$$
where~$\rho$ is smooth, bounded, and
bounded away from zero, with derivatives of order~$j$
in the variable~$y$ bounded by~$C_j|y|^{-j}$ for some constant~$C_j$.
In particular, Theorem~1.1 in~\cite{MR3744813}
gives the existence of a $C^\infty$ solution of
$$ \begin{cases}\displaystyle
\int_{\R^n}\big(u(x + y) + u(x - y) - 2u(x)\big)\frac{
\rho(x,y)
}{|y|^{n+2s}}\,dy=f(x) & {\mbox{ for all }}x\in B_1,
\\u=g&{\mbox{ in }}\R^n\setminus B_1,
\end{cases}$$
provided that~$f$ is $C^\infty$,
$g$ is bounded and uniformly continuous, and~$\rho$ is a nice weight
as above.

In general, however, the nonlocal setting is not expected
to always provide $C^\infty$ solutions to general
fully nonlinear equations in bounded domains. This is due to the fact that
the linearized equation exhibits coefficients which
depend on the global data and are in general not better
than H\"older continuous (no matter how regular the solution
is in the interior of the domain). 
As a consequence, the Schauder theory cannot be applied to
bootstrap regularity. As a matter of fact, the higher regularity
theory for fully nonlinear elliptic equations of nonlocal type in bounded domains
is, at the moment,
a field of research still under investigation.
It would be desirable to understand natural assumptions guaranteeing bootstrap regularity of~$ C^\infty$ type.

\end{appendix}

\bigskip\bigskip

\end{document}